\documentclass[11pt]{amsart}
\makeatletter
\@namedef{subjclassname@2020}{\textup{2020} Mathematics Subject Classification}
\makeatother

\usepackage[centering]{geometry}

    \usepackage[T1]{fontenc}

    \usepackage{amsmath}
    \usepackage{amssymb}
    \usepackage{bm}
    \usepackage{mathrsfs}
    \usepackage{dsfont}
    \usepackage{accents}
    \usepackage{mathtools}
    \allowdisplaybreaks
    \DeclareTextFontCommand{\emph}{\bfseries\em}

    \usepackage{amsfonts}
    \usepackage[bbgreekl]{mathbbol}
      \DeclareSymbolFontAlphabet{\mathbb}{AMSb}
      \DeclareSymbolFontAlphabet{\mathbbl}{bbold}
      \DeclareMathSymbol{\bbepsilon}{\mathord}{bbold}{"0F}

    \usepackage{xcolor}
    \definecolor{refkey}{rgb}{128,0,128}
    \definecolor{labelkey}{rgb}{128,0,128}
    \usepackage{comment}
    \usepackage[colorlinks=true, linktocpage=true, linkcolor=olive!50!red, citecolor=olive!50!green, backref=page]{hyperref}

    \usepackage{float}
    \usepackage{graphicx,xcolor}
    \usepackage{ascmac}
    \usepackage{harpoon}

    \usepackage{tabularx}

    \usepackage{tikz}
    \usepackage{pgfplots}
    \pgfplotsset{compat=1.17}
    \usetikzlibrary{calc, decorations.markings}
    \usetikzlibrary{arrows.meta}
    \usetikzlibrary{positioning}
    \usetikzlibrary{cd, intersections, decorations.pathmorphing, arrows, decorations.pathreplacing}
    \usetikzlibrary{pgfplots.fillbetween}

\usepackage{cleveref}
\crefname{thm}{Theorem}{Theorems}
\crefname{cor}{Corollary}{Corollaries}
\crefname{lem}{Lemma}{Lemmas}
\crefname{prop}{Proposition}{Propositions}
\crefname{dfn}{Definition}{Definitions}
\crefname{eg}{Example}{Examples}
\crefname{claim}{Claim}{Claims}
\crefname{conj}{Conjecture}{Conjectures}
\crefname{conv}{Notation}{Notations}
\crefname{rmk}{Remark}{Remarks}
\crefname{prob}{Problem}{Problems}
\crefname{figure}{Figure}{Figures}
\crefname{table}{Table}{Tables}
\crefname{section}{Section}{Sections}
\crefname{subsection}{Section}{Sections}
\crefname{appendix}{Appendix}{Appendices}
\crefname{lemdef}{Lemma-Definition}{Lemma-Definitions}
\crefname{conv}{Convention}{Conventions}
\crefname{introthm}{Theorem}{Theorems}
\crefname{introcor}{Corollary}{Corollaries}
\crefname{introconj}{Conjecture}{Conjectures}

\usepackage{amsthm}

\newtheorem{thm}{Theorem}
\newtheorem{lem}[thm]{Lemma}
\newtheorem{prop}[thm]{Proposition}

\newtheorem{introthm}{Theorem}

\theoremstyle{definition}
\newtheorem{dfn}[thm]{Definition}
\newtheorem{rmk}[thm]{Remark}
\newtheorem{eg}[thm]{Example}

\numberwithin{figure}{section}
\numberwithin{equation}{section}
\numberwithin{thm}{section}

    \newcommand{\cR}{\mathcal{R}}
    
    \newcommand{\bN}{\mathbb{N}}
    \newcommand{\bZ}{\mathbb{Z}}
    
    \newcommand{\bs}{{\boldsymbol{s}}}

    \newcommand{\G}{\mathcal{G}}
    \newcommand{\W}{\mathcal{W}}
    \newcommand{\TL}[2]{{\mathcal{T\!L}}(#1,#2)}
    \newcommand{\Graph}[2]{{\G}(#1;#2)}
    \newcommand{\ID}{\mathbbl{1}}
    
    \newcommand{\mdeg}[1]{\operatorname{d_{\ast}}\!\left(#1\right)}
    \newcommand{\BNet}[5]{{\left\uparrow\!\begin{smallmatrix}#5& #4\\ #2& #3\end{smallmatrix}\!\right\downarrow}_{#1}}
    \newcommand{\saddle}{\mathbf{S}}

\title{The zero stability for the one-row colored $\mathfrak{sl}_3$-Jones polynomial}

\author[W.~Yuasa]{Wataru Yuasa}
\date{}
\address{Graduate School of Science\\ Division of Mathematics and Mathematical Sciences\\ Kyoto University\\ Kitashirakawa Oiwake-cho, Sakyo-ku, Kyoto 606-8502, Japan}
\email{yuasa.wataru.6m@kyoto-u.ac.jp}
\urladdr{https://wataruyuasa.github.io/math/}


\begin{document}

\begin{abstract}
    The stability of coefficients of colored ($\mathfrak{sl}_2$-) Jones polynomials $\{J_{K,n}^{\mathfrak{sl}_2}(q)\}_n$ was discovered by Dasbach and Lin. This stability is now called the zero-stability of $J_{K,n}^{\mathfrak{sl}_2}(q)$. Armond showed zero stability for a $B$-adequate link by using the linear skein theory based on the Kauffman bracket. In this paper, we prove the zero stability of one-row colored $\mathfrak{sl}_{3}$-Jones polynomials $\{J_{K,n}^{\mathfrak{sl}_3}(q)\}_n$ for $B$-adequate links $L$ with anti-parallel twist regions by using the linear skein theory based on Kuperberg's $\mathfrak{sl}_3$-webs. It implies the existence of many $q$-series obtained from a quantum invariant associated with $\mathfrak{sl}_3$.
\end{abstract}
\maketitle
\tableofcontents



\tikzset{->-/.style={decoration={
  markings,
  mark=at position #1 with {\arrow[thick, black]{>}}},postaction={decorate}}}
\tikzset{-<-/.style={decoration={
  markings,
  mark=at position #1 with {\arrow[thick, black]{<}}},postaction={decorate}}}
\tikzset{-|-/.style={decoration={
  markings,
  mark=at position #1 with {\arrow[black, semithick]{|}}},postaction={decorate}}}
\tikzset{
    overarc/.style={
      white, double=black, double distance=0.4pt, line width=2.4pt
    }
}

\tikzset{
    triple/.style args={[#1] in [#2] in [#3]}{
        #1,preaction={preaction={draw,#3},draw,#2}
    }
}


\section{Introduction}\label{sect:intro}
the colored $\mathfrak{g}$-Jones polynomial of a knot $K$ is a quantum invariant obtained from an irreducible representation of a simple Lie algebra $\mathfrak{g}$.
In this paper, we will discuss some stability of coefficients of the one-row colored $\mathfrak{sl}_3$-Jones polynomials $\{J_{K,n}^{\mathfrak{sl}_3}(q)\in\bZ[q^{\pm\frac{1}{2}}]\mid n\in\bN\}$ which is a quantum invariant of $K$ associated with irreducible representations of $\mathfrak{sl}_3$ corresponding to the one-row Yang diagram $(n)$.
This kind of stability for the colored ($\mathfrak{sl}_2$-) Jones polynomials was discovered by Dasbach and Lin~\cite{DasbachLin06, DasbachLin07}.
They showed that some leading coefficients, concerning the degree of $q$, of $\{J_{{K},{n}}^{\mathfrak{sl}_2}(q)\}_{n}$ are independent of the colorings $n$ (where $n+1$ is the dimension of an irreducible representation) if a knot $K$ is alternating.
They also conjectured that the first $n$ coefficients of $J_{{K},{N}}^{\mathfrak{sl}_2}(q)$ are constant for all $N$ greater than $n$ if $K$ is alternating.
Armond~\cite{Armond13} proved this conjecture for a larger class of links called adequate links, which contain alternating links.
Independently, Garoufalidis and L\^{e}~\cite{GaroufalidisLe15} proved more general stability, called \emph{$k$-stability}, for alternating links where $k$ is a non-negative integer.
in the sense of Garoufalidis and L\^e, the stability proved by \cite{Armond13} corresponds to the zero stability. 
The $k$-stability also ensures the existence of the $q$-series called the $k$-limit, which is closely related to quantum modular forms. The $0$-limit is also known as the \emph{tail} of $K$.
\begin{dfn}
    Let $\hat{J}_{{K},{n}}^{\mathfrak{sl}_2}(q)\coloneqq \pm q^{d}J_{{K},{n}}^{\mathfrak{sl}_2}(q)=a_0+\sum_{i=1}^{\infty}a_iq^i$ be a normalization of the colored Jones polynomial $J_{{K},{n}}^{\mathfrak{sl}_2}(q)$ of a knot $K$ where the sign is determined so that $a_0$ becomes positive.
    The \emph{tail} of $K$ is a $q$-series $\Phi_{K}(q)\in\mathbb{Z}[[q]]$ satisfying
    \[
        \Phi_{K}(q)-\hat{J}_{{K},{n}}^{\mathfrak{sl}_2}(q) \in q^{n+1}\mathbb{Z}[[q]],
    \]
    for any positive integer $n$.
\end{dfn}
Note that the integrality theorem for the colored Jones polynomial proved in \cite{Le00} claims that the coefficients of $\hat{J}_{{K},{n}}(q)$ become integral, therefore its tail $\Phi_{K}(q)$ belongs to $\mathbb{Z}[[q]]$.
Armond and Dasbach~\cite{ArmondDasbach17} showed that the tail of an adequate knot is determined by its reduced $B$-graph.\footnote{They consider $A$-graphs. However, this corresponds to $B$-graphs in our convention. That is, our $q$ is $q^{-1}$ in \cite{ArmondDasbach17}.} 
A similar result was also obtained by Garoufalidis--Norin~\cite{GaroufalidisNorinVuong16}.
They stated that the first three coefficients of $\Phi_{K}(q)$ of an alternating link $K$ are described in terms of its reduced Tait graph.
From these results, we can see that the tail is not useful in distinguishing links. 
However, tails of knots and links give us interesting $q$-series related to quantum modular forms.
For example, Garoufalidis--L\^e~\cite{GaroufalidisLe15} showed that tails of alternating links are described as a generalization of Nahm sums.
In particular, the tail of $(2,m)$-torus link is the (false) theta series.
In \cite{ArmondDasbach11A, Hajij16, Yuasa18}, the Andrews-Gordon type identity for the (false) theta series was derived from two explicit formulas for the tail of a $(2,m)$-torus link. 
Explicit formulas for tails of other knots and links have been studied in \cite{GaroufalidisLe15, ElhamdadiHajij17, KeilthyOsburn16, BeirneOsburn17}, and for quantum spin networks in~\cite{Hajij16}.

Our goal is to develop a study of the stability and tails for $J_{K,n}^{\mathfrak{sl}_2}(q)$ to quantum invariants $J_{K,\lambda}^{\mathfrak{g}}(q)$ associated with a higher rank simple Lie algebra $\mathfrak{g}$.
Many problems arise when we consider higher rank cases.
For example, we have to choose a sequence of irreducible representations to consider the stability because the colored $\mathfrak{g}$-Jones polynomial of a knot is parametrized by dominant weights.
Moreover, the explicit computation of the colored $\mathfrak{g}$-Jones polynomials of a given knot is much more difficult than in the $\mathfrak{sl}_2$ case.

The aim of this paper is to show zero stability of the one-row colored $\mathfrak{sl}_3$-Jones polynomial $\{J^{\mathfrak{sl}_3}_{{K},{n}}(q)\}_n$ of a $B$-adequate link $K$ with anti-parallel twist regions.
The one-row coloring $n$ for $K$ means that all components of $K$ are colored by the irreducible representation of the highest weight $n\varpi_1$ (or we write it as $(n,0)$) where $\{\varpi_i\}_{i=1,2}$ correspond to the fundamental weights of $\mathfrak{sl}_3$.
There are some studies on the explicit computation of the colored $\mathfrak{sl}_3$ Jones polynomial, for the trefoil knot in \cite{Lawrence03}, for the $(2,2m+1)$- and $(4,5)$-torus knots with general coloring in \cite{GaroufalidisMortonVuong13, GaroufalidisVuong17}, and for $2$-bridge links with one-row coloring in \cite{Yuasa17}, for pretzel links with one-row coloring in \cite{Kawasoe22}.
These explicit formulae give tails of the colored $\mathfrak{sl}_{3}$-Jones polynomial of some links in \cite{GaroufalidisVuong17, Yuasa18, Yuasa20A}.
For the $\lambda$-colored $\mathfrak{g}$-Jones polynomial of the $(a,b)$-torus knot when $\mathfrak{g}$ is a simple Lie algebra of rank $2$, Garoufalidis and Vuong~\cite{GaroufalidisVuong17} proved the $k$-stability for any $k$.
They used the formula of Rosso and Jones in \cite{RossoJones93} to prove it.
In this paper, we prove the zero stability of the one-row colored $\mathfrak{sl}_3$-Jones polynomial for any anti-parallel $B$-adequate link using the linear skein theory for $\mathfrak{sl}_3$ developed by Kuperberg~\cite{Kuperberg96}.
Our proof is inspired by the work of Armond~\cite{Armond13} using the Kauffman bracket.

\begin{introthm}[Zero stability for the one-row colored $\mathfrak{sl}_{3}$-Jones polynomial, see \cref{thm:anti-thm}]
    For any anti-parallel $B$-adequate link $L$, there exists $\Phi^{\mathfrak{sl}_3}_{L}(q)$ in $\mathbb{Z}[[q]]$ such that
    \[
        \Phi^{\mathfrak{sl}_3}_{L}(q)-\hat{J}^{\mathfrak{sl}_3}_{{L},{n}}(q) \in q^{n+1}\mathbb{Z}[[q]].
    \]
    An anti-parallel $B$-adequate link is an oriented link whose representative is a $B$-adequate link diagram with only anti-parallel twist regions; see details in \cref{dfn:B-adequate}.
\end{introthm}

This result and proof are an extension of the work on the zero stability of colored 
$\mathfrak{sl}_2$-Jones polynomials in \cite{Armond13} to $\mathfrak{sl}_3$.
We will discuss the zero stability for general $B$-adequate link in the forthcoming paper.

This paper is organized as follows. 
In \cref{sect:pre}, we introduce the $\mathfrak{sl}_3$ version of the linear skein theory and review properties of $\mathfrak{sl}_3$-webs and $\mathfrak{sl}_3$-clasps.
In \cref{sec:degree}, we discuss the lower bound of the minimum degree of a clasped $\mathfrak{sl}_3$-web.
In \cref{sect:stability}, we prove the zero stability of the one-row colored $\mathfrak{sl}_{3}$-Jones polynomials by calculating clasped $\mathfrak{sl}_3$-webs.
In \cref{Appendix}, we prove some new formulas for the clasped $\mathfrak{sl}_3$-web used in this paper.

\paragraph*{\textbf{Acknowledgment}}
The author gratefully thanks the referees for their helpful comments and suggestions that improve the readability of our proofs.
This work was supported by Grant-in-Aid for Early-Career Scientists Grant Number JP19K14528, JP23K12972.

\section{$\mathfrak{sl}_3$-webs and $\mathfrak{sl}_3$-clasps}\label{sect:pre}
We mainly treat a space of $\mathfrak{sl}_3$-webs which is a linear combination of oriented planar trivalent graphs with coefficients in $\cR=\mathbb{Z}((q^{\frac{1}{6}}))$. 
Let us introduce some useful symbols for elements in $\cR$. We set
\begin{itemize}
	\item a \emph{quantum integer} by $\left[n\right]\coloneqq \frac{q^{\frac{n}{2}}-q^{-\frac{n}{2}}}{q^{\frac{1}{2}}-q^{-\frac{1}{2}}}$ for any non-negative integer $n\in\mathbb{Z}_{\geq 0}$, and
	\item a \emph{quantum binomial coefficient} by ${n \brack k}\coloneqq \frac{\left[n\right]!}{\left[k\right]!\left[n-k\right]!}$ for $0\leq k\leq n$, and ${n \brack k}=0$ for $k>n$ where $[n]!\coloneqq [n][n-1]\cdots[1]$.
\end{itemize}
Let us define $\mathfrak{sl}_3$-web spaces based on \cite{Kuperberg96}. 
We consider a surface $\Sigma$ equipped with signed marked points $(P,\bs)$ where $P\subset\partial \Sigma$ is a finite set and $\bs\colon P\to\{{+},{-}\}$ a map. 

A \emph{tangled trivalent graph} on $\Sigma$ is an immersion of a directed graph $G$ into $\Sigma$ satisfying:
\begin{enumerate}
	\item the valency of a vertex of $G$ is $1$ or $3$,
	\item all crossing points are transversal double points of two edges with under/over-passing information,
	\item the set of univalent vertices of $G$ coincides with $P$,
	\item a neighborhood of a vertex in $\Sigma$ is one of the followings;
	  \begin{itemize}
		  \item[] 
			{\ \tikz[baseline=-.6ex, scale=.1]{
			\draw[dashed, fill=white] (0,0) circle [radius=7];
			\draw[-<-=.5] (0:0) -- (90:7); 
			\draw[-<-=.5] (0:0) -- (210:7); 
			\draw[-<-=.5] (0:0) -- (-30:7);
			}\ }
			,
			{\ \tikz[baseline=-.6ex, scale=.1]{
			\draw[dashed, fill=white] (0,0) circle [radius=7];
			\draw[->-=.5] (0:0) -- (90:7); 
			\draw[->-=.5] (0:0) -- (210:7); 
			\draw[->-=.5] (0:0) -- (-30:7);
			}\ }
			,
			{\ \tikz[baseline=-.6ex, scale=.1]{
			\draw[dashed] (0,0) circle [radius=7];
			\draw[->-=.5] (0,0)--(7,0);
			\node at (0,0) [left]{${-}$};
			\draw[thick] (0,7) -- (0,-7);
			\draw[fill=cyan] (0,0) circle [radius=.5];
			}\ }
			,
			{\ 
			\tikz[baseline=-.6ex, scale=.1]{
			\draw[dashed] (0,0) circle [radius=7];
			\draw[-<-=.5] (0,0)--(7,0);
			\node at (0,0) [left]{${+}$};
			\draw[thick] (0,7) -- (0,-7);
			\draw[fill=cyan] (0,0) circle [radius=.5];
			}\ }.
	  \end{itemize}
\end{enumerate}
A tangled trivalent graph is \emph{flat} if it has no crossings.
An \emph{elliptic face} of a flat trivalent graph $G$ is a $0$-gon ({\em i.e.}, a disk), $2$- or $4$-gon in the set of connected components of $\Sigma\setminus G$ which does not touch the boundary of $\Sigma$.

\begin{dfn}[$\mathfrak{sl}_3$-web spaces~\cite{Kuperberg96}]\label{A2skein}
Let $\G(\bs;\Sigma)$ be the set of the boundary fixing isotopy classes of tangled trivalent graphs on $\Sigma$.
The \emph{$\mathfrak{sl}_3$-web space $\W(\bs;\Sigma)$} is a quotient of the $\mathcal{R}$-module spanned by $\G(\bs;\Sigma)$ modulo the following \emph{$\mathfrak{sl}_3$-skein relations}:
\begin{itemize}
\item
$
\mathord{\ \tikz[baseline=-.6ex, scale=.1]{
\draw[thin, dashed, fill=white] (0,0) circle [radius=7];
\draw[->-=.8] (-45:7) -- (135:7);
\draw[->-=.8, overarc] (-135:7) -- (45:7);
}\ }
=q^{\frac{1}{3}}
\mathord{\ \tikz[baseline=-.6ex, scale=.1]{
\draw[thin, dashed, fill=white] (0,0) circle [radius=7];
\draw[->-=.5] (-45:7) to[out=north west, in=south] (3,0) to[out=north, in=south west] (45:7);
\draw[->-=.5] (-135:7) to[out=north east, in=south] (-3,0) to[out=north, in=south east] (135:7);
}\ }
-q^{-\frac{1}{6}}
\mathord{\ \tikz[baseline=-.6ex, scale=.1]{
\draw[thin, dashed, fill=white] (0,0) circle [radius=7];
\draw[->-=.5] (-45:7) -- (0,-3);
\draw[->-=.5] (-135:7) -- (0,-3);
\draw[-<-=.5] (45:7) -- (0,3);
\draw[-<-=.5] (135:7) -- (0,3);
\draw[-<-=.5] (0,-3) -- (0,3);
}\ }$,\vspace{.5em}
\item
$\mathord{\ \tikz[baseline=-.6ex, scale=.1]{
\draw[thin, dashed, fill=white] (0,0) circle [radius=7];
\draw[->-=.8] (-135:7) -- (45:7);
\draw[->-=.8, overarc] (-45:7) -- (135:7);
}\ }
=q^{-\frac{1}{3}}
\mathord{\ \tikz[baseline=-.6ex, scale=.1]{
\draw[thin, dashed, fill=white] (0,0) circle [radius=7];
\draw[->-=.5] (-45:7) to[out=north west, in=south] (3,0) to[out=north, in=south west] (45:7);
\draw[->-=.5] (-135:7) to[out=north east, in=south] (-3,0) to[out=north, in=south east] (135:7);
}\ }
-q^{\frac{1}{6}}
\mathord{\ \tikz[baseline=-.6ex, scale=.1]{
\draw[thin, dashed, fill=white] (0,0) circle [radius=7];
\draw[->-=.5] (-45:7) -- (0,-3);
\draw[->-=.5] (-135:7) -- (0,-3);
\draw[-<-=.5] (45:7) -- (0,3);
\draw[-<-=.5] (135:7) -- (0,3);
\draw[-<-=.5] (0,-3) -- (0,3);
}\ }$,\vspace{.5em}
\item 
$\mathord{\ \tikz[baseline=-.6ex, scale=.1]{
\draw[thin, dashed, fill=white] (0,0) circle [radius=7];
\draw[-<-=.6] (-45:7) -- (-45:3);
\draw[->-=.6] (-135:7) -- (-135:3);
\draw[->-=.6] (45:7) -- (45:3);
\draw[-<-=.6] (135:7) -- (135:3);
\draw[-<-=.5] (45:3) -- (135:3);
\draw[->-=.5] (-45:3) -- (-135:3);
\draw[-<-=.5] (45:3) -- (-45:3);
\draw[->-=.5] (135:3) -- (-135:3);
}\ }
=
\mathord{\ \tikz[baseline=-.6ex, scale=.1]{
\draw[thin, dashed, fill=white] (0,0) circle [radius=7];
\draw[-<-=.5] (-45:7) to[out=north west, in=south] (3,0) to[out=north, in=south west] (45:7);
\draw[->-=.5] (-135:7) to[out=north east, in=south] (-3,0) to[out=north, in=south east] (135:7);
}\ }
+
\mathord{\ \tikz[baseline=-.6ex, rotate=90, scale=.1]{
\draw[thin, dashed, fill=white] (0,0) circle [radius=7];
\draw[-<-=.5] (-45:7) to[out=north west, in=south] (3,0) to[out=north, in=south west] (45:7);
\draw[->-=.5] (-135:7) to[out=north east, in=south] (-3,0) to[out=north, in=south east] (135:7);
}\ }
$ (the $4$-gon relation), \vspace{.5em}
\item 
$\mathord{\ \tikz[baseline=-.6ex, scale=.1]{
\draw[thin, dashed, fill=white] (0,0) circle [radius=7];
\draw[->-=.5] (0,-7) -- (0,-3);
\draw[->-=.5] (0,3) -- (0,7);
\draw[-<-=.5] (0,-3) to[out=east, in=south] (3,0) to[out=north, in=east] (0,3);
\draw[-<-=.5] (0,-3) to[out=west, in=south] (-3,0) to[out=north, in=west] (0,3);
}\ }
=
\left[2\right]
\mathord{\ \tikz[baseline=-.6ex, scale=.1]{
\draw[thin, dashed, fill=white] (0,0) circle [radius=7];
\draw[->-=.5] (0,-7) -- (0,7);
}\ }
$ (the $2$-gon relation), \vspace{.5em}
\item 
$
\mathord{\ \tikz[baseline=-.6ex, scale=.1]{
\draw[thin, dashed, fill=white] (0,0) circle [radius=7];
\draw[->-=.5] (0,0) circle [radius=3];
}\ }
=
\left[3\right]
\mathord{\ \tikz[baseline=-.6ex, scale=.1]{
\draw[thin, dashed, fill=white] (0,0) circle [radius=7];
}\ }
=
\mathord{\ \tikz[baseline=-.6ex, scale=.1]{
\draw[thin, dashed, fill=white] (0,0) circle [radius=7];
\draw[-<-=.5] (0,0) circle [radius=3];
}\ }
.
$ (the trivial loop relation)\vspace{.5em}
\end{itemize}
An \emph{$\mathfrak{sl}_3$-web} is an element in $\W(\bs;\Sigma)$ and a \emph{basis web} is an $\mathfrak{sl}_3$-web represented by a graph in $\Graph{\bs}{\Sigma}$ with no elliptic faces.
\end{dfn}

The $\mathfrak{sl}_3$-skein relation realizes the \emph{Reidemeister moves} (R1'), and (R2) -- (R4):
\begin{align*}
&\text{(R1')}~
\mathord{\ \tikz[baseline=-.6ex, scale=.1]{
\draw [thin, dashed] (0,0) circle [radius=5];
\draw (3,-2)
to[out=south, in=east] (2,-3)
to[out=west, in=south] (0,0)
to[out=north, in=west] (2,3)
to[out=east, in=north] (3,2);
\draw[overarc] (0,-5) 
to[out=north, in=west] (2,-1)
to[out=east, in=north] (3,-2);
\draw[overarc] (3,2)
to[out=south, in=east] (2,1)
to[out=west, in=south] (0,5);
}\ }
\tikz[baseline=-.6ex, scale=.5]{
\draw [<->, xshift=1.5cm] (1,0)--(2,0);
} 
\mathord{\ \tikz[baseline=-.6ex, scale=.1, xshift=3cm]{
\draw[thin, dashed] (0,0) circle [radius=5];
\draw (90:5) to (-90:5);
}\ }
&&\text{(R2)}~
\mathord{\ \tikz[baseline=-.6ex, scale=.1]{
\draw[thin, dashed] (0,0) circle [radius=5];
\draw (135:5) to[out=south east, in=west] (0,-2) to[out=east, in=south west](45:5);
\draw[overarc]
(-135:5) to[out=north east, in=west] (0,2) to[out=east, in=north west] (-45:5);
}\ }
\tikz[baseline=-.6ex, scale=.5]{
\draw [<->, xshift=1.5cm] (1,0)--(2,0);
}
\mathord{\ \tikz[baseline=-.6ex, scale=.1, xshift=3cm]{
\draw[thin, dashed] (0,0) circle [radius=5];
\draw (135:5) to[out=south east, in=west](0,2) to[out=east, in=south west](45:5);
\draw (-135:5) to[out=north east, in=west](0,-2) to[out=east, in=north west] (-45:5);
}\ },\\
&\text{(R3)}~
\mathord{\ \tikz[baseline=-.6ex, scale=.1]{
\draw [thin, dashed] (0,0) circle [radius=5];
\draw (-135:5) -- (45:5);
\draw[overarc] (135:5) -- (-45:5);
\draw[overarc] 
(180:5) to[out=east, in=west](0,3) to[out=east, in=west] (-0:5);
}\ }
\tikz[baseline=-.6ex, scale=.5]{
\draw[<->, xshift=1.5cm] (1,0)--(2,0);
}
\mathord{\ \tikz[baseline=-.6ex, scale=.1, xshift=3cm]{
\draw [thin, dashed] (0,0) circle [radius=5];
\draw (-135:5) -- (45:5);
\draw [overarc] (135:5) -- (-45:5);
\draw[overarc] 
(180:5) to[out=east, in=west] (0,-3) to[out=east, in=west] (-0:5);
}\ },
&&\text{(R4)}~
\mathord{\ \tikz[baseline=-.6ex, scale=.1]{
\draw [thin, dashed] (0,0) circle [radius=5];
\draw (0:0) -- (90:5); 
\draw (0:0) -- (210:5); 
\draw (0:0) -- (-30:5);
\draw[overarc]
(180:5) to[out=east, in=west] (0,3) to[out=east, in=west] (0:5);
}\ }
\tikz[baseline=-.6ex, scale=.5]{
\draw[<->, xshift=1.5cm] (1,0)--(2,0);
}
\mathord{\ \tikz[baseline=-.6ex, scale=.1]{
\draw [thin, dashed] (0,0) circle [radius=5];
\draw (0:0) -- (90:5); 
\draw (0:0) -- (210:5); 
\draw (0:0) -- (-30:5);
\draw[overarc]
(180:5) to[out=east, in=west] (0,-3) to[out=east, in=west](0:5);
}\ }, 
\mathord{\ \tikz[baseline=-.6ex, scale=.1]{
\draw [thin, dashed] (0,0) circle [radius=5];
\draw[overarc]
(180:5) to[out=east, in=west] (0,3) to[out=east, in=west] (0:5);
\draw[overarc] (0:0) -- (90:5); 
\draw (0:0) -- (210:5); 
\draw (0:0) -- (-30:5);
}\ }
\tikz[baseline=-.6ex, scale=.5]{
\draw[<->, xshift=1.5cm] (1,0)--(2,0);
}
\mathord{\ \tikz[baseline=-.6ex, scale=.1]{
\draw [thin, dashed] (0,0) circle [radius=5];
\draw[overarc]
(180:5) to[out=east, in=west] (0,-3) to[out=east, in=west](0:5);
\draw (0:0) -- (90:5); 
\draw[overarc] (0:0) -- (210:5); 
\draw[overarc] (0:0) -- (-30:5);
}\ }.
\end{align*}
The above means that $\mathfrak{sl}_3$-webs on the left and right sides represent the same element in an $\mathfrak{sl}_3$-web space for any choice of the orientation of edges.

It is easy to see that any tangled trivalent graph decomposes into a sum of basis web by using the $\mathfrak{sl}_3$ skein relation.
In fact, the set of basis webs consists of a basis of the $\mathfrak{sl}_3$-web space.
\begin{thm}[\cite{Kuperberg96}, \cite{SikoraWestbury}]
	The set of basis web on a surface $\Sigma$ with a signed marked points $\bs\colon P\to\{{+},{-}\}$ is a basis of $\W(\bs;\Sigma)$ as a $\mathbb{Z}[q^{\pm\frac{1}{6}}]$-module.
\end{thm}

In some cases, one can give the set of basis webs via an argument concerning the Euler characteristic.
\begin{eg}\label{eg:web-space}
	Let $D$ be a disk with a base point $\ast\in\partial D$.
	We identify signed marked points on $\partial D\setminus \{\ast\}$ with a sequence of signs.
	Then, the following isomorphisms hold for any $\epsilon\in\{{+},{-}\}$.
	\begin{enumerate}
		\item $\W(\emptyset;D)$ of a disk $D$ with no marked points is isomorphic to a free $\cR$-module spanned by the empty diagram $\varnothing$.
		\item $\W(\epsilon;D)=0$.
		\item $\W(\epsilon\epsilon;D)=0$.
		\item $\W(\epsilon\bar{\epsilon};D)$ is a free $\cR$-module spanned by an oriented simple arc.
		\item $\W(\epsilon\epsilon\bar{\epsilon};D)=0$.
		\item $\W(\epsilon\epsilon\epsilon;D)$ is a free $\cR$-module spanned by a tripod with a sink or source vertex.
	\end{enumerate}
	In the above, $\bar{\epsilon}$ means the opposite sign of $\epsilon$.
\end{eg}

We review a diagrammatic definition of an $\mathfrak{sl}_3$-clasp introduced in \cite{Kuperberg96, OhtsukiYamada97, Kim07, Yuasa20A} and note some useful properties. 
The $\mathfrak{sl}_3$-clasp plays a similar role to the Jones-Wenzl projector in the Kauffman bracket skein theory.

In what follows, we will mainly treat tangled trivalent graphs or $\mathfrak{sl}_{3}$-webs in a rectangle $D=[0,1]\times[0,1]$. 
We assume that the set of marked points lies in the top edge $I_1=[0,1]\times\{1\}$ and the bottom edge $I_0=[0,1]\times\{0\}$, and a base point $\ast$ at $(0,0)$.
In this situation, the set of marked points is divided into $P^{(0)}$ and $P^{(1)}$ where $P^{(j)}\coloneqq P\cap I_j$ and we denote the assignment of signs by $\bs^{(j)}\colon P^{(j)}\to \{{+},{-}\}$ for $j=0,1$.
One can identify $\bs^{(j)}$ with a sequence of signs on $[0,1]\times\{j\}$ arranged from $0$ to $1$.
We simplify a symbol $\G(\bs^{(0)}\sqcup \bs^{(1)};D)$ and $\W(\bs^{(0)}\sqcup \bs^{(1)};D)$ by $\G(\bs^{(0)},\bar{\bs}^{(1)})$ and $\TL{\bs^{(0)}}{\bar{\bs}^{(1)}}$ respectively where $\bar{\bs}^{(1)}$ is a sequence consists of the oppsite signs of $\bs^{(1)}$.\footnote{We take the opposite sign $\bar{\bs}^{(1)}$ to be consistent with the composition.}
When we describe diagrams representing $\mathfrak{sl}_3$-webs in $D$, we omit to write the signs, the basepoint, and the boundary of $D$.
The composition $\TL{\bs_{1}}{\bs_{2}}\otimes\TL{\bs_{0}}{\bs_{1}}\to\TL{\bs_{0}}{\bs_{2}}$ is defined by gluing the top side of an $\mathfrak{sl}_3$-web in $\TL{\bs_{0}}{\bs_{1}}$ and the bottom side of an $\mathfrak{sl}_3$-web in $\TL{\bs_{1}}{\bs_{2}}$.

We firstly define the $\mathfrak{sl}_3$-clasp in $\TL{{-}^{m}}{{-}^{m}}$.
\begin{dfn}[One-row colored $\mathfrak{sl}_3$-clasps]\label{onerowclasp}
The $\mathfrak{sl}_3$-clasp $\mathrm{JW}_{{-}^m}$ described by a white box with $m\in\mathbb{Z}_{{}>0}$ is defined as follows.
\begin{enumerate}
\item
$\mathrm{JW}_{{-}}
=\mathord{\ \tikz[baseline=-.6ex, scale=.1]{
\draw[->-=.8] (0,-6) -- (0,6);
\draw[fill=white] (-3,-1) rectangle (3,1);
\node at (0,5) [left]{${\scriptstyle 1}$};
}\ }
\coloneqq 
\,\tikz[baseline=-.6ex, scale=.1]{
\draw[->-=.5] (0,-6) -- (0,6); 
}\,$

\item
$\mathrm{JW}_{{-}^{m+1}}
=\hspace{-1em}\mathord{\tikz[baseline=-.6ex, scale=.1]{
\draw[->-=.8] (0,-6) -- (0,6);
\draw[fill=white] (-3,-1) rectangle (3,1);
\node at (0,5) [left]{${\scriptstyle m+1}$};
}\ }
\coloneqq \!
\mathord{\tikz[baseline=-.6ex, scale=.1]{
\draw[->-=.8] (0,-6) -- (0,6);
\draw[->-=.5] (5,-6) -- (5,6);
\draw[fill=white] (-2,-1) rectangle (2,1);
\node at (0,5) [left]{${\scriptstyle m}$};
}\ }
-\frac{\left[m\right]}{\left[m+1\right]}
\mathord{\ \tikz[baseline=-.6ex, scale=.1]{
\draw[->-=.6] (0,4) -- (0,7);
\draw[->-=.6] (0,-7) -- (0,-4);
\draw[->-=.5] (-2,-3) -- (-2,3);
\draw[-<-=.7] (5,-1) -- (5,1);
\draw[-<-=.7] (2,3) to[out=south, in=west] (5,1);
\draw[-<-=.3] (8,7) -- (8,3) to[out=south, in=east] (5,1);
\draw[->-=.7, yscale=-1] (2,3) to[out=south, in=west] (5,1);
\draw[->-=.3, yscale=-1] (8,7) -- (8,3) to[out=south, in=east] (5,1);
\draw[fill=white] (-3,3) rectangle (3,4);
\draw[fill=white] (-3,-3) rectangle (3,-4);
\node at (0,6) [left]{${\scriptstyle m}$};
\node at (0,-6) [left]{${\scriptstyle m}$};
\node at (-2,0) [left]{${\scriptstyle m-1}$};
}\ }$.
\end{enumerate}
In the above, an edge labeled by $m$ represents the $m$ parallelization of the edge. 
$\mathrm{JW}_{{+}^{m}}$ is defined by the same diagram with oppositely directed edges.
\end{dfn}

We next introduce the $\mathfrak{sl}_3$-clasp in $\TL{{-}^{m}{+}^{n}}{{-}^{m}{+}^{n}}$.

\begin{dfn}[Two-row colored $\mathfrak{sl}_3$-clasps]\label{tworowclasp}
\[
\mathrm{JW}_{{-}^{m}{+}^{n}}
=
\mathord{\ \tikz[baseline=-.6ex, scale=.1]{
\draw[->-=.2, ->-=.8] (-2,-5) -- (-2,5);
\draw[-<-=.2, -<-=.8] (2,-5) -- (2,5);
\draw[fill=white] (-4,-1) rectangle (4,1);
\node at (-2,5) [left]{${\scriptstyle m}$};
\node at (-2,-5) [left]{${\scriptstyle m}$};
\node at (2,5) [right]{${\scriptstyle n}$};
\node at (2,-5) [right]{${\scriptstyle n}$};
}\ }
=
\sum_{i=0}^{\min\{m,n\}}
(-1)^i
\frac{{m\brack i}{n\brack i}}{{m+n+1\brack i}}
\mathord{\ \tikz[baseline=-.6ex, scale=.1]{
\draw[->-=.5] (-4,5) -- (-4,8);
\draw[-<-=.5] (4,5) -- (4,8);
\draw[-<-=.5, yscale=-1] (-4,5) -- (-4,8);
\draw[->-=.5, yscale=-1] (4,5) -- (4,8);
\draw[->-=.5] (-5,-4) -- (-5,4);
\draw[-<-=.5] (5,-4) -- (5,4);
\draw[-<-=.5] (-3,4) to[out=south, in=south] (3,4);
\draw[->-=.5, yscale=-1] (-3,4) to[out=south, in=south] (3,4);
\draw[fill=white] (-6,4) rectangle (-2,5);
\draw[fill=white, xscale=-1] (-6,4) rectangle (-2,5);
\draw[fill=white, yscale=-1] (-6,4) rectangle (-2,5);
\draw[fill=white, xscale=-1, yscale=-1] (-6,4) rectangle (-2,5);
\node at (-5,0) [left]{${\scriptstyle m-i}$};
\node at (5,0) [right]{${\scriptstyle n-i}$};
\node at (0,5) {${\scriptstyle i}$};
\node at (0,-5) {${\scriptstyle i}$};
\node at (-4,7) [left]{${\scriptstyle m}$};
\node at (4,7) [right]{${\scriptstyle n}$};
\node at (-4,-7) [left]{${\scriptstyle m}$};
\node at (4,-7) [right]{${\scriptstyle n}$};
}\ }.\]
One can define $\mathrm{JW}_{{+}^{m}{-}^{n}}$ in the same way.
\end{dfn}

For convenience in computation of $\mathfrak{sl}_3$-web, we will introduce ``stair-step'' and ``triangle'' webs in \cref{def:step,def:triangle} that also used in \cite{Kim06, Kim07, Yuasa17, FrohmanSikora22}.
In these definitions, the orientation of edges of $\mathfrak{sl}_3$-webs is not explicitly given because it is uniquely determined according to your choice of the orientation of an edge around the box.
\begin{dfn}\label{def:step}
For positive integers $n$ and $m$, a \emph{stair-step web}
$
\mathord{\ \tikz[baseline=-.6ex, scale=.1]{
\draw (-4,0) -- (-2,0);
\draw (2,0) -- (4,0);
\draw (0,-4) -- (0,-2);
\draw (0,2) -- (0,4);
\draw[fill=white] (-2,-2) rectangle (2,2);
\draw (-2,2) -- (2,-2);
\node at (-4,0) [left]{$\scriptstyle{n}$};
\node at (4,0) [right]{$\scriptstyle{n}$};
\node at (0,4) [above]{$\scriptstyle{m}$};
\node at (0,-4) [below]{$\scriptstyle{m}$};
}\ }
$
is defined as
\begin{align*}
&\mathord{\ \tikz[baseline=-.6ex, scale=.1]{
\draw (-4,0) -- (-2,0);
\draw (2,0) -- (4,0);
\draw (0,-4) -- (0,-2);
\draw (0,2) -- (0,4);
\draw[fill=white] (-2,-2) rectangle (2,2);
\draw (-2,2) -- (2,-2);
\node at (-4,0) [left]{$\scriptstyle{n}$};
\node at (4,0) [right]{$\scriptstyle{n}$};
\node at (0,4) [above]{$\scriptstyle{1}$};
\node at (0,-4) [below]{$\scriptstyle{1}$};
}\ }
=
\mathord{\ \tikz[baseline=-.6ex, scale=.1]{
\draw (-7,-5) -- (7,-5);
\draw (-7,-3) -- (7,-3);
\draw (-7,5) -- (7,5);
\draw (4,5) -- (4,7);
\draw (3,3) -- (3,5);
\draw (-4,-7) -- (-4,-5);
\draw (-3,-5) -- (-3,-3);
\draw (-2,-3) -- (-2,-1);
\node[yscale=2.5] at (-9,0) {${\{}$};
\node at (-9,0) [left]{$\scriptstyle{n}$};
\node[rotate=90] at (-6,1){$\scriptstyle{\cdots}$};
\node[rotate=90] at (6,1){$\scriptstyle{\cdots}$};
\node[rotate=45] at (0,1){$\scriptstyle{\cdots}$};
}\ }\quad\text{and}
\mathord{\ \tikz[baseline=-.6ex, scale=.1]{
\draw (-4,0) -- (-2,0);
\draw (2,0) -- (4,0);
\draw (0,-4) -- (0,-2);
\draw (0,2) -- (0,4);
\draw[fill=white] (-2,-2) rectangle (2,2);
\draw (-2,2) -- (2,-2);
\node at (-4,0) [left]{$\scriptstyle{n}$};
\node at (4,0) [right]{$\scriptstyle{n}$};
\node at (0,4) [above]{$\scriptstyle{m}$};
\node at (0,-4) [below]{$\scriptstyle{m}$};
}\ }
=
\mathord{\ \tikz[baseline=-.6ex, scale=.1]{
\draw (-4,0) -- (-2,0);
\draw (2,0) -- (4,0);
\draw (0,-4) -- (0,-2);
\draw (0,2) -- (0,4);
\draw[fill=white] (-2,-2) rectangle (2,2);
\draw (-2,2) -- (2,-2);
\node at (-4,0) [left]{$\scriptstyle{n}$};
\node at (0,4) [above]{$\scriptstyle{m-1}$};
\node at (0,-4) [below]{$\scriptstyle{m-1}$};
\begin{scope}[xshift=8cm]
\draw (-4,0) -- (-2,0);
\draw (2,0) -- (4,0);
\draw (0,-4) -- (0,-2);
\draw (0,2) -- (0,4);
\draw[fill=white] (-2,-2) rectangle (2,2);
\draw (-2,2) -- (2,-2);
\node at (-4,0) [above]{$\scriptstyle{n}$};
\node at (4,0) [right]{$\scriptstyle{n}$};
\node at (0,4) [above]{$\scriptstyle{1}$};
\node at (0,-4) [below]{$\scriptstyle{1}$};
\end{scope}
}\ }
\quad\text{for}\quad m>1.
\end{align*}
\end{dfn}

\begin{dfn}\label{def:triangle}
For a positive integer $n$,
$
\mathord{\ \tikz[baseline=-.6ex, scale=.1]{
\draw (30:5) -- (0,0);
\draw (150:5) -- (0,0);
\draw (270:5) -- (0,0);
\draw[fill=white] (-30:4) -- (90:4) -- (210:4) -- cycle;
\node at (30:5) [below]{$\scriptstyle{n}$};
\node at (150:5) [below]{$\scriptstyle{n}$};
\node at (270:5) [left]{$\scriptstyle{n}$};
}\ }
$
is defined by
$\mathord{\ \tikz[baseline=-.6ex, scale=.1]{
\draw (30:5) -- (0,0);
\draw (150:5) -- (0,0);
\draw (270:5) -- (0,0);
\draw[fill=white] (-30:4) -- (90:4) -- (210:4) -- cycle;
\node at (30:5) [below]{$\scriptstyle{1}$};
\node at (150:5) [below]{$\scriptstyle{1}$};
\node at (270:5) [left]{$\scriptstyle{1}$};
}\ }
=\mathord{\ \tikz[baseline=-.6ex, scale=.1]{
\draw (30:5) -- (0,0);
\draw (150:5) -- (0,0);
\draw (270:5) -- (0,0);
}\ }$ and 
$\mathord{\ \tikz[baseline=-.6ex, scale=.1]{
\draw (30:5) -- (0,0);
\draw (150:5) -- (0,0);
\draw (270:5) -- (0,0);
\draw[fill=white] (-30:4) -- (90:4) -- (210:4) -- cycle;
\node at (30:5) [below]{$\scriptstyle{n}$};
\node at (150:5) [below]{$\scriptstyle{n}$};
\node at (270:5) [left]{$\scriptstyle{n}$};
}\ }
=\mathord{\ \tikz[baseline=-.6ex, scale=.1]{
\draw (-10,5) -- (5,5);
\draw (-5,-4) -- (-5,5);
\draw (-10,1) -- (5,1);
\draw (0,-4) -- (0,1);
\draw[fill=white] (-30:3) -- (90:3) -- (210:3) -- cycle;
\draw[fill=white] (-6,-2) rectangle (-4,3);
\draw (-6,3) -- (-4,-2);
\node at (-10,5) [left]{$\scriptstyle{1}$};
\node at (5,5) [right]{$\scriptstyle{1}$};
\node at (-5,-4) [below]{$\scriptstyle{1}$};
\node at (-10,0) [below]{$\scriptstyle{n-1}$};
\node at (4,0) [above]{$\scriptstyle{n-1}$};
\node at (0,-4) [below]{$\scriptstyle{n-1}$};
}\ }
$
for $n>1$.
\end{dfn}

\begin{dfn}[General type of the $\mathfrak{sl}_3$-clasps]
Let $\bs_{1}$ and $\bs_{2}$ be two sequences of signs which consist of $m$ pluses and $n$ minuses.
We define an $\mathfrak{sl}_3$-clasp in $\TL{\bs_{1}}{\bar{\bs}_{2}}$ by gluing stair-step webs to the top and the bottom of $\mathrm{JW}_{{-}^{m}{+}^{n}}$ as follows,
\[
\mathord{\ \tikz[baseline=-.6ex, scale=.1]{
\draw[->-=.5] (-3,1) to[out=north, in=south west] (0,4);
\draw[-<-=.5] (3,1) to[out=north, in=south east] (0,4);
\draw[->-=.8] (0,4) to[out=north east, in=south] (3,7);
\draw[-<-=.8] (0,4) to[out=north west, in=south] (-3,7);
\draw (0,-5) -- (0,0);
\draw[fill=white] (0,2) -- (-2,4) -- (0,6) -- (2,4) -- cycle;
\draw (0,2) -- (0,6);
\draw[fill=white] (-6,-1) rectangle (6,1);
\node at (5,4) {${\scriptstyle {\dots}}$};
\node at (-5,4) {${\scriptstyle {\dots}}$};
\node at (-3,7) [above]{${\scriptstyle a}$};
\node at (3,7) [above]{${\scriptstyle b}$};
}\ },
\quad\mathord{\ \tikz[baseline=-.6ex, scale=.1, yshift=2cm]{
\draw[-<-=.5] (-3,-1) to[out=south, in=north west] (0,-4);
\draw[->-=.5] (3,-1) to[out=south, in=north east] (0,-4);
\draw[-<-=.8] (0,-4) to[out=south east, in=north] (3,-7);
\draw[->-=.8] (0,-4) to[out=south west, in=north] (-3,-7);
\draw (0,5) -- (0,0);
\draw[fill=white] (0,-2) -- (-2,-4) -- (0,-6) -- (2,-4) -- cycle;
\draw (0,-2) -- (0,-6);
\draw[fill=white] (-6,1) rectangle (6,-1);
\node at (5,-4) {${\scriptstyle {\dots}}$};
\node at (-5,-4) {${\scriptstyle {\dots}}$};
\node at (-3,-7) [below]{${\scriptstyle a}$};
\node at (3,-7) [below]{${\scriptstyle b}$};
}\ }.
\]
Repeating these operations, one can arbitrarily exchange signs in sequences on the top and bottom of the disk respectively, and we obtain an $\mathfrak{sl}_3$-web in $\TL{\bs_{1}}{\bar{\bs}_{2}}$.
We denote it by $\mathrm{JW}_{\bs_1}^{\bar{\bs}_{2}}$ and depict it as
\[
	\mathord{\ 
		\tikz[baseline=-.6ex, scale=.1]{
		\draw (0,-6) -- (0,6);
		\draw[fill=white] (-4,-1) rectangle (4,1);
		\node at (0,4) [left]{${\scriptstyle \bs_{2}}$};
		\node at (0,-4) [left]{${\scriptstyle \bs_{1}}$};
		}
	\ }.
\]
\end{dfn}
The resulting $\mathfrak{sl}_3$-clasp is independent of a choice of sequences of stair-step webs, see \cite{Yuasa20A}.
Thus, $\mathrm{JW}_{\bs_1}^{\bar{\bs}_{2}}$ is uniquely determined by $\bs_1$ and $\bs_2$.

One can prove the following useful formulas for $\mathfrak{sl}_3$-clasps by straightforward computation.
\begin{lem}\label{claspformula}
	Let $\bs_1, \bs_2$, and $\bs_3$ are sequences of sings. An arc labeled by a positive integer $m$ (resp.~$n$) denotes $m$ (resp.~$n$) parallelization of the arc.
\begin{align*}
(1)&\mathord{\ \tikz[baseline=-.6ex, scale=.1]{
\draw (0,-6) -- (0,6);
\draw[fill=white] (-4,-4) rectangle (4,-2);
\draw[fill=white] (-4,4) rectangle (4,2);
\node at (0,5) [left]{${\scriptstyle \bs_3}$};
\node at (0,0) [left]{${\scriptstyle \bs_2}$};
\node at (0,-5) [left]{${\scriptstyle \bs_1}$};
}\ }
=
\mathord{\ \tikz[baseline=-.6ex, scale=.1]{
\draw (0,-6) -- (0,6);
\draw[fill=white] (-4,-1) rectangle (4,1);
\node at (0,4) [left]{${\scriptstyle \bs_3}$};
\node at (0,-4) [left]{${\scriptstyle \bs_1}$};
}\ }
,
\quad\mathord{\ \tikz[baseline=-.6ex, scale=.1]{
\draw (-2,1) -- (0,3);
\draw (2,1) -- (0,3);
\draw (0,6) -- (0,3);
\draw (0,-6) -- (0,0);
\draw[fill=white] (-6,-1) rectangle (6,1);
\node at (4,3) {${\scriptstyle {\dots}}$};
\node at (-4,3) {${\scriptstyle {\dots}}$};
}\ }=0
,
\quad\mathord{\ \tikz[baseline=-.6ex, scale=.1]{
\draw (-2,1) to[out=north, in=west] (0,3) to[out=east, in=north] (2,1);
\draw (0,-6) -- (0,0);
\draw[fill=white] (-6,-1) rectangle (6,1);
\node at (4,3) {${\scriptstyle {\dots}}$};
\node at (-4,3) {${\scriptstyle {\dots}}$};
}\ }=0
,\\
(2)&\mathord{\ \tikz[baseline=-.6ex, scale=.1]{
\draw[-<-=.5] (-3,1) to[out=north, in=south west] (0,4);
\draw[->-=.5] (3,1) to[out=north, in=south east] (0,4);
\draw[-<-=.8] (0,4) to[out=north east, in=south] (3,7);
\draw[->-=.8] (0,4) to[out=north west, in=south] (-3,7);
\draw (0,-5) -- (0,0);
\draw[fill=white] (0,2) -- (-2,4) -- (0,6) -- (2,4) -- cycle;
\draw (0,2) -- (0,6);
\draw[fill=white] (-6,-1) rectangle (6,1);
\node at (5,4) {${\scriptstyle {\dots}}$};
\node at (-5,4) {${\scriptstyle {\dots}}$};
\node at (-3,7) [above]{${\scriptstyle m}$};
\node at (3,7) [above]{${\scriptstyle n}$};
}\ }
=
\mathord{\ \tikz[baseline=-.6ex, scale=.1]{
\draw[->-=.5] (-2,1) -- (-2,7);
\draw[-<-=.5] (2,1) -- (2,7);
\draw (0,-5) -- (0,0);
\draw[fill=white] (-6,-1) rectangle (6,1);
\node at (5,4) {${\scriptstyle {\dots}}$};
\node at (-5,4) {${\scriptstyle {\dots}}$};
\node at (-2,7) [above]{${\scriptstyle m}$};
\node at (2,7) [above]{${\scriptstyle n}$};
}\ }
,
\quad\mathord{\ \tikz[baseline=-.6ex, scale=.1]{
\draw[->-=.5] (-3,1) to[out=north, in=south west] (0,4);
\draw[-<-=.5] (3,1) to[out=north, in=south east] (0,4);
\draw[->-=.8] (0,4) to[out=north east, in=south] (3,7);
\draw[-<-=.8] (0,4) to[out=north west, in=south] (-3,7);
\draw (0,-5) -- (0,0);
\draw[fill=white] (0,2) -- (-2,4) -- (0,6) -- (2,4) -- cycle;
\draw (0,2) -- (0,6);
\draw[fill=white] (-6,-1) rectangle (6,1);
\node at (5,4) {${\scriptstyle {\dots}}$};
\node at (-5,4) {${\scriptstyle {\dots}}$};
\node at (-3,7) [above]{${\scriptstyle m}$};
\node at (3,7) [above]{${\scriptstyle n}$};
}\ }
=
\mathord{\ \tikz[baseline=-.6ex, scale=.1]{
\draw[-<-=.5] (-2,1) -- (-2,7);
\draw[->-=.5] (2,1) -- (2,7);
\draw (0,-5) -- (0,0);
\draw[fill=white] (-6,-1) rectangle (6,1);
\node at (5,4) {${\scriptstyle {\dots}}$};
\node at (-5,4) {${\scriptstyle {\dots}}$};
\node at (-2,7) [above]{${\scriptstyle m}$};
\node at (2,7) [above]{${\scriptstyle n}$};
}\ }
,\\
(3)&
\mathord{\ \tikz[baseline=-.6ex, scale=.1]{
\draw[->-=.8, overarc] (3,1) to[out=north, in=south] (-3,7);
\draw[-<-=.8, overarc] (-3,1) to[out=north, in=south] (3,7);
\draw (0,-5) -- (0,0);
\draw[fill=white] (-6,-1) rectangle (6,1);
\node at (5,4) {${\scriptstyle {\dots}}$};
\node at (-5,4) {${\scriptstyle {\dots}}$};
\node at (-3,7) [above]{${\scriptstyle m}$};
\node at (3,7) [above]{${\scriptstyle n}$};
}\ }
=(-1)^{mn}q^{\frac{mn}{6}}
\mathord{\ \tikz[baseline=-.6ex, scale=.1]{
\draw[->-=.5] (-2,1) -- (-2,7);
\draw[-<-=.5] (2,1) -- (2,7);
\draw (0,-5) -- (0,0);
\draw[fill=white] (-6,-1) rectangle (6,1);
\node at (5,4) {${\scriptstyle {\dots}}$};
\node at (-5,4) {${\scriptstyle {\dots}}$};
\node at (-2,7) [above]{${\scriptstyle m}$};
\node at (2,7) [above]{${\scriptstyle n}$};
}\ }
,\quad
\mathord{\ \tikz[baseline=-.6ex, scale=.1]{
\draw[-<-=.8, overarc] (-3,1) to[out=north, in=south] (3,7);
\draw[->-=.8, overarc] (3,1) to[out=north, in=south] (-3,7);
\draw (0,-5) -- (0,0);
\draw[fill=white] (-6,-1) rectangle (6,1);
\node at (5,4) {${\scriptstyle {\dots}}$};
\node at (-5,4) {${\scriptstyle {\dots}}$};
\node at (-3,7) [above]{${\scriptstyle m}$};
\node at (3,7) [above]{${\scriptstyle n}$};
}\ }
=(-1)^{mn}q^{-\frac{mn}{6}}
\mathord{\ \tikz[baseline=-.6ex, scale=.1]{
\draw[->-=.5] (-2,1) -- (-2,7);
\draw[-<-=.5] (2,1) -- (2,7);
\draw (0,-5) -- (0,0);
\draw[fill=white] (-6,-1) rectangle (6,1);
\node at (5,4) {${\scriptstyle {\dots}}$};
\node at (-5,4) {${\scriptstyle {\dots}}$};
\node at (-2,7) [above]{${\scriptstyle m}$};
\node at (2,7) [above]{${\scriptstyle n}$};
}\ }
,\\
(4)&
\mathord{\ \tikz[baseline=-.6ex, scale=.1]{
\draw[->-=.8, overarc] (3,1) to[out=north, in=south] (-3,7);
\draw[->-=.8, overarc] (-3,1) to[out=north, in=south] (3,7);
\draw (0,-5) -- (0,0);
\draw[fill=white] (-6,-1) rectangle (6,1);
\node at (5,4) {${\scriptstyle {\dots}}$};
\node at (-5,4) {${\scriptstyle {\dots}}$};
\node at (-3,7) [above]{${\scriptstyle m}$};
\node at (3,7) [above]{${\scriptstyle n}$};
}\ }
=q^{\frac{mn}{3}}
\mathord{\ \tikz[baseline=-.6ex, scale=.1]{
\draw[->-=.5] (-2,1) -- (-2,7);
\draw[->-=.5] (2,1) -- (2,7);
\draw (0,-5) -- (0,0);
\draw[fill=white] (-6,-1) rectangle (6,1);
\node at (5,4) {${\scriptstyle {\dots}}$};
\node at (-5,4) {${\scriptstyle {\dots}}$};
\node at (-2,7) [above]{${\scriptstyle m}$};
\node at (2,7) [above]{${\scriptstyle n}$};
}\ }
,\quad
\mathord{\ \tikz[baseline=-.6ex, scale=.1]{
\draw[->-=.8, overarc] (-3,1) to[out=north, in=south] (3,7);
\draw[->-=.8, overarc] (3,1) to[out=north, in=south] (-3,7);
\draw (0,-5) -- (0,0);
\draw[fill=white] (-6,-1) rectangle (6,1);
\node at (5,4) {${\scriptstyle {\dots}}$};
\node at (-5,4) {${\scriptstyle {\dots}}$};
\node at (-3,7) [above]{${\scriptstyle m}$};
\node at (3,7) [above]{${\scriptstyle n}$};
}\ }
=q^{-\frac{mn}{3}}
\mathord{\ \tikz[baseline=-.6ex, scale=.1]{
\draw[->-=.5] (-2,1) -- (-2,7);
\draw[->-=.5] (2,1) -- (2,7);
\draw (0,-5) -- (0,0);
\draw[fill=white] (-6,-1) rectangle (6,1);
\node at (5,4) {${\scriptstyle {\dots}}$};
\node at (-5,4) {${\scriptstyle {\dots}}$};
\node at (-2,7) [above]{${\scriptstyle m}$};
\node at (2,7) [above]{${\scriptstyle n}$};
}\ }
,\\
(5)&\mathord{\ \tikz[baseline=-.6ex, scale=.1]{
\draw[->-=.8] (0,0) -- (-10,0);
\draw[->-=.8] (0,0) -- (10,0);
\draw[-<-=.2] (0,-8) -- (0,0);
\draw[fill=white] (-4,-3) -- (0,4) -- (4,-3) -- cycle;
\draw[fill=white] (-3,-5) rectangle (3,-4);
\draw[fill=white] (-6,-3) rectangle (-5,3);
\draw[fill=white] (6,-3) rectangle (5,3);
\node at (-8,0) [above]{${\scriptstyle n}$};
\node at (8,0) [above]{${\scriptstyle n}$};
\node at (0,-7) [right]{${\scriptstyle n}$};
}\ }
=\mathord{\ \tikz[baseline=-.6ex, scale=.1]{
	\draw[->-=.8] (0,0) -- (-10,0);
	\draw[->-=.8] (0,0) -- (10,0);
	\draw[-<-=.2] (0,-8) -- (0,0);
	\draw[fill=white] (-4,-3) -- (0,4) -- (4,-3) -- cycle;
	\draw[fill=white] (-3,-5) rectangle (3,-4);
	\draw[fill=white] (-6,-3) rectangle (-5,3);
	\node at (-8,0) [above]{${\scriptstyle n}$};
	\node at (6,0) [above]{${\scriptstyle n}$};
	\node at (0,-7) [right]{${\scriptstyle n}$};
}\ },
\mathord{\ \tikz[baseline=-.6ex, scale=.1]{
	\draw[-<-=.8] (0,0) -- (-10,0);
	\draw[-<-=.8] (0,0) -- (10,0);
	\draw[->-=.2] (0,-8) -- (0,0);
	\draw[fill=white] (-4,-3) -- (0,4) -- (4,-3) -- cycle;
	\draw[fill=white] (-3,-5) rectangle (3,-4);
	\draw[fill=white] (-6,-3) rectangle (-5,3);
	\draw[fill=white] (6,-3) rectangle (5,3);
	\node at (-8,0) [above]{${\scriptstyle n}$};
	\node at (8,0) [above]{${\scriptstyle n}$};
	\node at (0,-7) [right]{${\scriptstyle n}$};
}\ }
=\mathord{\ \tikz[baseline=-.6ex, scale=.1]{
	\draw[-<-=.8] (0,0) -- (-10,0);
	\draw[-<-=.8] (0,0) -- (10,0);
	\draw[->-=.2] (0,-8) -- (0,0);
	\draw[fill=white] (-4,-3) -- (0,4) -- (4,-3) -- cycle;
	\draw[fill=white] (-3,-5) rectangle (3,-4);
	\draw[fill=white] (-6,-3) rectangle (-5,3);
	\node at (-8,0) [above]{${\scriptstyle n}$};
	\node at (6,0) [above]{${\scriptstyle n}$};
	\node at (0,-7) [right]{${\scriptstyle n}$};
}\ },\\
(6)&\ \tikz[baseline=-.6ex, scale=.1,yshift=-2cm]{
	\draw[overarc] (-4,0) -- (-1,0) to[out=north east, in=east] (0,5);
	\draw[->-=.8, overarc] (0,5) to[out=west, in=north west] (1,0) -- (4,0);
	\draw (4,0) -- (6,0);
	\draw[fill=white] (4,-2) rectangle (5,2);
	\node at (4,2) [above]{${\scriptstyle n}$};
\ }
=q^{\frac{n^2+3n}{3}}
\!\tikz[baseline=-.6ex, scale=.1,yshift=-2cm]{
	\draw[->-=.5] (-4,0) -- (4,0);
	\draw (4,0) -- (6,0);
	\draw[fill=white] (3,-2) rectangle (4,2);
	\node at (3,2) [above]{${\scriptstyle n}$};
\ }
,\quad
\ \tikz[baseline=-.6ex, scale=.1,yshift=-2cm]{
	\draw[overarc] (0,5) to[out=west, in=north west] (1,0) -- (4,0);
	\draw[->-=.2, overarc] (-4,0) -- (-1,0) to[out=north east, in=east] (0,5);
	\draw (4,0) -- (6,0);
	\draw[fill=white] (4,-2) rectangle (5,2);
	\node at (4,2) [above]{${\scriptstyle n}$};
\ }
=q^{-\frac{n^2+3n}{3}}
\!\tikz[baseline=-.6ex, scale=.1,yshift=-2cm]{
	\draw[->-=.5] (-4,0) -- (4,0);
	\draw (4,0) -- (6,0);
	\draw[fill=white] (3,-2) rectangle (4,2);
	\node at (3,2) [above]{${\scriptstyle n}$};
\ }.
\end{align*}
\end{lem}
\begin{proof}
One can prove $(1)$--$(6)$ by using induction on labels and skein relations.
See \cite{Yuasa17,Yuasa20A} for exmaple.
\end{proof}

We give a definition of the one-row colored $\mathfrak{sl}_3$-Jones polynomial of oriented framed links via an $\mathfrak{sl}_3$-web.
Firstly, we introduce a normalization of a Laurent series by shifting the $q$-degree and changing the sign.
\begin{dfn}[Minimum degree]\label{degnormalization} 
    We define the \emph{minimum degree $\mathrm{d}_*\colon\cR=\mathbb{Z}((q^{\frac{1}{6}}))\to \frac{1}{6}\bZ\cup\infty$} by $\mdeg{f(q)}\coloneqq d/6$ for a non-zero series $f(q)=\sum_{i=d}^{\infty}a_{i}q^{\frac{i}{6}}$ in $\mathbb{Z}((q^{\frac{1}{6}}))$ such that $a_d\neq 0$.
    For the zero polynomial, we define its minimum degree as $\infty$.
    We also introduce a normalization $\hat{f}(q)$ of a non-zero Laurent series $f(q)$ as
    \[
	\hat{f}(q)\coloneqq \pm q^{-\mdeg{f(q)}}f(q)=\pm\sum_{i=0}^{\infty}a_{i+d}q^{\frac{i}{6}}\in\mathbb{Z}[[q^\frac{1}{6}]].
    \]
    In the above, we choose the sign so that the constant term $\pm a_{d}$ becomes positive.
\end{dfn}

We note some properties for the minimum degree and useful examples.
\begin{lem}\label{lem:deg}
	For any $f(q),g(q)\in \cR$,
	\begin{enumerate}
		\item $\mdeg{f(q)+g(q)}\geq\min\{\mdeg{f(q)},\mdeg{g(q)}\}$,	
		\item $\mdeg{f(q)g(q)}=\mdeg{f(q)}+\mdeg{g(q)}$.
	\end{enumerate}
    The equality in (1) holds if and only if $\mdeg{f(q)}\neq\mdeg{g(q)}$ or $\mdeg{f(q)}=\mdeg{g(q)}=:d$ with $a_d+b_d\neq 0$ where $f(q)=\sum_{i=d}^{\infty}a_{i}q^{\frac{i}{6}}$ and $g(q)=\sum_{i=d}^{\infty}b_{i}q^{\frac{i}{6}}$.
\end{lem}
\begin{eg}\label{eg:deg}
	For any positive integer $n$ and $1\leq k\leq n$,
	\begin{align*}
		&\mdeg{[n]}=-(n-1)/2,\quad \mdeg{[n]^{-1}}=(n-1)/2,\\
		&\mdeg{{n\brack k}}=-k(n-k)/2,\quad \mdeg{{n\brack k}^{-1}}=k(n-k)/2.
	\end{align*}
\end{eg}
We remark that one can confirm them by $(1-q^m)^{-1}=1+q^m+q^{2m}+\cdots\in\cR=\bZ((q^{\frac{1}{6}}))$ and \cref{lem:deg}.

\begin{dfn}
Let $L$ be a link diagram of a framed link whose framing is given by the blackboard framing.
One can replace arcs of the link diagram with $n$ parallelized arcs and put white boxes on the $n$ parallelized arcs.
The resulting diagram denoted by $L^{(n)}$ represents an $\mathfrak{sl}_3$-web in a disk $D$ with no marked points.\footnote{Such $\mathfrak{sl}_3$-web space $\W(\emptyset;D)$ is spanned by the empty diagram $\varnothing$, see \cref{eg:web-space}.}
The \emph{one-row colored $\mathfrak{sl}_3$-Jones polynomial $J^{\mathfrak{sl}_3}_{L,n}(q)$} with $(n,0)$-coloring (or $n$ boxes) is defined by
$L^{(n)}=J^{\mathfrak{sl}_3}_{L,n}(q)\varnothing$.
We also define a variation of the one-row colored $\mathfrak{sl}_3$-Jones polynomial as $\hat{J}^{\mathfrak{sl}_3}_{{L},{n}}(q)$ according to \cref{degnormalization}.
\end{dfn}

\begin{rmk}\label{rem:invariant}
	\begin{itemize}
		\item Skein relations realize the Reidemeister moves (R1')--(R4) for arcs with one-row colored clasps because clasped arcs are expressed as a linear combination of $\mathfrak{sl}_3$-webs. Hence $J^{\mathfrak{sl}_3}_{L,n}(q)$ is an invariant of framed links.
		\item The choice of framing of $L$ appears as multiplication by $\pm q^{\bullet}$, see $(6)$ in \cref{claspformula}. This difference is ignored in the normalization $\hat{J}^{\mathfrak{sl}_3}_{{L},{n}}(q)$, thus it is an invariant of links.
		\item L\^{e} showed the integrality theorem for a quantum $\mathfrak{g}$ invariant of links in~\cite{Le00}. 
		It says that $\hat{J}^{\mathfrak{sl}_3}_{{L},{n}}(q)$ belongs to $\mathbb{Z}[q]$.		
	\end{itemize}
\end{rmk}

We will discuss \emph{zero stability} of $\hat{J}^{\mathfrak{sl}_3}_{{L},{n}}(q)$ for a certain class of links in the following sections.
Let us recall the definition of zero stability and tails of the one-row colored $\mathfrak{sl}_3$-Jones polynomials.
\begin{dfn}[One-row colored $\mathfrak{sl}_3$-tail]
The one-row colored $\mathfrak{sl}_{3}$-Jones polynomial $\{\hat{J}^{\mathfrak{sl}_3}_{{L},{n}}(q)\}_{n}$ of a link $L$ is \emph{zero stable} if there exists a formal power series $\Phi^{\mathfrak{sl}_{3}}_{L}(q)$ in $\mathbb{Z}[[q]]$ such that 
\[
		\Phi^{\mathfrak{sl}_{3}}_{L}(q)-\hat{J}^{\mathfrak{sl}_3}_{{L},{n}}(q)\in q^{n+1}\mathbb{Z}[[q]]
\]
for all $n\geq 1$.
We call $\Phi^{\mathfrak{sl}_{3}}_{L}(q)$ the \emph{one-row colored $\mathfrak{sl}_3$-tail} of $L$ or simply the \emph{$\mathfrak{sl}_3$-tail} of $L$ when $\{\hat{J}^{\mathfrak{sl}_3}_{{L},{n}}(q)\}_{n}$ is zero stable.
\end{dfn}
\section{The minimum degree of clasped $\mathfrak{sl}_3$-webs}\label{sec:degree}
We will prove the existence of the $\mathfrak{sl}_3$-tail of the one-row colored $\mathfrak{sl}_3$-Jones polynomial by developing $\mathfrak{sl}_3$ analog of Armond's argument using the Kauffman bracket in \cite{Armond13}.
In the present section, we will discuss a lower bound of the minimum degree of a clasped $\mathfrak{sl}_3$-web with no crossings in a disk.	

First, we prepare a lemma that studies isomorphisms in \cref{eg:web-space} in detail
\begin{lem}\label{lem:reduction-seq}
    Let $G\in\G({-}{+};D)$ be a connected flat trivalent graph in a disk $D$ with two marked points such that $G\neq 0$.
    There exists a sequence composed only of $2$- and $4$-gon relations such that it reduces $G$ to an oriented simple arc $\gamma$ without increasing the number of connected components of intermediate graphs.
\end{lem}
\begin{proof}
    Let us assume that $G$ has at least one elliptic face and $G\neq 0$.
    We only have to attend to the $4$-gon relation because the $2$-gon relation does not change the number of connected components of a graph.
    Let us prove the claim by induction on the number $v(G)$ of vertices.
    A connected flat trivalent graph $G$ with $v(G)=2$ has to become a diagram in the left-hand side of the $2$-gon relation in \cref{A2skein}, and $G$ with $v(G)=3$ does not exist because a trivalent vertex is a sink or source.
    Let $G$ be a connected flat trivalent graph with $v(G)\geq 4$ and we assume that it has at least one internal $4$-gon $F$.
    \cref{eg:web-space}~(2) claims that we cannot describe a circle in $D$ which intersect with $G$ at a single edge.
    \cref{eg:web-space}~(5) claims that no circle intersects two incoming (resp. outgoing) edges and one outgoing (rep. incoming) edge of $G$.
    This fact requires that four edges incident to corners of $F$ connect to other parts of $G$ as (i) or (ii) in below;
	\[
            \text{(i)}~
            \tikz[baseline=-.6ex, scale=.2]{
			\draw (0,1.5) -- (-1.5,0) -- (0,-1.5) -- (1.5,0) -- cycle;
			\draw (0,1.5) -- (4,1.5);
			\draw (1.5,0) -- (4,0);
			\draw (0,-1.5) -- (4,-1.5);
			\draw (-1.5,0) -- (-4,0);
			\draw[fill=lightgray] (3,-2) rectangle (5,2);
			\draw[fill=lightgray] (-3,-2) rectangle (-5,2);
			\draw[thick] (-6,-2) -- (-6,2);
			\draw[thick] (6,-2) -- (6,2);
			\draw (-6,0) -- (-5,0);
			\draw (6,0) -- (5,0);
			\draw[fill=cyan] (-6,0) circle [radius=.2];
			\draw[fill=cyan] (6,0) circle [radius=.2];
			\node at (0,0) {\scriptsize $F$};
			\node at (-4,0) {\scriptsize $X_1$};
			\node at (4,0) {\scriptsize $X_2$};
		}\ \hspace{2em}
		\text{(ii)}~
		\tikz[baseline=-.6ex, scale=.2]{
			\draw (1,1) -- (1,-1) -- (-1,-1) -- (-1,1) -- cycle;
			\draw (1,1) -- (1,4);
			\draw (-1,1) -- (-1,4);
			\draw (-1,-1) -- (-4,-1);
			\draw (1,-1) -- (4,-1);
			\draw[fill=lightgray] (3,-3) rectangle (5,1);
			\draw[fill=lightgray] (-3,-3) rectangle (-5,1);
			\draw[fill=lightgray] (-2,2) rectangle (2,4);
			\draw[thick] (-6,-3) -- (-6,1);
			\draw[thick] (6,-3) -- (6,1);
			\draw (-6,-1) -- (-5,-1);
			\draw (6,-1) -- (5,-1);
			\draw[fill=cyan] (-6,-1) circle [radius=.2];
			\draw[fill=cyan] (6,-1) circle [radius=.2];
			\node at (0,0) {\scriptsize $F$};
			\node at (-4,-1) {\scriptsize $X_1$};
			\node at (4,-1) {\scriptsize $X_2$};
			\node at (0,3) {\scriptsize $X_3$};
		}\ 
	\]
	where subgraphs $X_1$, $X_2$, and $X_3$ of $G$ are connected.
	We apply the $4$-gon relation at $F$ in (i);
	\[
		\tikz[baseline=-.6ex, scale=.2]{
			\draw (0,1.5) -- (-1.5,0) -- (0,-1.5) -- (1.5,0) -- cycle;
			\draw (0,1.5) -- (4,1.5);
			\draw (1.5,0) -- (4,0);
			\draw (0,-1.5) -- (4,-1.5);
			\draw (-1.5,0) -- (-4,0);
			\draw[fill=lightgray] (3,-2) rectangle (5,2);
			\draw[fill=lightgray] (-3,-2) rectangle (-5,2);
			\draw[thick] (-6,-2) -- (-6,2);
			\draw[thick] (6,-2) -- (6,2);
			\draw (-6,0) -- (-5,0);
			\draw (6,0) -- (5,0);
			\draw[fill=cyan] (-6,0) circle [radius=.2];
			\draw[fill=cyan] (6,0) circle [radius=.2];
			\node at (0,0) {\scriptsize $F$};
			\node at (-4,0) {\scriptsize $X_1$};
			\node at (4,0) {\scriptsize $X_2$};
		}\ 
		=
		\tikz[baseline=-.6ex, scale=.2]{
			\draw (0,1.5) -- (-1.5,0); 
			\draw (0,-1.5) -- (1.5,0);
			\draw (0,1.5) -- (4,1.5);
			\draw (1.5,0) -- (4,0);
			\draw (0,-1.5) -- (4,-1.5);
			\draw (-1.5,0) -- (-4,0);
			\draw[fill=lightgray] (3,-2) rectangle (5,2);
			\draw[fill=lightgray] (-3,-2) rectangle (-5,2);
			\draw[thick] (-6,-2) -- (-6,2);
			\draw[thick] (6,-2) -- (6,2);
			\draw (-6,0) -- (-5,0);
			\draw (6,0) -- (5,0);
			\draw[fill=cyan] (-6,0) circle [radius=.2];
			\draw[fill=cyan] (6,0) circle [radius=.2];
			\node at (-4,0) {\scriptsize $X_1$};
			\node at (4,0) {\scriptsize $X_2$};
		}\ 
		+
		\tikz[baseline=-.6ex, scale=.2]{
			\draw (-1.5,0) -- (0,-1.5);
			\draw (0,1.5) -- (1.5,0);
			\draw (0,1.5) -- (4,1.5);
			\draw (1.5,0) -- (4,0);
			\draw (0,-1.5) -- (4,-1.5);
			\draw (-1.5,0) -- (-4,0);
			\draw[fill=lightgray] (3,-2) rectangle (5,2);
			\draw[fill=lightgray] (-3,-2) rectangle (-5,2);
			\draw[thick] (-6,-2) -- (-6,2);
			\draw[thick] (6,-2) -- (6,2);
			\draw (-6,0) -- (-5,0);
			\draw (6,0) -- (5,0);
			\draw[fill=cyan] (-6,0) circle [radius=.2];
			\draw[fill=cyan] (6,0) circle [radius=.2];
			\node at (-4,0) {\scriptsize $X_1$};
			\node at (4,0) {\scriptsize $X_2$};
		}\ .
	\]
    Graphs after applying the $4$-gon relation are divided into the left and right parts containing $X_1$ and $X_2$, respectively, by cutting along a vertical line in $D$. 
    Right and left subgraphs are considered $\mathfrak{sl}_3$-web in a disk with two marked points whose number of vertices is smaller than $v(G)$. 
    The right subgraph containing $X_2$ is connected due to \cref{eg:web-space}~(2).
    Hence these subgraphs satisfy the induction hypothesis.
    For case (ii), a sequence of the $2$- and $4$-gon relations changes $X_3$ into an arc without increasing the number of components because $v(X_3)<v(G)$. Then, one can obtain a graph consisting of $X_1$ and $X_2$ by applying the $2$-gon relation twice.
    The resulting graph also satisfies the induction hypothesis. 
\end{proof}

\begin{prop}\label{graphdeg}
    Let $G$ be a flat trivalent graph in a disk with no marked points, and we identify $G$ with its value in $\cR$ (see \cref{eg:web-space}). 
    Then,
    \begin{align*}
        \mdeg{G}\geq-\frac{v(G)}{4}-c(G),
    \end{align*}
    where $v(G)$ is the number of trivalent vertices of $G$, $c(G)$ the number of connected components of $G$.
    Moreover, the equality $\mdeg{G}=-c(G)$ holds when $v(G)=0$.
\end{prop}
\begin{proof}
    We first prove $\mdeg{G}=-c(G)$ when $v(G)=0$.
    If $G$ has no trivalent vertices, then it consists only of loop components.
    By using an innermost argument and the trivial loop relation in \cref{A2skein}, it is easy to see that $G=[3]^{c(G)}$.
    We obtain $\mdeg{G}=c(G)\mdeg{[3]}=-c(G)$ because $[3]=q+1+q^{-1}$.

    Let us consider when $G=\sqcup_{i=1}^{c(G)} G_i$ is a non-trivial flat trivalent graph with $v(G)>0$ where $G_i$'s are connected components of $G$.
    Choose a point $p_i$ on the outermost edge of $G_i$ and a small interval $I_{p_i}$ for each $i=1,2,\ldots, c(G)$.
    Then, one can take disks $\{D_i\}_i^{c(G)}$ with two marked points for all $i=1,2,\ldots c(G)$ such that $G_i\setminus \mathrm{int}(I_{p_i})\subset D_i$, $\partial I_{p_i}$ are identical with its marked points, and $D_i\cap D_j=\emptyset$ for $i\neq j$.
    Apply \cref{lem:reduction-seq} to $G_i\cap D_i$ for all $i\in\{1,2,\ldots, c(G)\}$ and we obtain a disjoint union $\Gamma\coloneqq \sqcup_{i=1}^{c(G)}\gamma_i$ of simple loops from $G$. 
    Each component $\gamma_i$ is obtained from $G_i$ by a sequence of $2$- and $4$-gon relations preserving the number of connected components.
    Let $G=G'+G''$ be a $4$-gon relation appearing in the above sequence, and we can assume that $G'\neq 0$ and $\mdeg{G'}\leq\mdeg{G''}$ without loss of generality.
    Then,
	\begin{itemize}
		\item $\mdeg{G}\geq\min\{\mdeg{G'},\mdeg{G''}\}=\mdeg{G'}$,
		\item $v(G)=v(G')+4$, and
		\item $c(G)=c(G')$.
	\end{itemize}
    Thus, 
    \[
	\mdeg{G}+\frac{1}{4}v(G)+c(G)\geq\mdeg{G'}+\frac{1}{4}v(G')+c(G')+1.
    \]
    Suppose instead that $G'$ is obtained by a $2$-gon relation, that is, $G=[2]G'$.
    Then, 
    \begin{itemize}
	\item $\mdeg{G}=\mdeg{G'}-\frac{1}{2}$,
	\item $v(G)=v(G')+2$, and
	\item $c(G)=c(G')$.
    \end{itemize}
    Thus,
    \[
	\mdeg{G}+\frac{1}{4}v(G)+c(G)=\mdeg{G'}+\frac{1}{4}v(G')+c(G').
    \]
    As mentioned above, we can choose a reduction sequence from $G$ to $\Gamma$ so that flat trivalent graphs in this sequence satisfy the above inequality for the minimum degree.
    We remark that $\mdeg{\Gamma}=-c(\Gamma)$ because $v(\Gamma)=0$.
    Hence, $G$ and $\Gamma$ should satisfy
    \begin{align*}
        \mdeg{G}+\frac{1}{4}v(G)+c(G)\geq\mdeg{\Gamma}+\frac{1}{4}v(\Gamma)+c(\Gamma)
        =0.
    \end{align*}
\end{proof}
	
Next, we give a lower bound of the minimum degree of a flat trivalent graph with $\mathfrak{sl}_3$-clasps.
Let us consider the minimum degree of coefficients appearing in expansion formulas of $\mathfrak{sl}_3$-clasps.

\begin{lem}[The single clasp expansion formula~{\cite[Proposition~3.1]{Kim07}}]\label{singleexp}
For any positive integer $m$,
\[
	\mathrm{JW}_{{-}^{m}}
	=
	\,\tikz[baseline=-.6ex, scale=.1]{
		\draw[->-=.3,->-=.9] (0,-3) -- (0,4);
		\draw[fill=white] (-3,0) rectangle (3,1);
		\node at (0,3) [left]{${\scriptstyle m}$};
		}\,
	=\sum_{j=0}^{m-1}f_{j}^{(m)}(q)
	\mathord{\,\tikz[baseline=-.6ex, scale=.1, yshift=-2cm]{
		\draw[->-=.2, ->-=.9, rounded corners] (-4,-3) -- (-4,3) -- (2,3) -- (2,8);
		\draw[->-=.8] (-1,0) -- (-1,8);
		\draw[->-=.5] (4,0) -- (4,8);
		\draw[->-=.5] (1,-3) -- (1,0);
		\draw[fill=white] (-3,2) rectangle (1,4);
		\draw (-3,4) -- (1,2);
		\draw[fill=white] (-3,0) rectangle (5,1);
		\node at (-1,8) [above]{${\scriptstyle j}$};
		\node at (2,8) [above]{${\scriptstyle 1}$};
		\node at (3,8) [above right]{${\scriptstyle m-j-1}$};
		\node at (1,-3) [below]{${\scriptstyle m-1}$};
		\node at (-4,-3) [below]{${\scriptstyle 1}$};
		}
	\,},
\]
where $f_{j}^{(m)}(q)\coloneqq (-1)^{j}\frac{[m-j]}{[m]}$.
\end{lem}

One can obtain the following lemma from the single clasp expansion formula and induction on $m$.

\begin{lem}\label{onerowexpansion}
    The one-row colored $\mathfrak{sl}_3$-clasp has an expansion
    \[
        \mathrm{JW}_{{-}^{m}}=\sum_{M}f_{M}(q)M
    \]
    with $\mdeg{f_{M}(q)}=v(M)/4$ where the sum runs over finitely many flat trivalent graphs $M$, and $v(M)$ is the number of trivalent vertices in $M$.
    We remark that $M$ may contain $4$-gons or $2$-gons.
\end{lem}
\begin{proof}
    It is obvious that the claim is true for $m=1,2$.
    We prove it by induction on $m$.
    A flat trivalent graph in the right-hand side of the single clasp expansion of $\mathrm{JW}_{{-}^m}$ has a stair-step web with $2j$ vertices and $\mathrm{JW}_{{-}^{m-1}}$.
    We know that $\mdeg{f_j^{(m)}(q)}=j/2$ by \cref{eg:deg}.
    \cref{singleexp} and the induction hypothesis derive an expansion;
    \[
        \mathrm{JW}_{{-}^{m}}
        =\sum_{j=0}^{m-1}\sum_{M}f_j^{(m)}(q)f_{M}(q)
        \mathord{\,\tikz[baseline=-.6ex, scale=.1, yshift=-2cm]{
            \draw[->-=.2, ->-=.9, rounded corners] (-4,-3) -- (-4,3) -- (2,3) -- (2,8);
            \draw[->-=.8] (-1,0) -- (-1,8);
            \draw[->-=.5] (4,0) -- (4,8);
            \draw (1,-3) -- (1,0);
            \draw[fill=white] (-3,2) rectangle (1,4);
            \draw (-3,4) -- (1,2);
            \draw[fill=lightgray] (-2.5,-2) rectangle (4.5,1);
            \node at (-1,8) [above]{${\scriptstyle j}$};
            \node at (2,8) [above]{${\scriptstyle 1}$};
            \node at (3,8) [above right]{${\scriptstyle m-j-1}$};
            \node at (1,-3) [below]{${\scriptstyle m-1}$};
            \node at (-4,-3) [below]{${\scriptstyle 1}$};
            \node at (1,-.5) {${\scriptstyle M}$};
            }
            \,}
    \]
    where $\mdeg{f_M(q)}=v(M)/4$.
    The flat trivalent graph in the right-hand side has $2j+v(M)$ vertices, and $\mdeg{f_j^{(m)}(q)f_{M}(q)}=\mdeg{f_j^{(m)}(q)}+\mdeg{f_{M}(q)}=(2j+v(M))/4$ by \cref{lem:deg}.
    This expansion satisfies the condition of our claim.
\end{proof}

\begin{rmk}\label{rmk:I-web}
    The proof of \cref{onerowexpansion} can be used to show that $M$ constructed by composing $I_j$s where
	\[
		I_j\coloneqq \mathord{\,\tikz[baseline=-.6ex, scale=.1]{
			\draw[-<-=.5] (0,-2) -- (0,2);
			\draw[->-=.5] (-2,-5) to[out=north, in=west] (0,-2);
			\draw[->-=.5] (2,-5) to[out=north, in=east] (0,-2);
			\draw[-<-=.5] (-2,5) to[out=south, in=west] (0,2);
			\draw[-<-=.5] (2,5) to[out=south, in=east] (0,2);
			\draw[->-=.5] (-4,-5) -- (-4,5);
			\draw[->-=.5] (-10,-5) -- (-10,5);
			\node at (-7,0) {${\scriptstyle \cdots}$};
			\draw[->-=.5] (4,-5) -- (4,5);
			\draw[->-=.5] (10,-5) -- (10,5);
			\node at (7,0) {${\scriptstyle \cdots}$};
			\node at (-2,-5) [below]{${\scriptstyle j}$};
			\node at (2,-5) [below]{${\scriptstyle j+1}$}
			}\,
        \in \TL{{-}^{m}}{{-}^{m}}
		}
        \]
    for $j=1,2,\dots,m-1$.
    Labels $j$ and $j+1$ in the bottom mean the $j$- and $(j+1)$-th marked points, respectively.
\end{rmk}

\begin{lem}\label{tworowexpansion}
The two-row colored $\mathfrak{sl}_3$-clasp has the following expansion:
\[
	\mathrm{JW}_{{-}^{m}{+}^{n}}
	=\sum_{t=0}^{\min\{m,n\}}\sum_{M_1,M_2,M_3,M_4}f_{(m,n;t)}(M_1,M_2,M_3,M_4;q)\BNet{(m,n;t)}{M_1}{M_2}{M_3}{M_4}
\]
with $\mdeg{f_{(m,n;t)}(M_1,M_2,M_3,M_4;q)}=\frac{t(t+1)}{2}+\sum_{i=1}^{4}\frac{1}{4}v\left(M_i\right)$ where
\[
    \BNet{(m,n;t)}{M_1}{M_2}{M_3}{M_4}
    \coloneqq \mathord{
	\ \tikz[baseline=-.6ex, scale=.1]{
            \draw[->-=.7] (-4,7) -- (-4,10);
            \draw[-<-=.7] (4,7) -- (4,10);
            \draw[-<-=.7, yscale=-1] (-4,7) -- (-4,10);
            \draw[->-=.7, yscale=-1] (4,7) -- (4,10);
            \draw[->-=.5] (-5,-3) -- (-5,3);
            \draw[-<-=.5] (5,-3) -- (5,3);
            \draw[-<-=.5] (-3,3) to[out=south, in=south] (3,3);
            \draw[->-=.5, yscale=-1] (-3,3) to[out=south, in=south] (3,3);
            \draw[fill=lightgray] (-7,3) rectangle (-2,7);
            \draw[fill=lightgray, xscale=-1] (-7,3) rectangle (-2,7);
            \draw[fill=lightgray, yscale=-1] (-7,3) rectangle (-2,7);
            \draw[fill=lightgray, xscale=-1, yscale=-1] (-7,3) rectangle (-2,7);
            \node at (-5,0) [left]{${\scriptstyle m-t}$};
            \node at (5,0) [right]{${\scriptstyle n-t}$};
            \node at (0,4) {${\scriptstyle t}$};
            \node at (0,-4) {${\scriptstyle t}$};
            \node at (-4,8) [left]{${\scriptstyle m}$};
            \node at (4,8) [right]{${\scriptstyle n}$};
            \node at (-4,-8) [left]{${\scriptstyle m}$};
            \node at (4,-8) [right]{${\scriptstyle n}$};
            \node at (-4,5) {${\scriptstyle M_{4}}$};
            \node at (-4,-5) {${\scriptstyle M_{1}}$};
            \node at (4,5) {${\scriptstyle M_{3}}$};
            \node at (4,-5) {${\scriptstyle M_{2}}$};
	}\ 
    }.
\] 
\end{lem}
\begin{proof}
    Apply \cref{onerowexpansion} to each one-row colored $\mathfrak{sl}_3$-clasp in the right-hand side of \cref{tworowclasp}.
    Then, we obtain an expansion as in the statement such that
    \[
        f_{(m,n;t)}(M_1,M_2,M_3,M_4;q)\coloneqq (-1)^{t}\frac{{m \brack t}{n \brack t}}{{m+n+1 \brack t}}f_{M_{1}}(q)f_{M_{2}}(q)f_{M_{3}}(q)f_{M_{4}}(q).
    \]
    One can calculate the minimum degree as follows:
    \begin{align*}
        \mdeg{f_{(m,n;t)}(M_1,M_2,M_3,M_4;q)}
        &=\mdeg{(-1)^t\frac{{m \brack t}{n \brack t}}{{m+n+1 \brack t}}}+\sum_{i=1}^4\mdeg{f_{M_i}(q)}\\
        &=\frac{t(t+1)}{2}+\sum_{i=1}^4\frac{v(M_i)}{4}.
    \end{align*}
\end{proof}

We introduce a notation of a planar algebra specializing in our situation because it is useful for writing $\mathfrak{sl}_3$-webs in the form of an equation.
Let $D$ be a disk, and $D_i$ $(i=1,2,\ldots,k)$ specified disjoint $k$ rectangles in $D\setminus\partial D$. 
$D_i$ is homeomorphic to $[0,1]\times[0,1]$ and it has a base point at $(0,0)$ and marked points $P_i=P_i^{(0)}\sqcup P_i^{(1)}$.
The set $P^{(j)}_{i}$ of marked points lies in the edge of $D_i$ corresponding to $[0,1]\times\{j\}$.
We have a $k$-holed disk $D(k)\coloneqq D\setminus \cup_{i=1}^{k}\mathrm{int}(D_i)$ with marked points $\cup_{i=1}^{k} P_i$.
A small disk $D_i$ share $P_i$ with $D(k)$ for $i=1,2,\ldots,k$, see \cref{fig:holeddisk}.
Let a sequence of signs $\bs_i^{(0)}$ (resp.~$\bs_i^{(1)}$) be an assignment of signs to a set of marked points $P_i^{(0)}$ (resp.~$P_i^{(1)}$) of $D(k)$ for each $i=1,2,\ldots,k$.
Then, we consider the following $\mathfrak{sl}_3$-web spaces:
\begin{itemize}
	\item $\W(\cup_{i=1}^k\bs_i;D(k))$ where $\bs_i=\bs_i^{(0)}\cup\bs_i^{(1)}$,
	\item $\TL{\bar{\bs}_i^{(0)}}{\bs_i^{(1)}}=\W(\bar{\bs}_i^{(0)}\cup\bar{\bs}_i^{(1)};D_i)$ for $i=1,2,\ldots,k$.
\end{itemize}
As I mentioned above, $D(k)$ and $D_i$ share the set of marked points $P_i=P_i^{(0)}\sqcup P_i^{(1)}$.
Then, the sequence of sign $\bs_i^{(j)}$ of $P_i^{(j)}$ in $D(k)$ consistent with $\bar{\bs}_i^{(j)}$ of $P_i^{(j)}$ in $D_i$ for $j=1,2$.
For example, an edge terminating at $p\in P_i$ in $D(k)$ can be composed with an edge starting from $p$ in $D_i$. Thus, the sign of $p$ in $D(k)$ and $D_i$ are different.
For a tangled trivalent graph $\mathsf{G}\in\G(\cup_{i=1}^k\bs_i;D(k))$, 
a linear map
\[
	\mathsf{G}\colon\bigotimes_{i=1}^{k}\TL{\bar{\bs}_i^{(0)}}{\bs_i^{(1)}}\to\W(\emptyset,D)\cong\mathbb{Z}((q^{\frac{1}{6}}))
\]
is induced by a map $D(k)\sqcup \left( D_1\sqcup D_2\sqcup\cdots\sqcup D_k\right)\to D$.
This map composes $\mathfrak{sl}_3$-webs in $D_i$ ($i=1,2,\ldots,k$) with $\mathsf{G}$ in $D(k)$.

In this paper, we only consider a segregated sign sequence $\bs_i^{(0)}=\epsilon^{m_i}\bar{\epsilon}^{n_i}$ and $\bs_i^{(1)}=\bar{\bs}_i^{(0)}=\bar{\epsilon}^{m_i}\epsilon^{n_i}$ where $\epsilon$ is $+$ or $-$ and $m_i,n_i\in\mathbb{Z}_{{}\geq 0}$ satisfy $m_i+n_i=\#P_i^{(0)}=\#P_i^{(1)}$.
The \emph{identity web} denoted by $\ID_{\bs_i^{(1)}}$ in $\TL{\bar{\bs}_i^{(0)}}{\bs_i^{(1)}}$ is $(m_{i}+n_{i})$ parallel strands in $D_i$. 
The identity web $\ID_{\bs_i^{(1)}}$ and the $\mathfrak{sl}_3$-clasp $\mathrm{JW}_{\bs_i^{(1)}}$ in $D_i$ are simply denoted by $\ID_{D_i}$ and $\mathrm{JW}_{D_i}$, respectively.
\begin{figure}
	\begin{tikzpicture}[scale=.1]
		\draw (0,0) circle [radius=20];
		\begin{scope}[yshift=10cm]
			\draw[fill=lightgray!50] (-5,-3) rectangle (5,3);
			\node at (0,0) {\small $1$};
			\node at (-5,-3) {$\ast$};
			\draw[fill] (-2,-3) circle [radius=.5];
			\draw[fill=white] (2,-3) circle [radius=.5];
			\draw[fill=white] (-2,3) circle [radius=.5];
			\draw[fill] (2,3) circle [radius=.5];
		\end{scope}
		\begin{scope}[xshift=10cm, yshift=-3cm]
			\draw[fill=lightgray!50] (-5,-3) rectangle (5,3);
			\node at (0,0) {\small $2$};
			\node at (-5,-3) {$\ast$};
			\draw[fill=white] (-2,-3) circle [radius=.5];
			\draw[fill=white] (2,-3) circle [radius=.5];
			\draw[fill] (-2,3) circle [radius=.5];
			\draw[fill] (2,3) circle [radius=.5];
		\end{scope}
		\begin{scope}[yshift=-5cm, xshift=-8cm]
			\draw[fill=lightgray!50] (-7,-4) rectangle (7,4);
			\node at (0,0) {\small $3$};
			\node at (-7,-4) {$\ast$};
			\draw[fill=white] (-4,-4) circle [radius=.5];
			\draw[fill] (0,-4) circle [radius=.5];
			\draw[fill] (4,-4) circle [radius=.5];
			\draw[fill] (-4,4) circle [radius=.5];
			\draw[fill=white] (0,4) circle [radius=.5];
			\draw[fill=white] (4,4) circle [radius=.5];
		\end{scope}
		\node at (-12,5) {\small $D(3)$};
	\end{tikzpicture}
\caption{It is a $k$-holed disk $D(k)$ with $k=3$. A shaded rectangle labeled by $i$ is $D_i$. Marked points with sign ${+}$ (resp.~${-}$) are described as black (resp.~white) dots. In this case, $\bs_{1}^{(0)}=\bar{\bs}_{1}^{(1)}={+}{-}$, $\bs_{2}^{(0)}=\bar{\bs}_{2}^{(1)}={-}{-}$, and $\bs_{3}^{(0)}=\bar{\bs}_{3}^{(1)}={-}{+}{+}$. A tangled trivalent graph $\mathsf{G}$ defines $\TL{{-}{+}}{{-}{+}}\otimes\TL{{+}{+}}{{+}{+}}\otimes\TL{{+}{-}{-}}{{+}{-}{-}}\to\cR$. }
\label{fig:holeddisk}
\end{figure}

\begin{dfn}\label{dfn:adequate}
    Set $\bs_i^{(0)}=\bar{\bs}_i^{(1)}=\epsilon^{m_i}\bar{\epsilon}^{n_i}$ for all $i=1,2,\ldots,k$.
    $\mathsf{G}\in\G(\cup_{i=1}^k\bs_i;D(k))$ is \emph{adequate} if 
    \begin{itemize}
        \item $\mathsf{G}$ is a disjoint union of oriented simple arcs, and
        \item For every $j=1,2,\ldots,k$, any pair of strands in $\ID_{D_j}$ belongs to different connected components of the graph $\mathsf{G}(\otimes_{i=1}^{k}\ID_{D_i})$  which composed of oriented simple loops.
    \end{itemize}
    See \cref{adequateweb}.
    We also call the clasped $\mathfrak{sl}_3$-web $\mathsf{G}(\otimes_{i=1}^{k}\mathrm{JW}_{D_i})$ is \emph{adequate} when $\mathsf{G}$ is adequate.
\end{dfn}
	
	\begin{eg}\label{adequateweb}
	The left $\mathfrak{sl}_3$-web is adequate, and the right is not adequate because of the red arc.
	\[
	\mathord{\ \tikz[baseline=-.6ex, scale=.1]
	  {
		\draw[->-=.5, rounded corners] (-12,1) -- (-14,1) -- (-14,12) -- (-4,12) -- (-4,10);
		\draw[->-=.5, rounded corners] (-4,10) -- (-4,1) -- (-10,1);
		\draw[->-=.5, rounded corners] (-12,-1) -- (-14,-1) -- (-14,-12) -- (-4,-12) -- (-4,-10);
		\draw[->-=.5, rounded corners] (-4,-10) -- (-4,-1) -- (-10,-1);
		\draw[->-=.5, rounded corners] (-2,-8) -- (-2,-1) -- (10,-1);
		\draw[-<-=.5, rounded corners] (-2,-10) -- (-2,-13) -- (14,-13) -- (14,-1) -- (12,-1);
		\draw[-<-=.5, rounded corners] (6,-4) -- (4,-4) -- (4,-10) -- (6,-10);
		\draw[->-=.5, rounded corners] (8,-4) -- (10,-4) -- (10,-10) -- (8,-10);
		\draw[-<-=.5, rounded corners] (-2,8) -- (-2,1) -- (10,1);
		\draw[->-=.5, rounded corners] (-2,10) -- (-2,12) -- (14,12) -- (14,1) -- (10,1);
		\draw[fill=lightgray!50] (-12,3) rectangle (-10,-3);
		\draw[fill=lightgray!50] (-6,-10) rectangle (0,-8);
		\draw[fill=lightgray!50] (-6,8) rectangle (0,10);
		\draw[fill=lightgray!50] (12,-3) rectangle (10,3);
		\draw[fill=lightgray!50] (8,0) rectangle (6,-6);
		\draw[fill=lightgray!50] (8,-8) rectangle (6,-14);
		\node at (-12,3) {\small $\ast$};
		\node at (-6,-10) {\small $\ast$};
		\node at (-6,8) {\small $\ast$};
		\node at (12,-3) {\small $\ast$};
		\node at (6,0) {\small $\ast$};
		\node at (6,-8) {\small $\ast$};
		}
	}
	\quad\quad
	\mathord{\ \tikz[baseline=-.6ex, scale=.1]
	  {
	  \draw[->-=.5, rounded corners, very thick, red] (-12,1) -- (-12,12) -- (-4,12) -- (-4,10);
		\draw[->-=.5, rounded corners, very thick, red] (-4,10) -- (-4,4) -- (-10,4) -- (-10,1);
	  \draw[-<-=.5, rounded corners, very thick, red] (-12,-1) -- (-12,-12) -- (-4,-12) -- (-4,-10);
		\draw[-<-=.5, rounded corners, very thick, red] (-4,-10) -- (-4,-4) -- (-10,-4) -- (-10,-1);
		\draw[->-=.5, rounded corners] (-2,-8) -- (-2,-1) -- (10,-1);
		\draw[-<-=.5, rounded corners] (-2,-10) -- (-2,-13) -- (14,-13) -- (14,-1) -- (12,-1);
		\draw[-<-=.5, rounded corners] (6,-4) -- (4,-4) -- (4,-10) -- (6,-10);
		\draw[->-=.5, rounded corners] (8,-4) -- (10,-4) -- (10,-10) -- (8,-10);
		\draw[-<-=.5, rounded corners] (-2,8) -- (-2,1) -- (10,1);
		\draw[->-=.5, rounded corners] (-2,10) -- (-2,12) -- (14,12) -- (14,1) -- (10,1);
		\draw[fill=lightgray!50] (-14,-1) rectangle (-8,1);
		\draw[fill=lightgray!50] (-6,-10) rectangle (0,-8);
		\draw[fill=lightgray!50] (-6,10) rectangle (0,8);
		\draw[fill=lightgray!50] (12,-3) rectangle (10,3);
		\draw[fill=lightgray!50] (8,0) rectangle (6,-6);
		\draw[fill=lightgray!50] (8,-8) rectangle (6,-14);
		\node at (-14,-1) {\small $\ast$};
		\node at (-6,-10) {\small $\ast$};
		\node at (-6,8) {\small $\ast$};
		\node at (12,-3) {\small $\ast$};
		\node at (6,0) {\small $\ast$};
		\node at (6,-8) {\small $\ast$};
		}
	}
	\]
	\end{eg}
	

\begin{prop}\label{claspgraphdeg}
    Let $D(k)=D\setminus \sqcup_{i=1}^k \mathrm{int}(D_i)$ be a $k$ holed disk with signed marked point $\bs_i^{(0)}=\bar{\bs}_i^{(1)}=\epsilon^{m_i}\bar{\epsilon}^{n_i}$ on the $i$-th boundary component for $i=1,2,\dots,k$.
    For any flat trivalent graph $\mathsf{G}$ in $\G(\cup_{i=1}^k\bs_i;D(k))$,
    \[
        \mdeg{\mathsf{G}(\otimes_{i=1}^{k}\mathrm{JW}_{D_i})}\geq -\frac{v(\mathsf{G})}{4}-c\left(\mathsf{{G}}(\otimes_{i=1}^{k}{\ID}_{D_i})\right).
    \]
    Particularly, $\mdeg{\mathsf{G}(\otimes_{i=1}^{k}\mathrm{JW}_{D_i})}\geq\mdeg{\mathsf{{G}}(\otimes_{i=1}^{k}{\ID}_{D_i})}$ holds when $\mathsf{G}$ has no trivalent vertices due to \cref{graphdeg}.
    Moreover, the equality $\mdeg{\mathsf{G}(\otimes_{i=1}^{k}\mathrm{JW}_{D_i})}=\mdeg{\mathsf{{G}}(\otimes_{i=1}^{k}{\ID}_{D_i})}$ holds when $\mathsf{G}$ is adequate.
\end{prop}
\begin{proof}
\cref{tworowexpansion} expands all $\mathfrak{sl}_3$-clasps in $D_i$ ($i=1,2,\ldots,k$) as below.
\begin{align*}
    &\mathsf{G}(\otimes_{i=1}^{k}\mathrm{JW}_{D_i})
    =\sum_{t_{1}=0}^{\min\{m_{1},n_{1}\}}\!\cdots\hspace{-1em}\sum_{t_{k}=0}^{\min\{m_{k},n_{k}\}}\sum_{M^{(1)}}\cdots\sum_{M^{(k)}}
    \Big(\prod_{i=1}^{k}f_{t_{i}}(M^{(i)};q)\Big)
    \mathsf{G}(\otimes_{i=1}^{k}{\uparrow\!M^{(i)}\!\downarrow}_{t_{i}}),
\end{align*}
where
\begin{itemize}
    \item $\sum_{M^{(i)}}$ means a summation over $M_{1}^{(i)},M_{2}^{(i)},M_{3}^{(i)},M_{4}^{(i)}$,
    \item ${\uparrow\!\!M^{(i)}\!\!\downarrow}_{t_{i}}$ is the $\mathfrak{sl}_3$-web $\BNet{(m_i,n_i;t_i)}{M_1^{(i)}}{M_2^{(i)}}{M_3^{(i)}}{M_4^{(i)}}$ defined in \cref{tworowexpansion}, and
    \item $f_{t_{i}}(M^{(i)};q)\coloneqq f_{(m_i,n_i;t_{i})}(M_1^{(i)},M_2^{(i)},M_3^{(i)},M_4^{(i)};q)$ with $\mdeg{f_{t_{i}}(M^{(i)};q)}=\frac{t_i(t_i+1)}{2}+\sum_{j=1}^4\frac{1}{4}v(M^{(i)}_j)$.
\end{itemize}
We remark that
$v(\mathsf{G}(\otimes_{i=1}^{k}{\uparrow\!M^{(i)}\!\downarrow}_{t_{i}}))=\sum_{i=1}^{k}\sum_{j=1}^4v(M_j^{(i)})+v(\mathsf{G})$ by definition. 
\begin{align}
    &\mdeg{\Big(\prod_{i=1}^{k}f_{t_{i}}(M^{(i)};q)\Big)
    \mathsf{G}(\otimes_{i=1}^{k}{\uparrow\!M^{(i)}\!\downarrow}_{t_{i}})}\label{eq:clasped-graph-deg1}\\
    &\quad=\sum_{i=1}^k\frac{t_i(t_i+1)}{2}+\sum_{i=1}^k\sum_{j=1}^4\frac{v(M_j^{(i)})}{4}+\mdeg{\mathsf{G}(\otimes_{i=1}^{k}{\uparrow\!M^{(i)}\!\downarrow}_{t_{i}})}.\notag
\end{align}
\cref{graphdeg} gives a lower bound of $\mdeg{\mathsf{G}(\otimes_{i=1}^{k}{\uparrow\!M^{(i)}\!\downarrow}_{t_{i}})}$ using the number of vertices and connected components;
\begin{align*}
    \mdeg{\mathsf{G}(\otimes_{i=1}^{k}{\uparrow\!M^{(i)}\!\downarrow}_{t_{i}})}
    &\geq -\frac{1}{4}\left(\sum_{i=1}^{k}\sum_{j=1}^4 v(M_j^{(i)})+v(\mathsf{G})\right)-c(\mathsf{G}(\otimes_{i=1}^{k}{\uparrow\!M^{(i)}\!\downarrow}_{t_{i}})).
\end{align*}

Let $\mathsf{G}(\otimes_{i=1}^{k}{\uparrow\!\ID^{(i)}\!\downarrow}_{t_{i}})$ be an $\mathfrak{sl}_3$-web such that $\mathsf{G}(\otimes_{i=1}^{k}{\uparrow\!\ID^{(i)}\!\downarrow}_{t_{i}})\cap D(k)=\mathsf{G}$ and 
$\mathsf{G}(\otimes_{i=1}^{k}{\uparrow\!\ID^{(i)}\!\downarrow}_{t_{i}})\cap D_i={\uparrow\!\!\ID^{(i)}\!\!\downarrow}_{t_{i}}$ where 
$
{\uparrow\!\!\ID^{(i)}\!\!\downarrow}_{t_{i}}
\coloneqq 
\mathord{
	\ \tikz[baseline=-.6ex, scale=.1]{
	\draw[->-=.5] (-5,-5) -- (-5,5);
	\draw[-<-=.5] (5,-5) -- (5,5);
	\draw[-<-=.5] (-3,5) to[out=south, in=south] (3,5);
	\draw[->-=.5, yscale=-1] (-3,5) to[out=south, in=south] (3,5);
	\node at (-5,0) [left]{${\scriptstyle m_{i}-t_{i}}$};
	\node at (5,0) [right]{${\scriptstyle n_{i}-t_{i}}$};
	\node at (0,0) {${\scriptstyle t_{i}}$};
	}\ 
}
$.
In other words, $\mathsf{G}(\otimes_{i=1}^{k}{\uparrow\!\ID^{(i)}\!\downarrow}_{t_{i}})$ is obtained by replacing all $M_{j}^{(i)}$ in $\mathsf{G}(\otimes_{i=1}^{k}{\uparrow\!M^{(i)}\!\downarrow}_{t_{i}})$ with identity webs.

\cref{rmk:I-web} says that $\mathsf{G}(\otimes_{i=1}^{k}{\uparrow\!M^{(i)}\!\downarrow}_{t_{i}})$ is obtained from $\mathsf{G}(\otimes_{i=1}^{k}{\uparrow\!\ID^{(i)}\!\downarrow}_{t_{i}})$ by a sequence of \emph{zip cobordisms} which replace parallel strands with $I_j$'s.
Then, $c(\mathsf{G}(\otimes_{i=1}^{k}{\uparrow\!M^{(i)}\!\downarrow}_{t_{i}}))\leq c(\mathsf{G}(\otimes_{i=1}^{k}{\uparrow\!\ID^{(i)}\!\downarrow}_{t_{i}}))$ holds because a zip cobordism reduces the number of connected components.
Combining these inequalities with \cref{eq:clasped-graph-deg1}, we obtain
\begin{align}
    &\mdeg{\Big(\prod_{i=1}^{k}f_{t_{i}}(M^{(i)};q)\Big)\mathsf{G}(\otimes_{i=1}^{k}{\uparrow\!M^{(i)}\!\downarrow}_{t_{i}})}
    \geq
    \sum_{i=1}^k\frac{t_i(t_i+1)}{2}-\frac{v(\mathsf{G})}{4}-c(\mathsf{G}(\otimes_{i=1}^{k}{\uparrow\!\ID^{(i)}\!\downarrow}_{t_{i}})).\label{eq:clasped-graph-deg2}
\end{align}
We remark that one can make a similar argument when $\mathrm{JW}_{D_i}$s have one-row colored $\mathfrak{sl}_3$-clasps by using \cref{onerowexpansion}.
If the $i$-th disk $D_i$ has a one-row colored $\mathfrak{sl}_3$-clasp, then we read ${\uparrow\!M^{(i)}\!\downarrow}_{t_{i}}$ as $M^{(i)}$, $f_{t_{i}}(M^{(i)};q)$ as $f_{M^{(i)}}(q)$, and replace $\sum_{j=1}^{4}v(M_{j}^{(i)})/4$ with $v(M^{(i)})/4$.

Finally, we observe how the right-hand side of \cref{eq:clasped-graph-deg2} changes by a single \emph{orientable saddle cobordism}, which transforms $t_i$ to $t_i+1$.
One can see that the single orientable saddle cobordism changes $c(\mathsf{G}(\otimes_{i=1}^{k}{\uparrow\!\ID^{(i)}\!\downarrow}_{t_{i}}))$ by $\pm 1$ by considering the orientation of strands in ${\uparrow\!\!\ID^{(i)}\!\!\downarrow}_{t_{i}}$.
Hence, the right-hand side of \cref{eq:clasped-graph-deg2} is a monotonically increasing function on $0 \leq t_{i}\leq \min\{m_{i},n_{i}\}$.
Consequently,
\begin{align*}
    \mdeg{\Big(\prod_{i=1}^{k}f_{t_{i}}(M^{(i)};q)\Big)\mathsf{G}(\otimes_{i=1}^{k}{\uparrow\!M^{(i)}\!\downarrow}_{t_{i}})}
    &\geq -\frac{v(\mathsf{G})}{4}-c(\mathsf{G}(\otimes_{i=1}^{k}{\uparrow\!\ID^{(i)}\!\downarrow}_{0}))\\
    &= -\frac{v(\mathsf{G})}{4}-c(\mathsf{G}(\otimes_{i=1}^{k}\ID_{D_i})),
\end{align*}
because ${\uparrow\!\ID^{(i)}\!\downarrow}_{0}=\ID_{D_i}$ by definition.

When $v(\mathsf{G})=0$, the right-hand side becomes $\mdeg{\mathsf{{G}}(\otimes_{i=1}^{k}{\ID}_{D_i})}$ by \cref{graphdeg}.
Moreover, if $\mathsf{G}$ is adequate,
$c(\mathsf{G}(\otimes_{i=1}^{k}{\uparrow\!\ID^{(i)}\!\downarrow}_{t_i}))$ strictly decreases by the orientable saddle cobordism changing $t_i$ from $0$ to $1$ for some $i$.
Hence, we obtain $\mdeg{\mathsf{G}(\otimes_{i=1}^{k}\mathrm{JW}_{D_i})}=\mdeg{\mathsf{{G}}(\otimes_{i=1}^{k}{\ID}_{D_i})}$.
We remark that one can make a similar argument when all $\mathrm{JW}_{D_i}$ are one-row colored $\mathfrak{sl}_3$-clasps and $G$ is adequate.
\end{proof}
\section{Zero stability of the one-row colored $\mathfrak{sl}_3$-Jones polynomial}\label{sect:stability}

We show the zero stability of the one-row colored $\mathfrak{sl}_3$-tail for a certain class of $B$-adequate oriented links.

\begin{dfn}\label{dfn:B-adequate}
    Let $D$ be a disk equipped with twist regions $\sqcup_{i=1}^{k}D_i$ in $\mathrm{int}(D)$ such that all $D_i$s are isomorphic to $[0,1]\times [0,1]$ with a base point at $(0,0)$ and an assignment $\bm{l}\colon \{D_i\}_{i=1}^k\to 2\mathbb{Z}_{>0}$.
    An \emph{anti-parallel $B$-adequate link} is an oriented link represented by an oriented link diagram $L$ in $D$ satisfying the following condition:
    \begin{itemize}
        \item $L\cap D(k)$ is an adequate graph $\mathsf{G}$ where $D(k)\coloneqq D\setminus\sqcup_{i=1}^{k}\mathrm{int}(D_i)$, and
        \item $L\cap D_i$ is a twist region $R_{l_i}$ with negative $l_i\coloneqq \bm{l}(D_i)$ half twists of antiparallel strands for each $i$, see \cref{fig:twist-region}.
    \end{itemize}
\end{dfn}

\begin{figure}
    \[
		R_{l_{i}}=
		\tikz[baseline=-.6ex, scale=.1]{
            \draw[overarc] (-5,-7) to[out=north, in=south] (5,-2);
		    \draw[overarc] (5,-7) to[out=north, in=south] (-5,-2);
		    \draw[overarc] (-5,2) to[out=north, in=south] (5,7);
		    \draw[overarc] (5,2) to[out=north, in=south] (-5,7);
            \draw[-<-=.2] (-5,7) -- (-5,9);
            \draw[->-=.2] (5,7) -- (5,9);
            \draw[-<-=.8] (-5,-9) -- (-5,-7);
            \draw[->-=.8] (5,-9) -- (5,-7);
            \draw (-9,9) rectangle (9,-9);
		    \node at (-9,-9) {$\ast$};
		    \node at (0,0) [rotate=90]{${\cdots}$};
        }\ \text{ or }
		\tikz[baseline=-.6ex, scale=.1]{
            \draw[overarc] (-5,-7) to[out=north, in=south] (5,-2);
		    \draw[overarc] (5,-7) to[out=north, in=south] (-5,-2);
		    \draw[overarc] (-5,2) to[out=north, in=south] (5,7);
		    \draw[overarc] (5,2) to[out=north, in=south] (-5,7);
            \draw[->-=.2] (-5,7) -- (-5,9);
            \draw[-<-=.2] (5,7) -- (5,9);
            \draw[->-=.8] (-5,-9) -- (-5,-7);
            \draw[-<-=.8] (5,-9) -- (5,-7);
            \draw (-9,9) rectangle (9,-9);
		    \node at (-9,-9) {$\ast$};
		    \node at (0,0) [rotate=90]{${\cdots}$};
        }\ 
    \]
    \caption{The diagram $R_{l_i}$ has $l_i$ crossings where $l_i\in 2\bZ_{>0}$.}
    \label{fig:twist-region}    
\end{figure}

\begin{eg}\label{eg:B-adequate}
	The diagram below represents an anti-parallel $B$-adequate link.
	\[
		\mathord{\ 
			\tikz[baseline=-.6ex, scale=.2]{
    			\draw[-<-=.5, rounded corners] (-12,1) -- (-14,1) -- (-14,12) -- (-4,12) -- (-4,10);
    			\draw[-<-=.5, rounded corners] (-4,6) -- (-4,1) -- (-6,1);
    			\draw[->-=.5, rounded corners] (-12,-1) -- (-14,-1) -- (-14,-12) -- (-4,-12) -- (-4,-10);
    			\draw[->-=.5, rounded corners] (-4,-6) -- (-4,-1) -- (-6,-1);
                \draw[-<-=.3, rounded corners] (-2,-6) -- (-2,-1) -- (6,-1);
    			\draw[-<-=.5, rounded corners] (8,-1) -- (10,-1);
    			\draw[->-=.5, rounded corners] (-2,-10) -- (-2,-13) -- (6,-13);
    			\draw[->-=.5, rounded corners] (8,-13) -- (14,-13) -- (14,-1) -- (12,-1);
    			\draw[-<-=.5, rounded corners] (6,-4) -- (4,-4) -- (4,-10) -- (6,-10);
    			\draw[->-=.5, rounded corners] (8,-4) -- (10,-4) -- (10,-10) -- (8,-10);
    			\draw[->-=.5, rounded corners] (-2,6) -- (-2,1) -- (10,1);
    			\draw[-<-=.5, rounded corners] (-2,10) -- (-2,12) -- (14,12) -- (14,1) -- (12,1);
                \begin{scope}[xscale=.8, xshift=-3cm]
                    \draw (-12,1) to[out=east, in=west] (-10,-1);
                    \draw[overarc] (-12,-1) to[out=east, in=west] (-10,1);
                    \draw[overarc] (-10,1) to[out=east, in=west] (-8,-1);
                    \draw[overarc] (-10,-1) to[out=east, in=west] (-8,1);
                    \draw[overarc] (-8,1) to[out=east, in=west] (-6,-1);
                    \draw[overarc] (-8,-1) to[out=east, in=west] (-6,1);
                    \draw[overarc] (-6,1) to[out=east, in=west] (-4,-1);
                    \draw[overarc] (-6,-1) to[out=east, in=west] (-4,1);    
                \end{scope}
    			\draw[overarc] (-4,-10) to[out=north, in=south] (-2,-8);
    			\draw[overarc] (-2,-10) to[out=north, in=south] (-4,-8);
    			\draw[overarc] (-4,-8) to[out=north, in=south] (-2,-6);
    			\draw[overarc] (-2,-8) to[out=north, in=south] (-4,-6);
    			\draw[overarc] (-4,6) to[out=north, in=south] (-2,8);
    			\draw[overarc] (-2,6) to[out=north, in=south] (-4,8);
    			\draw[overarc] (-4,8) to[out=north, in=south] (-2,10);
    			\draw[overarc] (-2,8) to[out=north, in=south] (-4,10);
                \begin{scope}[xscale=.8, xshift=3.5cm]
                    \draw[overarc] (12,-1) to[out=west, in=east] (10,1);
                    \draw[overarc] (12,1) to[out=west, in=east] (10,-1);
                    \draw[overarc] (10,-1) to[out=west, in=east] (8,1);
                    \draw[overarc] (10,1) to[out=west, in=east] (8,-1);
                \end{scope}
                \begin{scope}[xscale=.8, xshift=2cm]
                    \draw[overarc] (8,-4) to[out=west, in=east] (6,-1);
                    \draw[overarc] (8,-1) to[out=west, in=east] (6,-4);
                    \draw[overarc] (6,-4) to[out=west, in=east] (4,-1);
                    \draw[overarc] (6,-1) to[out=west, in=east] (4,-4);
                \end{scope}
                \begin{scope}[xscale=.8, xshift=2cm]
                    \draw[overarc] (6,-10) to[out=east, in=west] (8,-13);
                    \draw[overarc] (8,-10) to[out=west, in=east] (6,-13);
                    \draw[overarc] (4,-10) to[out=east, in=west] (6,-13);
                    \draw[overarc] (6,-10) to[out=west, in=east] (4,-13);
                \end{scope}    			
    		}
    	\ }
    \]
\end{eg}
\begin{eg}[Plumbed-like links]\label{eg:plumbed}
    Let $X$ be a planar embedded graph equipped with a weight $\bm{l}\colon E(X)\to 2\mathbb{Z}_{\geq 0}$ for the edge set $E(X)$.
    Then, we obtain an anti-parallel $B$-adequate link diagram from $X$ by replacing all vertices with positively oriented circles and then adding a twist $R_{\bm{l}(e)}$ between two circles connected by an edge $e\in E(X)$.
\end{eg}

Before we prove the zero stability for anti-parallel $B$-adequate links, let us introduce some symbols for values of special $\mathfrak{sl}_3$-webs and coefficients.
We can describe the one-row colored $\mathfrak{sl}_3$-Jones polynomial by using these values.

\begin{lem}\label{evaluation}
    \begin{align*}
        \Delta^{(n)}(j)&=\frac{[n-j+1]^2[2n-2j+2]}{[2]},\quad 
        \Theta^{(n)}(j)=\frac{{2n-j+2\brack 2n-2j+2}}{{n\brack j}^{2}}\Delta^{(n)}(j),\\
        \gamma^{(n)}(j)&=(-1)^{n-j}q^{-\frac{1}{6}n^{2}}q^{-\frac{1}{2}j^{2}+(n+1)j},
    \end{align*}
    where
    \begin{align*}
        \Delta^{(n)}(j)=
        \tikz[baseline=-.6ex, scale=.1]{
            \draw[->-=.2] (0,0) circle [radius=5];
            \draw[-<-=.2] (0,0) circle [radius=8];
            \draw[fill=white] (-10,-1) rectangle (-3,1);
            \node at (0,8) [above]{${\scriptstyle n-j}$};
            \node at (0,4) [below]{${\scriptstyle n-j}$};
        }\ ,\quad
        \Theta^{(n)}(j)=
         \tikz[baseline=-.6ex, scale=.1]{
            \draw[-<-=.5] (-5,-8) -- (-5,-1);
            \draw[->-=.5] (-5,8) -- (-5,1);
            \draw[->-=.5] (5,-8) -- (5,-1);
            \draw[-<-=.5] (5,8) -- (5,1);
            \draw[->-=.5] (-2,8) to[out=south, in=west] (0,6) to[out=east, in=south] (2,8);
            \draw[-<-=.5] (-2,-8) to[out=north, in=west] (0,-6) to[out=east, in=north] (2,-8);
            \draw[->-=.5, rounded corners] (4,10) -- (4,12) -- (10,12) -- (10,-12) -- (4,-12) -- (4,-10);
            \draw[-<-=.5, rounded corners] (-4,10) -- (-4,12) -- (-10,12) -- (-10,-12) -- (-4,-12) -- (-4,-10);
            \draw[fill=white] (-6,8) rectangle (-1,10);
            \draw[fill=white] (6,8) rectangle (1,10);
            \draw[fill=white] (-6,-1) rectangle (6,1);
            \draw[fill=white] (-6,-8) rectangle (-1,-10);
            \draw[fill=white] (6,-8) rectangle (1,-10);
            \node at (0,6) [below]{${\scriptstyle j}$};
            \node at (0,-6) [above]{${\scriptstyle j}$};
            \node at (3,10) [above]{${\scriptstyle n}$};
            \node at (-3,10) [above]{${\scriptstyle n}$};
        }\ ,\quad
        \tikz[baseline=-.6ex, scale=.1, yshift=-9cm]{
            \draw[->-=.5] (-5,8) -- (-5,1);
            \draw[-<-=.5] (5,8) -- (5,1);
            \draw[->-=.5] (-2,8) to[out=south, in=west] (0,6) to[out=east, in=south] (2,8);
            \draw[overarc, -<-=.8] (-4,10) to[out=north, in=south] (4,17);
            \draw[overarc, ->-=.8] (4,10) to[out=north, in=south] (-4,17);
            \draw[fill=white] (-6,8) rectangle (-1,10);
            \draw[fill=white] (6,8) rectangle (1,10);
            \draw[fill=white] (-6,-1) rectangle (6,1);
            \node at (0,6) [below]{${\scriptstyle j}$};
            \node at (4,17) [right]{${\scriptstyle n}$};
            \node at (-4,17) [left]{${\scriptstyle n}$};
        }
        =\gamma^{(n)}(j)
        \tikz[baseline=-.6ex, scale=.1, yshift=-7cm]{
            \draw[-<-=.5] (-5,8) -- (-5,1);
            \draw[->-=.5] (5,8) -- (5,1);
            \draw[-<-=.5] (-2,8) to[out=south, in=west] (0,6) to[out=east, in=south] (2,8);
            \draw[->-=.5] (-4,10) -- (-4,14);
            \draw[-<-=.5] (4,10) -- (4,14);
            \draw[fill=white] (-6,8) rectangle (-1,10);
            \draw[fill=white] (6,8) rectangle (1,10);
            \draw[fill=white] (-6,-1) rectangle (6,1);
            \node at (0,6) [below]{${\scriptstyle j}$};
            \node at (4,14) [right]{${\scriptstyle n}$};
            \node at (-4,14) [left]{${\scriptstyle n}$};
        }.
    \end{align*}
\end{lem}
\begin{proof}
    It is well-known that the value of a closure of $JW_{{-}^{m}{+}^{n}}$ is obtained by quantum dimension $\frac{[m+1][n+1][m+n+2]}{[2]}$, and it is just $\Delta_j^{(n)}$.
    One can compute the value $\gamma$ by using \cref{claspformula}~(3),(4), and (6).
    $\Theta^{(n)}(j)$ was computed in \cite{Yuasa18}.  
\end{proof}
\begin{lem}\label{comparedeg}
    \begin{align*}
        \mdeg{\Delta^{(n)}(j)}&=\mdeg{\Delta^{(n)}(j-1)}+2,\\
        \mdeg{\Theta^{(n)}(j)}&=\mdeg{\Theta^{(n)}(j-1)}+1,\\
        \mdeg{\gamma^{(n)}(j)}&=\mdeg{\gamma^{(n)}(j-1)}+\left(n-j+\frac{3}{2}\right).
    \end{align*}
\end{lem}
\begin{proof}
    From \cref{lem:deg} and \cref{eg:deg}, one can compute the minimum degrees of $\Delta^{(n)}(j)$ and $\Theta^{(n)}(j)$ as $\mdeg{\Delta^{(n)}(j)}=-2n+2j$ and $\mdeg{\Theta^{(n)}(j)}=-2n+j$.
\end{proof}

An $n$-parallelization $L_n$ of the anti-parallel $B$-adequate link diagram $L$ defines an adequate graph $\mathsf{G}_n\coloneqq L_n\cap D(k)\in\Graph{\cup_{i=1}^{k}\bs_i}{D(k)}$ where $\bs_i^{(0)}=\bar{\bs}_i^{(1)}=\epsilon_i^n\bar{\epsilon}_i^n$ for some $\epsilon_i\in\{\pm\}$ and an $n$-parallelization $(R_{l_i})_n\coloneqq L_n\cap D_i$ of $l_i$ half twists for any $i$.
Then, we define a clasped $\mathfrak{sl}_3$-web $R_{l_i}^{(n)}$ in $D_i$ by inserting one-row colored $\mathfrak{sl}_3$-clasps for each $n$ parallelized strands of $(R_{l_i})_{n}$.
The one-row colored $\mathfrak{sl}_3$-Jones polynomial $J_{L,n}^{\mathfrak{sl}_3}(q)$ of the anti-parallel $B$-adequate link $L$ is obtained by replacing $(R_{l_i})_n$ of $L_n$ with $R_{l_i}^{(n)}$.
Using a linear map defined by $\mathsf{G}_n$, this replacement is described as $\mathsf{G}_{n}(\otimes_{i=1}^{k}({R_{l_i}}^{(n)}))$.

\begin{lem}\label{linkexpansion}
    For an anti-parallel $B$-adequate link diagam $L$ with twist region $D_i$
    \[
        L^{(n)}=\mathsf{G}_{n}(\otimes_{i=1}^{k}({R_{l_i}}^{(n)}))
        =\sum_{t_1,t_2\ldots,t_k=0}^{n}\prod_{i=1}^{k}\Gamma^{(n)}(t_{i};l_{i})\mathsf{G}_{n}(\otimes_{i=1}^{k}M(t_{i};n)),
    \]
    where 
    \[
        \Gamma^{(n)}(t_{i};l_{i})=\gamma^{(n)}(t_{i})^{l_{i}}\frac{\Delta^{(n)}(t_{i})}{\Theta^{(n)}(t_{i})},
    \]
    and
    \[
        M(t_{i};n)=
		\tikz[baseline=-.6ex, scale=.1]{
			\draw[->-=.7] (6,11) -- (6,14);
			\draw[-<-=.7] (-6,11) -- (-6,14);
			\draw[-<-=.7] (6,-11) -- (6,-14);
			\draw[->-=.7] (-6,-11) -- (-6,-14);
			\draw[-<-=.5] (-8,-10) -- (-8,0);
			\draw[->-=.5] (-8,10) -- (-8,0);
			\draw[->-=.5] (8,-10) -- (8,0);
			\draw[-<-=.5] (8,10) -- (8,0);
			\draw[->-=.5] (0,7) to[out=east, in=south] (3,10);
			\draw (0,7) to[out=west, in=south] (-3,10);
			\draw[-<-=.5] (0,-7) to[out=east, in=north] (3,-10);
			\draw (0,-7) to[out=west, in=north] (-3,-10);
			\draw[fill=white] (-10,10) rectangle (-2,11);
			\draw[fill=white] (10,10) rectangle (2,11);
			\draw[fill=white] (-10,-.5) rectangle (10,.5);
			\draw[fill=white] (-10,-10) rectangle (-2,-11);
			\draw[fill=white] (10,-10) rectangle (2,-11);
			\node at (2,7) [right]{${\scriptstyle t_{i}}$};
			\node at (2,-7) [right]{${\scriptstyle t_{i}}$};
			\node at (8,7) [right]{${\scriptstyle n-t_{i}}$};
			\node at (-8,7) [left]{${\scriptstyle n-t_{i}}$};
			\node at (8,-7) [right]{${\scriptstyle n-t_{i}}$};
			\node at (-8,-7) [left]{${\scriptstyle n-t_{i}}$};
			\node at (6,13) [right]{${\scriptstyle n}$};
			\node at (-6,13) [right]{${\scriptstyle n}$};
			\node at (6,-13) [right]{${\scriptstyle n}$};
			\node at (-6,-13) [right]{${\scriptstyle n}$};
		}\ 
	\text{or}
		\tikz[baseline=-.6ex, scale=.1]{
			\draw[-<-=.7] (6,11) -- (6,14);
			\draw[->-=.7] (-6,11) -- (-6,14);
			\draw[->-=.7] (6,-11) -- (6,-14);
			\draw[-<-=.7] (-6,-11) -- (-6,-14);
			\draw[->-=.5] (-8,-10) -- (-8,0);
			\draw[-<-=.5] (-8,10) -- (-8,0);
			\draw[-<-=.5] (8,-10) -- (8,0);
			\draw[->-=.5] (8,10) -- (8,0);
			\draw[-<-=.5] (0,7) to[out=east, in=south] (3,10);
			\draw (0,7) to[out=west, in=south] (-3,10);
			\draw[->-=.5] (0,-7) to[out=east, in=north] (3,-10);
			\draw (0,-7) to[out=west, in=north] (-3,-10);
			\draw[fill=white] (-10,10) rectangle (-2,11);
			\draw[fill=white] (10,10) rectangle (2,11);
			\draw[fill=white] (-10,-.5) rectangle (10,.5);
			\draw[fill=white] (-10,-10) rectangle (-2,-11);
			\draw[fill=white] (10,-10) rectangle (2,-11);
			\node at (2,7) [right]{${\scriptstyle t_{i}}$};
			\node at (2,-7) [right]{${\scriptstyle t_{i}}$};
			\node at (8,7) [right]{${\scriptstyle n-t_{i}}$};
			\node at (-8,7) [left]{${\scriptstyle n-t_{i}}$};
			\node at (8,-7) [right]{${\scriptstyle n-t_{i}}$};
			\node at (-8,-7) [left]{${\scriptstyle n-t_{i}}$};
			\node at (6,13) [right]{${\scriptstyle n}$};
			\node at (-6,13) [right]{${\scriptstyle n}$};
			\node at (6,-13) [right]{${\scriptstyle n}$};
			\node at (-6,-13) [right]{${\scriptstyle n}$};
		}.
    \]
\end{lem}
\begin{proof}
    Apply the formula  
    \[
        \ \tikz[baseline=-.6ex, scale=.1]{
            \draw[-<-=.3] (-2,-3) -- (-2,3);
            \draw[->-=.3] (2,-3) -- (2,3);
            \draw[fill=white] (-3,0) rectangle (-1,1);
            \draw[fill=white] (3,0) rectangle (1,1);
            \node at (-2,3) [above]{${\scriptstyle n}$};
            \node at (2,3) [above]{${\scriptstyle n}$};
        }\ 
        =\sum_{t=0}^{n}\frac{\Delta^{(n)}(t)}{\Theta^{(n)}(t)}M(t;n)
    \]
    shown in \cite{Yuasa18} to all twist regions, and resolve twists by definition of $\gamma^{(n)}(j)$ in \cref{evaluation}.
    We obtain the desired formula.
\end{proof}

\begin{lem}\label{lem:Gamma-deg}
    $\mdeg{\Gamma^{(n)}(t_{i};l_{i})}=\mdeg{\Gamma^{(n)}(t_{i}-1;l_{i})}+l_i\left(n-t_i+\frac{3}{2}\right)+1$.
\end{lem}
\begin{proof}
    By \cref{lem:deg} and \cref{comparedeg},
    \begin{align*}
        \mdeg{\Gamma^{(n)}(t_{i};l_{i})}
        &=l_{i}\mdeg{\gamma^{(n)}(t_{i})}+\mdeg{\Delta^{(n)}(t_{i})}-\mdeg{\Theta^{(n)}(t_{i})}\\
        &=l_{i}\mdeg{\gamma^{(n)}(t_{i}-1)}+\mdeg{\Delta^{(n)}(t_{i}-1)}-\mdeg{\Theta^{(n)}(t_{i}-1)}\\
        &\qquad +l_i\left(n-t_i+\frac{3}{2}\right)+1\\
        &=\mdeg{\Gamma^{(n)}(t_{i}-1;l_{i})}+l_i\left(n-t_i+\frac{3}{2}\right)+1
    \end{align*}
\end{proof}

\begin{prop}\label{linktograph}
Let $L$ be an anti-parallel $B$-adequate link diagram with twist regions $\sqcup_{i=1}^k D_i$, $\bm{l}\colon \{D_i\}_{i=1}^k\to 2\mathbb{Z}_{>0}$, and $\mathsf{G}=L\cap D(k)$.
Then,
\[
    \mathsf{G}_{n}(\otimes_{i=1}^{k}({R_{l_i}}^{(n)}))-\prod_{i=1}^{k}\gamma^{(n)}(0)^{l_{i}}\mathsf{G}_{n}(\otimes_{i=1}^{k}\mathrm{JW}_{D_i})\in q^{2(n+2)}+\mdeg{L^{(n)}}\mathbb{Z}[q].
\]
\end{prop}
\begin{proof}
    The proof is located below \cref{lem:step4}.
\end{proof}

To prove \cref{linktograph}, we prepare several lemmas.
First of all, let us introduce an operation $\saddle_{j}$ corresponding to the single orientable saddle cobordism at $D_j$ for $i=1,2,\ldots,k$.
More precisely, $\saddle_j$ acts on the set $\{\otimes_{i=1}^{k}\uparrow\!\ID^{(i)}\!\downarrow_{t_i} \mid 0\geq t_i\geq n\}$ of an $\mathfrak{sl}_3$-webs in $\sqcup_{i=1}^{k}D_i$ as follows:
$\saddle_j$ replaces $\uparrow\!\ID^{(j)}\!\downarrow_{t_j}$ with $\uparrow\!\ID^{(j)}\!\downarrow_{t_j-1}$, and acts by identity on $\uparrow\!\ID^{(j)}\!\downarrow_{0}$ in $D_j$ or elements in $D_i$ with $i\neq j$.
We remark that $\otimes_{i=1}^{k}\uparrow\!\ID^{(i)}\!\downarrow_{t_i}$ change to $\otimes_{i=1}^{k}\ID_{D_i}$ by a composition of orientable saddle cobordisms.
For instance, $\saddle_k^{t_k}\cdots\saddle_2^{t_2}\saddle_1^{t_1}$ realize this deformation.

\begin{lem}\label{lem:step1}
    For an adequate graph $\mathsf{G}$ of the anti-parallel $B$-adequate link diagram $L$ in \cref{linktograph} and any fixed tupple $(t_1,\ldots,t_k)\in\{0,\ldots,n\}^k$,
    \[
        \mdeg{\mathsf{G}_{n}(\otimes_{i=1}^{k}M(t_{i};n))}\geq\mdeg{\mathsf{G}_{n}(\otimes_{i=1}^{k}\uparrow\!\ID^{(i)}\!\downarrow_{t_i})}.
    \]
\end{lem}
\begin{proof}
    A clasped $\mathfrak{sl}_3$-web $M(t_{i};n)$ has five $\mathfrak{sl}_3$-clasps in $D_i$.
    For all $i=1,2,\ldots,k$, one can choose five small disks $D_{(i,1)},\ldots,D_{(i,5)}$ in $D_i$ so that each small disk surrounds a single $\mathfrak{sl}_3$-clasp.
    The intersection of $L$ and $D\setminus\sqcup\{D_{(i,j)}\mid 1\leq i \leq k, 1\leq j\leq 5\}$ become a graph $\mathsf{G}'$ with no trivalent vertices.
    Then, $\mathsf{G}_{n}(\otimes_{i=1}^{k}M(t_{i};n))=\mathsf{G}'_{n}(\otimes_{i=1}^{k}\otimes_{j=1}^{5}\mathrm{JW}_{D_{(i,j)}})$ by the construction of $\mathsf{G}'$ where $\mathrm{JW}_{D_{(i,j)}}$ is a one- or two-row colored clasp.
    Apply \cref{claspgraphdeg} to $\mathsf{G}'_{n}(\otimes_{i=1}^{k}\otimes_{j=1}^{5}\mathrm{JW}_{D_{(i,j)}})$ and we obtain
    \begin{align*}
        \mdeg{\mathsf{G}_{n}(\otimes_{i=1}^{k}M(t_{i};n))}
        &=\mdeg{\mathsf{G}'_{n}(\otimes_{i=1}^{k}\otimes_{j=1}^{5}\mathrm{JW}_{D_{(i,j)}})}\\
        &\geq\mdeg{\mathsf{G}'_{n}(\otimes_{i=1}^{k}\otimes_{j=1}^{5}\ID_{D_{(i,j)}})}\\
        &=\mdeg{\mathsf{G}_{n}(\otimes_{i=1}^{k}\uparrow\!\ID^{(i)}\!\downarrow_{t_i})}.
    \end{align*}
\end{proof}
\begin{lem}\label{lem:step2}
    For an adequate graph $\mathsf{G}$ of the anti-parallel $B$-adequate link diagram $L$ in \cref{linktograph} and any fixed tupple $(t_1,\ldots,t_k)\in\{0,\ldots,n\}^k$ with $0<t_j\leq n$,
    \[
        \mdeg{\Gamma^{(n)}(t_{j};l_{j})\mathsf{G}_{n}(\otimes_{i=1}^{k}\uparrow\!\ID^{(i)}\!\downarrow_{t_i})}>\mdeg{\Gamma^{(n)}(t_{j}-1;l_{j})\mathsf{G}_{n}(\saddle_{j}(\otimes_{i=1}^{k}\uparrow\!\ID^{(i)}\!\downarrow_{t_i}))}
    \]
\end{lem}
\begin{proof}
    We firstly note that $\mdeg{\mathsf{G}_{n}(\otimes_{i=1}^{k}{\uparrow\!\ID^{(i)}\!\downarrow_{t_i}})}=-c(\mathsf{G}_{n}(\otimes_{i=1}^{k}{\uparrow\!\ID^{(i)}\!\downarrow_{t_i}}))$ and $\mdeg{\mathsf{G}_{n}(\saddle_{j}(\otimes_{i=1}^{k}{\uparrow\!\ID^{(i)}\!\downarrow_{t_i}}))}=-c(\mathsf{G}_{n}(\saddle_{j}(\otimes_{i=1}^{k}{\uparrow\!\ID^{(i)}\!\downarrow_{t_i}})))$ by \cref{graphdeg}.
    The orientable saddle cobordism $\saddle_j$ changes the number of connected components by $\pm 1$ or $0$.
    Hence $\mdeg{\mathsf{G}_{n}(\otimes_{i=1}^{k}{\uparrow\!\ID^{(i)}\!\downarrow_{t_i}})}\geq\mdeg{\mathsf{G}_{n}(\saddle_{j}(\otimes_{i=1}^{k}{\uparrow\!\ID^{(i)}\!\downarrow_{t_i}}))}-1$.
    This inequality and \cref{lem:Gamma-deg} conclude
    \begin{align*}
        \mdeg{\Gamma^{(n)}(t_{j};l_{j})\mathsf{G}_{n}(\otimes_{i=1}^{k}{\uparrow\!\ID^{(i)}\!\downarrow_{t_i}})}
        &=\mdeg{\Gamma^{(n)}(t_{j};l_{j})}+\mdeg{\mathsf{G}_{n}(\otimes_{i=1}^{k}{\uparrow\!\ID^{(i)}\!\downarrow_{t_i}})}\\
        &\geq\mdeg{\Gamma^{(n)}(t_{j};l_{j})}+\mdeg{\mathsf{G}_{n}(\saddle_{j}(\otimes_{i=1}^{k}{\uparrow\!\ID^{(i)}\!\downarrow_{t_i}}))}-1\\
        &=\mdeg{\Gamma^{(n)}(t_{j}-1;l_{j})}+\mdeg{\mathsf{G}_{n}(\saddle_{j}(\otimes_{i=1}^{k}{\uparrow\!\ID^{(i)}\!\downarrow_{t_i}}))}+l_i(n-t_i+\frac{3}{2})\\
        &=\mdeg{\Gamma^{(n)}(t_{j}-1;l_{j})\mathsf{G}_{n}(\saddle_{j}(\otimes_{i=1}^{k}{\uparrow\!\ID^{(i)}\!\downarrow_{t_i}}))}+l_i(n-t_i+\frac{3}{2}).
    \end{align*}
    One can easy to see that $l_i(n-t_i+\frac{3}{2})\geq 3$ because $l_i\in 2\mathbb{Z}_{>0}$ and $0<t_j\leq n$.
\end{proof}
\begin{lem}\label{lem:step3}
    For an adequate graph $\mathsf{G}$ of the anti-parallel $B$-adequate link diagram $L$ in \cref{linktograph} and any $0<j\leq k$,
    \[
        \mdeg{\Big(\prod_{i=1}^{k}\Gamma^{(n)}(\delta_{ij};l_{i})\Big)\mathsf{G}_{n}(\otimes_{i=1}^{k}\uparrow\!\ID^{(i)}\!\downarrow_{\delta_{ij}})}-\mdeg{(\prod_{i=1}^{k}\Gamma^{(n)}(0;l_{i}))\mathsf{G}_{n}(\otimes_{i=1}^{k}\ID_{D_{i}})}\geq 2n+3,
    \]
    where $\delta_{ij}$ is the Kronecker delta function.
\end{lem}
\begin{proof}
    The adequacy of $\mathsf{G}_n$ says that $c(\mathsf{G}_{n}(\otimes_{i=1}^{k}{\uparrow\!\ID^{(i)}\!\downarrow}_{\delta_{ij}}))-c(\mathsf{G}_{n}(\otimes_{i=1}^{k}\ID_{D_{i}}))=-1$.
    We know $\mdeg{\Gamma^{(n)}(1;l_{j})}-\mdeg{\Gamma^{(n)}(0;l_{j})}=l_j(n+\frac{1}{2})+1$ by \cref{lem:Gamma-deg},
    $\mdeg{\mathsf{G}_{n}(\otimes_{i=1}^{k}\uparrow\!\ID^{(i)}\!\downarrow_{\delta_{ij}})}=-c(\mathsf{G}_{n}(\otimes_{i=1}^{k}\uparrow\!\ID^{(i)}\!\downarrow_{\delta_{ij}}))$ and $\mdeg{\mathsf{G}_{n}(\otimes_{i=1}^{k}\ID_{D_{i}})}=-c(\mathsf{G}_{n}(\otimes_{i=1}^{k}\ID_{D_{i}}))$ by \cref{graphdeg}.
    Hence,
    \begin{align*}
        &\mdeg{\Big(\prod_{i=1}^{k}\Gamma^{(n)}(\delta_{ij};l_{i})\Big)\mathsf{G}_{n}(\otimes_{i=1}^{k}\uparrow\!\ID^{(i)}\!\downarrow_{\delta_{ij}})}-\mdeg{(\prod_{i=1}^{k}\Gamma^{(n)}(0;l_{i}))\mathsf{G}_{n}(\otimes_{i=1}^{k}\ID_{D_{i}})}\\
        &=\mdeg{\Gamma^{(n)}(1;l_{j})}-\mdeg{\Gamma^{(n)}(0;l_{j})}+\mdeg{\mathsf{G}_{n}(\otimes_{i=1}^{k}\uparrow\!\ID^{(i)}\!\downarrow_{\delta_{ij}})}-\mdeg{\mathsf{G}_{n}(\otimes_{i=1}^{k}\ID_{D_{i}})}\\
        &=l_j(n+\frac{1}{2})+1-c(\mathsf{G}_{n}(\otimes_{i=1}^{k}\uparrow\!\ID^{(i)}\!\downarrow_{\delta_{ij}}))+c(\mathsf{G}_{n}(\otimes_{i=1}^{k}\ID_{D_{i}}))\\
        &=l_j(n+\frac{1}{2})+2\geq 2n+3.
    \end{align*}
    The last inequality holds because $l_j$ is a positive even integer.
\end{proof}
\begin{lem}\label{lem:step4}
    For an adequate graph $\mathsf{G}$ of the anti-parallel $B$-adequate link diagram $L$ in \cref{linktograph},
    \[
        \mdeg{\Big(\prod_{i=1}^k\Gamma^{(n)}(0;l_{i})\Big)\mathsf{G}_{n}(\otimes_{i=1}^{k}\ID_{D_i})}=\mdeg{(\prod_{i=1}^k\Gamma^{(n)}(0;l_{i}))\mathsf{G}_{n}(\otimes_{i=1}^{k}\mathrm{JW}_{D_i})}
    \]
\end{lem}
\begin{proof}
    This assertion comes from the adequacy of $\mathsf{G}$ and \cref{claspgraphdeg}.
\end{proof}
\begin{proof}[Proof of \cref{linktograph}]
    By \cref{lem:deg} and \cref{lem:step1}, we obtain
    \begin{align*}
        \mdeg{\Big(\prod_{i=1}^{k}\Gamma^{(n)}(t_{i};l_{i})\Big)\mathsf{G}_{n}(\otimes_{i=1}^{k}M(t_{i};n))}
        &=\left(\sum_{i=1}^{k}\mdeg{\Gamma^{(n)}(t_{i};l_{i})}\right)+\mdeg{\mathsf{G}_{n}(\otimes_{i=1}^{k}M(t_{i};n))}\\
        &\geq\left(\sum_{i=1}^{k}\mdeg{\Gamma^{(n)}(t_{i};l_{i})}\right)+\mdeg{\mathsf{G}_{n}(\otimes_{i=1}^{k}\uparrow\!\ID^{(i)}\!\downarrow_{t_i})}\\
        &=\mdeg{\Big(\prod_{i=1}^{k}\Gamma^{(n)}(t_{i};l_{i})\Big)\mathsf{G}_{n}(\otimes_{i=1}^{k}\uparrow\!\ID^{(i)}\!\downarrow_{t_i})}.
    \end{align*}
    Choose a sequence of orientable saddle cobordisms that changes $\mathsf{G}_{n}(\otimes_{i=1}^{k}{\uparrow\!\ID^{(i)}\!\downarrow_{t_i}})$ to $\mathsf{G}_{n}(\otimes_{i=1}^{k}{\uparrow\!\ID^{(i)}\!\downarrow_{0}})=\mathsf{G}_{n}(\otimes_{i=1}^{k}\ID_{D_i})$, and apply \cref{lem:step2} to $\mdeg{(\prod_{i=1}^{k}\Gamma^{(n)}(t_{i};l_{i}))\mathsf{G}_{n}(\otimes_{i=1}^{k}{\uparrow\!\ID^{(i)}\!\downarrow_{t_i}})}$ along the sequence until just before the last step.
    We can apply \cref{lem:step3} to the last orientable saddle cobordism given by $\saddle_j$.
    This operation gives
    \begin{align*}
        \mdeg{\Big(\prod_{i=1}^{k}\Gamma^{(n)}(t_{i};l_{i})\Big)\mathsf{G}_{n}(\otimes_{i=1}^{k}\uparrow\!\ID^{(i)}\!\downarrow_{t_i})}
        &>\mdeg{\Big(\prod_{i=1}^{k}\Gamma^{(n)}(\delta_{ij};l_{i})\Big)\mathsf{G}_{n}(\otimes_{i=1}^{k}\uparrow\!\ID^{(i)}\!\downarrow_{\delta_{ij}})}\\
        &\geq\mdeg{\Big(\prod_{i=1}^k\Gamma^{(n)}(0;l_{i})\Big)\mathsf{G}_{n}(\otimes_{i=1}^{k}\ID_{D_i})}+2n+3.
    \end{align*}
    The above two inequalities and \cref{lem:step4} conclude the following:
    \begin{align}
        \mdeg{\Big(\prod_{i=1}^{k}\Gamma^{(n)}(t_{i};l_{i})\Big)\mathsf{G}_{n}(\otimes_{i=1}^{k}M(t_{i};n))}
        &>\mdeg{\Big(\prod_{i=1}^k\Gamma^{(n)}(0;l_{i})\Big)\mathsf{G}_{n}(\otimes_{i=1}^{k}\ID_{D_i})}+2n+3\label{eq:step_ineq}\\
        &=\mdeg{\Big(\prod_{i=1}^k\Gamma^{(n)}(0;l_{i})\Big)\mathsf{G}_{n}(\otimes_{i=1}^{k}\mathrm{JW}_{D_i})}+2n+3\notag
    \end{align}
    \sloppy for any $(t_1,t_2,\ldots,t_k)\neq (0,0,\ldots,0)$.
    Finally, we will compare $\mdeg{\mathsf{G}_{n}(\otimes_{i=1}^{k}({R_{l_i}}^{(n)}))}$ to $\mdeg{\prod_{i=1}^{k}\gamma^{(n)}(0)^{l_{i}}\mathsf{G}_{n}(\otimes_{i=1}^{k}M(0;n))}$ by using the expansion in \cref{linkexpansion} and \cref{eq:step_ineq}.
    By \cref{linkexpansion},
    \begin{align*}
        &\mathsf{G}_{n}(\otimes_{i=1}^{k}({R_{l_i}}^{(n)}))\\
        &\quad=\Big(\prod_{i=1}^{k}\Gamma^{(n)}(0;l_{i})\Big)\mathsf{G}_{n}(\otimes_{i=1}^{k}M(0;n))
        +\hspace{-2em}\sum_{(t_1,t_2\ldots,t_k)\neq (0,0,\ldots,0)}\Big(\prod_{i=1}^{k}\Gamma^{(n)}(t_{i};l_{i})\Big)\mathsf{G}_{n}(\otimes_{i=1}^{k}M(t_{i};n))\\
        &\quad=\Big(\prod_{i=1}^{k}\Gamma^{(n)}(0;l_{i})\Big)\mathsf{G}_{n}(\otimes_{i=1}^{k}\mathrm{JW}_{D_i})
        +\hspace{-2em}\sum_{(t_1,t_2\ldots,t_k)\neq (0,0,\ldots,0)}\Big(\prod_{i=1}^{k}\Gamma^{(n)}(t_{i};l_{i})\Big)\mathsf{G}_{n}(\otimes_{i=1}^{k}M(t_{i};n)),
    \end{align*}
    By \cref{lem:deg} and \cref{eq:step_ineq}, one can obtain
    \begin{align*}
        &\mdeg{\mathsf{G}_{n}(\otimes_{i=1}^{k}({R_{l_i}}^{(n)}))-\Big(\prod_{i=1}^{k}\Gamma^{(n)}(0;l_{i})\Big)\mathsf{G}_{n}(\otimes_{i=1}^{k}\mathrm{JW}_{D_i})}\\
        &\quad=\min_{(t_1,t_2\ldots,t_k)\neq (0,0,\ldots,0)}\left\{\mdeg{\Big(\prod_{i=1}^{k}\Gamma^{(n)}(t_{i};l_{i})\Big)\mathsf{G}_{n}(\otimes_{i=1}^{k}M(t_{i};n))}\right\}\\
        &\quad>\mdeg{\Big(\prod_{i=1}^k\Gamma^{(n)}(0;l_{i})\Big)\mathsf{G}_{n}(\otimes_{i=1}^{k}\mathrm{JW}_{D_i})}+2n+3
    \end{align*}
    We remark that 
    \[
        \mdeg{\mathsf{G}_{n}(\otimes_{i=1}^{k}({R_{l_i}}^{(n)}))}=\mdeg{\Big(\prod_{i=1}^k\Gamma^{(n)}(0;l_{i})\Big)\mathsf{G}_{n}(\otimes_{i=1}^{k}\mathrm{JW}_{D_i})}
    \]
    by \cref{eq:step_ineq} and $\Gamma^{(n)}(0;l_{i})=\prod_{i=1}^{k}\gamma^{(n)}(0)^{l_i}$ by definition.
\end{proof}

\begin{dfn}\label{nequiv}
    For $f(q)$ and $g(q)$ in $\mathbb{Z}((q))$, we define $f(q)\equiv_{n}g(q)$ if $\mdeg{\hat{f}(q)-\hat{g}(q)}\geq n+1$.
\end{dfn}

\begin{prop}\label{graphstability}
    Let $\mathrm{JW}_{D_i}$ represent $M(0;n)$ in $D_i$ defined in \cref{linkexpansion}.
    If $\mathsf{G}$ is adequate, then
    \[
        \mathsf{G}_{n+1}(\otimes_{i=1}^{k}\mathrm{JW}_{D_i})\equiv_{n+1}\mathsf{G}_{n}(\otimes_{i=1}^{k}\mathrm{JW}_{D_i}).
    \]
\end{prop}
\begin{proof}
    We prove it below \cref{unclasp}.
\end{proof}
We will prove this proposition similar strategy to \cite{Armond13}.
Let us explain it in our situation.
Choose a one-row colored $\mathfrak{sl}_3$-clasp with $n+1$ strands in $\mathsf{G}_{n+1}(\otimes_{i=1}^{k}\ID_{D_i})$. 
It corresponds to $n+1$ parallel circles in $\mathsf{G}(\otimes_{i=1}^{k}\ID_{D_i})$.
First, we move the $n+1$ parallel strands to the left side of the two-row colored $\mathfrak{sl}_3$-clasps at the center of $D_i$.
The ``left side'' is determined by the orientation of $n+1$ parallel strands at each $\mathfrak{sl}_3$-clasp, see the following picture:
\begin{align*}
	\mathord{\ 
		\tikz[baseline=-.6ex, scale=.1]{
			\draw[very thick, red] (6,11) -- (6,14);
			\draw (-6,11) -- (-6,14);
			\draw[very thick, red] (6,-11) -- (6,-14);
			\draw (-6,-11) -- (-6,-14);
			\draw[-<-=.5] (-6,-10) -- (-6,0);
			\draw[->-=.5] (-6,10) -- (-6,0);
			\draw[->-=.5, very thick, red] (6,-10) -- (6,0);
			\draw[-<-=.5, very thick, red] (6,10) -- (6,0);
			\draw[fill=white] (-10,10) rectangle (-2,11);
			\draw[fill=white] (10,10) rectangle (2,11);
			\draw[fill=white] (-10,-.5) rectangle (10,.5);
			\draw[fill=white] (-10,-10) rectangle (-2,-11);
			\draw[fill=white] (10,-10) rectangle (2,-11);
			\node at (6,13) [right]{${\scriptstyle n+1}$};
			\node at (-6,13) [right]{${\scriptstyle n+1}$};
			\node at (6,-13) [right]{${\scriptstyle n+1}$};
			\node at (-6,-13) [right]{${\scriptstyle n+1}$};
		}
	\ }
	=\mathord{\ 
		\tikz[baseline=-.6ex, scale=.1]{
			\draw[very thick, red] (6,11) -- (6,14);
			\draw (-6,11) -- (-6,14);
			\draw[very thick, red] (6,-11) -- (6,-14);
			\draw (-6,-11) -- (-6,-14);
			\draw[-<-=.8] (-6,-10) to[out=north, in=south] (6,0);
			\draw[->-=.8] (-6,10) to[out=south, in=north] (6,0);
			\draw[white, double=red, double distance=1.2pt, line width=2.4pt, ->-=.8] (6,-10) to[out=north, in=south] (-6,0);
			\draw[white, double=red, double distance=1.2pt, line width=2.4pt, -<-=.8] (6,10) to[out=south, in=north] (-6,0);
			\draw[fill=white] (-10,10) rectangle (-2,11);
			\draw[fill=white] (10,10) rectangle (2,11);
			\draw[fill=white] (-10,-.5) rectangle (10,.5);
			\draw[fill=white] (-10,-10) rectangle (-2,-11);
			\draw[fill=white] (10,-10) rectangle (2,-11);
			\node at (6,13) [right]{${\scriptstyle n+1}$};
			\node at (-6,13) [right]{${\scriptstyle n+1}$};
			\node at (6,-13) [right]{${\scriptstyle n+1}$};
			\node at (-6,-13) [right]{${\scriptstyle n+1}$};
		}
	\ }.
\end{align*}
In the above, the chosen $n+1$ parallel strands are expressed as a red arc labeled by $n+1$.
We remark that this deformation of $\mathfrak{sl}_3$-webs does not change the coefficients.
We assume that the chosen $n+1$ parallel strands pass through $m$ two-row colored $\mathfrak{sl}_3$-clasps $\mathrm{JW}_{D_1},\mathrm{JW}_{D_2},\ldots,\mathrm{JW}_{D_m}$ in this order by replacing labels of twist regions if necessary.
We denote the initial $\mathfrak{sl}_3$-web by $\mathsf{G}^{(n+1)}(0)$, and a clasped $\mathfrak{sl}_3$-web obtained by unclasping the leftmost strand of the $n+1$ strands from $\mathrm{JW}_{D_1},\ldots,\mathrm{JW}_{D_{j-1}}$, and $\mathrm{JW}_{D_j}$ in $\mathsf{G}^{(n+1)}(0)$ by $\mathsf{G}^{(n+1)}(j)$ for $j=1,2,\ldots,m-1$.
If one could unclasp the leftmost strand from $\mathrm{JW}_{D_1},\ldots,\mathrm{JW}_{D_{m-1}}$, then the $\mathfrak{sl}_3$-web becomes $\mathsf{G}^{(n+1)}(m-1)$.
One can shrink the unclasped strand in $\mathsf{G}^{(n+1)}(m-1)$ to $\mathrm{JW}_{D_{m}}$ as follows.
\begin{align*}
    \mathsf{G}^{(n+1)}(0)&\to\mathsf{G}^{(n+1)}(1)=
	\mathord{\ 
		\tikz[baseline=-.6ex, scale=.1]{
			\begin{scope}[xshift=15cm]
				\draw[->-=.5, very thick, red] (-2,-15) -- (-2,0);
				\draw[-<-=.5, very thick, red] (-2,15) -- (-2,0);
				\draw[->-=.5, thick, rounded corners, red] (-4,-15) -- (-4,-13) -- (-9,-13) -- (-9,-8) -- (-4,-8) -- (-4,13) -- (-4,15);
				\draw[-<-=.2, rounded corners] (9,8) -- (4,8) -- (4,13) -- (9,13);
				\draw[-<-=.2, rounded corners] (9,-2) -- (4,-2) -- (4,2) -- (9,2);
				\draw[->-=.8, rounded corners] (9,-8) -- (4,-8) -- (4,-13) -- (9,-13);
				\draw[fill=white] (-7,10) rectangle (7,11);
				\draw[fill=white] (-7,-.5) rectangle (7,.5);
				\draw[fill=white] (-7,-10) rectangle (7,-11);
				\node at (-2,13) [right]{${\scriptstyle n}$};
				\node at (-2,-13) [right]{${\scriptstyle n}$};	
				\node at (7,13) [above]{${\scriptstyle n+1}$};
				\node at (7,2) [above]{${\scriptstyle n+1}$};
				\node at (7,-13) [below]{${\scriptstyle n+1}$};
			\end{scope}
			\draw[dashed, rounded corners, very thick, red] (-13,15) -- (-13,20) -- (13,20) -- (13,15);
			\draw[dashed, rounded corners, thick, red] (-11,15) -- (-11,17) -- (11,17) -- (11,15);
			\draw[dashed, rounded corners, very thick, red] (-13,-15) -- (-13,-20) -- (13,-20) -- (13,-15);
			\draw[dashed, rounded corners, thick, red] (-11,-15) -- (-11,-17) -- (11,-17) -- (11,-15);
			\begin{scope}[xshift=-15cm]
				\draw[->-=.2, rounded corners] (-9,-2) -- (-4,-2) -- (-4,2) -- (-9,2);
				\draw[-<-=.5, very thick, red] (2,-15) -- (2,0);
				\draw[->-=.5, very thick, red] (2,15) -- (2,0);
				\draw[-<-=.5, thick, red] (4,-15) -- (4,0);
				\draw[->-=.5, thick, red] (4,15) -- (4,0);
				\draw[fill=white] (-7,-.5) rectangle (7,.5);
				\node at (-7,2) [above]{${\scriptstyle n+1}$};
				\node at (2,13) [left]{${\scriptstyle n}$};
				\node at (2,-13) [left]{${\scriptstyle n}$};	
			\end{scope}
		}
	\ }\longrightarrow
	\mathord{\ 
		\tikz[baseline=-.6ex, scale=.1]{
			\begin{scope}[xshift=15cm]
				\draw[->-=.5, very thick, red] (-2,-15) -- (-2,0);
				\draw[-<-=.5, very thick, red] (-2,15) -- (-2,0);
				\draw[->-=.5, thick, rounded corners, red] (-4,-15) -- (-4,-13) -- (-9,-13) -- (-9,-8) -- (-4,-8) -- (-4,-2) -- (-9,-2) -- (-9,2) -- (-4,2) -- (-4,8) -- (-4,13) -- (-4,15);
				\draw[-<-=.2, rounded corners] (9,8) -- (4,8) -- (4,13) -- (9,13);
				\draw[-<-=.2, rounded corners] (9,-2) -- (4,-2) -- (4,2) -- (9,2);
				\draw[->-=.8, rounded corners] (9,-8) -- (4,-8) -- (4,-13) -- (9,-13);
				\draw[fill=white] (-7,10) rectangle (7,11);
				\draw[fill=white] (-7,-.5) rectangle (7,.5);
				\draw[fill=white] (-7,-10) rectangle (7,-11);
				\node at (-2,13) [right]{${\scriptstyle n}$};
				\node at (-2,-13) [right]{${\scriptstyle n}$};	
				\node at (7,13) [above]{${\scriptstyle n+1}$};
				\node at (7,2) [above]{${\scriptstyle n+1}$};
				\node at (7,-13) [below]{${\scriptstyle n+1}$};
			\end{scope}
			\draw[dashed, rounded corners, very thick, red] (-13,15) -- (-13,20) -- (13,20) -- (13,15);
			\draw[dashed, rounded corners, thick, red] (-11,15) -- (-11,17) -- (11,17) -- (11,15);
			\draw[dashed, rounded corners, very thick, red] (-13,-15) -- (-13,-20) -- (13,-20) -- (13,-15);
			\draw[dashed, rounded corners, thick, red] (-11,-15) -- (-11,-17) -- (11,-17) -- (11,-15);
			\begin{scope}[xshift=-15cm]
				\draw[->-=.2, rounded corners] (-9,-2) -- (-4,-2) -- (-4,2) -- (-9,2);
				\draw[-<-=.5, very thick, red] (2,-15) -- (2,0);
				\draw[->-=.5, very thick, red] (2,15) -- (2,0);
				\draw[-<-=.5, thick, red] (4,-15) -- (4,0);
				\draw[->-=.5, thick, red] (4,15) -- (4,0);
				\draw[fill=white] (-7,-.5) rectangle (7,.5);
				\node at (-7,2) [above]{${\scriptstyle n+1}$};
				\node at (2,13) [left]{${\scriptstyle n}$};
				\node at (2,-13) [left]{${\scriptstyle n}$};	
			\end{scope}
		}
	\ }\\
	&\longrightarrow\cdots\longrightarrow
	\mathord{\ 
		\tikz[baseline=-.6ex, scale=.1]{
			\begin{scope}[xshift=15cm]
				\draw[->-=.5, very thick, red] (-2,-15) -- (-2,0);
				\draw[-<-=.5, very thick, red] (-2,15) -- (-2,0);
				\draw[-<-=.2, rounded corners] (9,8) -- (4,8) -- (4,13) -- (9,13);
				\draw[-<-=.2, rounded corners] (9,-2) -- (4,-2) -- (4,2) -- (9,2);
				\draw[->-=.8, rounded corners] (9,-8) -- (4,-8) -- (4,-13) -- (9,-13);
				\draw[fill=white] (-7,10) rectangle (7,11);
				\draw[fill=white] (-7,-.5) rectangle (7,.5);
				\draw[fill=white] (-7,-10) rectangle (7,-11);
				\node at (-2,13) [right]{${\scriptstyle n}$};
				\node at (-2,-13) [right]{${\scriptstyle n}$};	
				\node at (7,13) [above]{${\scriptstyle n+1}$};
				\node at (7,2) [above]{${\scriptstyle n+1}$};
				\node at (7,-13) [below]{${\scriptstyle n+1}$};
			\end{scope}
			\draw[dashed, rounded corners, very thick, red] (-13,15) -- (-13,20) -- (13,20) -- (13,15);
			\draw[dashed, rounded corners, very thick, red] (-13,-15) -- (-13,-20) -- (13,	-20) -- (13,-15);
			\begin{scope}[xshift=-15cm]
				\draw[->-=.2, rounded corners] (-9,-2) -- (-4,-2) -- (-4,2) -- (-9,2);
				\draw[-<-=.5, very thick, red] (2,-15) -- (2,0);
				\draw[->-=.5, very thick, red] (2,15) -- (2,0);
				\draw[-<-=.5, thick, rounded corners, red] (4,-3) -- (4,3) -- (9,3) -- (9,-3) -- cycle;
				\draw[fill=white] (-7,-.5) rectangle (7,.5);
				\node at (-7,2) [above]{${\scriptstyle n+1}$};
				\node at (2,13) [left]{${\scriptstyle n}$};
				\node at (2,-13) [left]{${\scriptstyle n}$};	
			\end{scope}
		}
	\ }=\mathsf{G}^{(n+1)}(m-1)
\end{align*}

We will see that the above sequence of $\mathfrak{sl}_3$-webs can be realized by computing $\mathsf{G}^{(n+1)}(j)$ modulo $q^{n+1}\bZ[[q]]$.

\begin{lem}\label{unclasp}
	\begin{align*}
		\mathord{\ 
			\tikz[baseline=-.6ex, scale=.1]{
				\draw[->-=.5] (-1,-5) -- (-1,0);
				\draw[-<-=.5] (-1,11) -- (-1,0);
				\draw[-<-=.8, rounded corners] (4,0) -- (4,3)-- (9,3);
				\draw[-<-=.2, rounded corners] (9,-3) -- (4,-3) -- (4,0);
				\draw[fill=white] (-4,-.5) rectangle (7,.5);
				\node at (-1,10) [left]{${\scriptstyle n+1}$};
				\node at (-1,-4) [left]{${\scriptstyle n+1}$};	
				\node at (9,3) [right]{${\scriptstyle k_j}$};
				\node at (9,-3) [right]{${\scriptstyle k_j}$};
			}
		\ }
		=\mathord{\ 
			\tikz[baseline=-.6ex, scale=.1, yshift=-2cm]{
				\draw[->-=.5] (-1,-5) -- (-1,0);
				\draw[->-=.5] (-2,11) -- (-2,15);
				\draw[-<-=.5] (-1,11) -- (-1,0);
				\draw[->-=.5] (-5,-5) -- (-5,11);
				\draw[-<-=.8, rounded corners] (4,0) -- (4,3)-- (9,3);
				\draw[-<-=.2, rounded corners] (9,-3) -- (4,-3) -- (4,0);
				\draw[->-=.5] (-2,-10) -- (-2,-5);
				\draw[fill=white] (-7,10) rectangle (5,11);
				\draw[fill=white] (-7,-5) rectangle (5,-6);
				\draw[fill=white] (-4,-.5) rectangle (7,.5);
				\node at (-2,13) [left]{${\scriptstyle n+1}$};
				\node at (-2,-8) [left]{${\scriptstyle n+1}$};
				\node at (-1,3) [right]{${\scriptstyle n}$};	
				\node at (-1,-3) [right]{${\scriptstyle n}$};	
				\node at (-5,6) [left]{${\scriptstyle 1}$};	
				\node at (9,3) [right]{${\scriptstyle k_j}$};
				\node at (9,-3) [right]{${\scriptstyle k}_j$};
			}
		\ }
		+(-1)^{n+1}\frac{[k_j]}{[n+k_j+2]}\mathord{\ 
			\tikz[baseline=-.6ex, scale=.1, yshift=-2cm]{
				\draw[->-=.5] (-1,-5) -- (-1,0);
				\draw[->-=.5] (2,11) -- (2,15);
				\draw[-<-=.3, -<-=.8] (0,11) -- (0,0);
				\draw[->-=.2, ->-=.9, rounded corners] (-7,-5) -- (-7,6) -- (4,6) -- (4,0);
				\draw[->-=.5] (9,6) to[out=west, in=south] (5,10);
				\draw[-<-=.8, rounded corners] (6,0) -- (6,3)-- (9,3);
				\draw[-<-=.5] (9,5) -- (15,5);
				\draw[-<-=.2, rounded corners] (14,-3) -- (5,-3) -- (5,0);
				\draw[->-=.5] (-2,-10) -- (-2,-5);
				\draw[fill=white] (9,2) rectangle (10,7);
				\draw[fill=white] (-2,5) rectangle (2,7);
				\draw (-2,7) -- (2,5);
				\draw[fill=white] (-7,10) rectangle (7,11);
				\draw[fill=white] (-8,-5) rectangle (5,-6);
				\draw[fill=white] (-4,-.5) rectangle (7,.5);
				\node at (2,13) [left]{${\scriptstyle n+1}$};
				\node at (-2,-8) [left]{${\scriptstyle n+1}$};
				\node at (0,3) [left]{${\scriptstyle n}$};	
				\node at (-1,-3) [right]{${\scriptstyle n}$};	
				\node at (-6,7) [left]{${\scriptstyle 1}$};	
				\node at (6,9) [right]{${\scriptstyle 1}$};
				\node at (14,5) [below]{${\scriptstyle k_j}$};
				\node at (9,-3) [below]{${\scriptstyle k_j}$};
			}
		\ }
	\end{align*}
\end{lem}
\begin{proof}
	We will show it in \cref{prop:unclasp-appendix}.
\end{proof}
	
\begin{proof}[Proof of \cref{graphstability}]
Let us do the unclasping operation that we explained. 
Choose $n+1$ parallel circles passing the left side of $\mathfrak{sl}_3$-clasps $\mathrm{JW}_{D_1},\ldots,\mathrm{JW}_{D_m}$.
We apply \cref{unclasp} to the $j$-th $\mathfrak{sl}_3$-clasp $\mathrm{JW}_{D_j}$ in $\mathsf{G}^{(n+1)}(j-1)$ for $j=1,2,\ldots,m-2$.
Then, we obtain
\[
    \mathsf{G}^{(n+1)}(j-1)=\mathsf{G}^{(n+1)}(j)+(-1)^{n+1}\frac{[k_j]}{[n+k_j+2]}\mathsf{H}^{(n+1)}(j),
\]
where $\mathsf{H}^{(n+1)}(j)$ is a clasped $\mathfrak{sl}_3$-web corresponding to the second term in \cref{unclasp}.
We use $k_j$ above although $k_j=n+1$ in this situation because it is useful for later discussion. 
We compare $\mdeg{\mathsf{G}^{(n+1)}(j)}$ to $\mdeg{\mathsf{H}^{(n+1)}(j)}$.
Let $\widetilde{\mathsf{G}}^{(n+1)}(j)$ and $\widetilde{\mathsf{H}}^{(n+1)}(j)$ denote a flat trivalent graph obtained by replacing all $\mathfrak{sl}_3$-clasps in $\mathsf{G}^{(n+1)}(j)$ and $\mathsf{H}^{(n+1)}(j)$ with identity webs, respectively. 
\cref{claspgraphdeg} says that the lower bound of the minimum degree $\mdeg{\mathsf{H}^{(n+1)}(j)}$ is given by the number of vertices and connected components of $\widetilde{\mathsf{H}}$.
By tracing strands of $\widetilde{\mathsf{H}}^{(n+1)}(j)$ as in \cref{connectedcomp}, one can see that $c(\widetilde{\mathsf{G}}^{(n+1)}(j-1))-c(\widetilde{\mathsf{H}}^{(n+1)}(j))=n+1$ and $v(\widetilde{\mathsf{H}}^{(n+1)}(j))=2n$.
\begin{figure}
	\centering
	\begin{tikzpicture}[scale=.1]
		\draw[->-=.5] (-1,6) -- (-1,12);
		\draw[->-=.5] (1,6) -- (1,12);
		\draw[->-=.5] (-1,-3) -- (-1,6);
		\draw[->-=.2, ->-=.9, rounded corners] (-7,-3) -- (-7,6) -- (4,6) -- (4,-3);
		\draw[rounded corners, dashed] (4,-3) -- (4,-5) -- (20,-5) -- (20,6) -- (14,6);
		\draw[->-=.8, rounded corners] (14,6) -- (10,6) -- (10,12);
		\draw[-<-=.8, rounded corners] (14,-1) -- (6,-1) -- (6,3)-- (14,3);
		\draw[rounded corners, dashed] (14,-1) -- (18,-1) -- (18,3)-- (14,3);
		\draw[->-=.5, rounded corners] (1,-3) -- (1,6);
		\draw[rounded corners, dashed] (1,-3) -- (1,-9) -- (-14,-9) -- (-14,18) -- (10,18) -- (10,12);
		\draw[rounded corners, dashed] (-7,-3) -- (-7,-5) -- (-10,-5) -- (-10,14) -- (-1,14) -- (-1,12);
		\draw[rounded corners, dashed] (-1,-3) -- (-1,-7) -- (-12,-7) -- (-12,16) -- (1,16) -- (1,12);
		\draw[fill=white] (-2,5) rectangle (2,7);
		\draw (-2,7) -- (2,5);
		\node at (1,11) [right]{${\scriptstyle n-1}$};
		\node at (-1,11) [left]{${\scriptstyle 1}$};
		\node at (0,-2) [left]{${\scriptstyle n-1}$};	
		\node at (-7,6) [above]{${\scriptstyle 1}$};	
		\node at (6,1) [right]{${\scriptstyle k_j-1}$};
		\node at (1,-7) [right]{${\scriptstyle 1}$};
		\node at (3,6) [right]{${\scriptstyle 1}$};
		\node at (10,10) [right]{${\scriptstyle 1}$};
	\end{tikzpicture}
	\caption{A flat $\mathfrak{sl}_3$-web $\widetilde{\mathsf{H}}^{(n+1)}(j)$ obtained from $\mathsf{H}^{(n+1)}(j)$.}
	\label{connectedcomp}
\end{figure}
By \cref{claspgraphdeg},
\begin{align*}
    \mdeg{\mathsf{H}^{(n+1)}(j)}
    &\geq-\frac{1}{4}v(\widetilde{\mathsf{H}}^{(n+1)}(j))-c(\widetilde{\mathsf{H}}^{(n+1)}(j))\\
    &=-\frac{1}{4}(2n)-\left(c(\widetilde{\mathsf{G}}^{(n+1)}(j))-(n+1)\right)\\
    &=\frac{n+2}{2}-c(\widetilde{\mathsf{G}}^{(n+1)}(j))
\end{align*}
\cref{graphdeg} and $v(\mathsf{G}^{(n+1)}(j))=0$ claim $\mdeg{\widetilde{\mathsf{G}}^{(n+1)}(j)}=-c(\widetilde{\mathsf{G}}^{(n+1)}(j))$.
Moreover $\mdeg{\mathsf{G}^{(n+1)}(j)}=\mdeg{\widetilde{\mathsf{G}}^{(n+1)}(j)}$ holds by adequacy of $\mathsf{G}^{n+1}(j)$ and \cref{claspgraphdeg}.
Using \cref{lem:deg} and \cref{eg:deg}, the above facts lead to the following inequality. 
\begin{align*}
    \mdeg{(-1)^{n+1}\frac{[k_j]}{[n+k_j+2]}\widetilde{\mathsf{H}}^{(n+1)}(j)}
    &\geq\frac{n+2}{2}+\left(\frac{n+2}{2}+\mdeg{\mathsf{G}^{(n+1)}(j)}\right)\\
    &=(n+2)+\mdeg{\mathsf{G}^{(n+1)}(j)}.
\end{align*}
It also holds that $\mdeg{\mathsf{G}^{(n+1)}(j-1)}=\mdeg{\mathsf{G}^{(n+1)}(j)}$ due to \cref{graphdeg} and \cref{claspgraphdeg}.
In fact, the adequacy of these clasped $\mathfrak{sl}_3$-webs and $\widetilde{\mathsf{G}}^{(n+1)}(j-1)=\widetilde{\mathsf{G}}^{(n+1)}(j)$ imply
\begin{align*}
    \mdeg{\mathsf{G}^{(n+1)}(j-1)}
    &=\mdeg{\widetilde{\mathsf{G}}^{(n+1)}(j-1)}
    =-c(\widetilde{\mathsf{G}}^{(n+1)}(j-1))\\
    &=-c(\widetilde{\mathsf{G}}^{(n+1)}(j))
    =\mdeg{\widetilde{\mathsf{G}}^{(n+1)}(j)}
    =\mdeg{\mathsf{G}^{(n+1)}(j)}.
\end{align*}
Thus we obtain $\mathsf{G}^{(n+1)}(j)-\mathsf{G}^{(n+1)}(j-1)=(-1)^{n+1}\frac{[k_j]}{[n+k_j+2]}\widetilde{\mathsf{H}}^{(n+1)}(j)\in q^{(n+2)+\mdeg{\mathsf{G}^{(n+1)}}}\mathbb{Z}[[q^{\frac{1}{6}}]]$ where $\mdeg{\mathsf{G}^{(n+1)}}\coloneqq\mdeg{\mathsf{G}^{(n+1)}(j-1)}=\mdeg{\mathsf{G}^{(n+1)}(j)}$.

We remark that this result holds independently of the number $m-j$ of $\mathfrak{sl}_3$-clasps that the $n+1$ paralleled strands pass through, and besides, the number $k_j$ of the oppositely oriented strands adjacent to the $n+1$ paralleled strands.
We repeatedly apply \cref{unclasp} to $\mathrm{JW}_{D_1},\ldots,\mathrm{JW}_{D_{m-1}}$ and obtain
\[
	\mathsf{G}^{(n+1)}(0)\equiv_{n+1}\mathsf{G}^{(n+1)}(1)\equiv_{n+1}\cdots\equiv_{n+1}\mathsf{G}^{(n+1)}(m-1).
\]
Let $\mathsf{G}^{(n+1)}(m)$ be a $\mathfrak{sl}_3$-web removing the small circle from $\mathsf{G}^{(n+1)}(m-1)$, see below:
\[
	\mathsf{G}^{(n+1)}(m-1)=\!\!
	\mathord{\ 
		\tikz[baseline=-.6ex, scale=.1]{
			\begin{scope}[xshift=15cm]
				\draw[->-=.5, very thick, red] (-2,-15) -- (-2,0);
				\draw[-<-=.5, very thick, red] (-2,15) -- (-2,0);
				\draw[-<-=.2, rounded corners] (9,8) -- (4,8) -- (4,13) -- (9,13);
				\draw[-<-=.2, rounded corners] (9,-2) -- (4,-2) -- (4,2) -- (9,2);
				\draw[->-=.8, rounded corners] (9,-8) -- (4,-8) -- (4,-13) -- (9,-13);
				\draw[fill=white] (-7,10) rectangle (7,11);
				\draw[fill=white] (-7,-.5) rectangle (7,.5);
				\draw[fill=white] (-7,-10) rectangle (7,-11);
				\node at (-2,13) [right]{${\scriptstyle n}$};
				\node at (-2,-13) [right]{${\scriptstyle n}$};	
				\node at (7,13) [above]{${\scriptstyle k_1}$};
				\node at (7,2) [above]{${\scriptstyle k_2}$};
				\node at (7,-13) [below]{${\scriptstyle k_3}$};
			\end{scope}
            \draw[dashed, rounded corners, very thick, red] (-13,15) -- (-13,20) -- (13,20) -- (13,15);
			\draw[dashed, rounded corners, very thick, red] (-13,-15) -- (-13,-20) -- (13,-20) -- (13,-15);
			\begin{scope}[xshift=-15cm]
				\draw[->-=.2, rounded corners] (-9,-2) -- (-4,-2) -- (-4,2) -- (-9,2);
				\draw[-<-=.5, very thick, red] (2,-15) -- (2,0);
				\draw[->-=.5, very thick, red] (2,15) -- (2,0);
				\draw[-<-=.6, thick, rounded corners, red] (4,-3) -- (4,3) -- (9,3) -- (9,-3) 	-- cycle;
				\draw[fill=white] (-7,-.5) rectangle (7,.5);
				\node at (-7,2) [above]{${\scriptstyle k_m}$};
				\node at (2,13) [left]{${\scriptstyle n}$};
				\node at (2,-13) [left]{${\scriptstyle n}$};
				\node at (9,0) [right]{${\scriptstyle 1}$};
			\end{scope}
		}
	\ },
	\mathsf{G}^{(n+1)}(m)=\!\!
	\mathord{\ 
		\tikz[baseline=-.6ex, scale=.1]{
		\begin{scope}[xshift=15cm]
				\draw[->-=.5, very thick, red] (-2,-15) -- (-2,0);
				\draw[-<-=.5, very thick, red] (-2,15) -- (-2,0);
				\draw[-<-=.2, rounded corners] (9,8) -- (4,8) -- (4,13) -- (9,13);
				\draw[-<-=.2, rounded corners] (9,-2) -- (4,-2) -- (4,2) -- (9,2);
				\draw[->-=.8, rounded corners] (9,-8) -- (4,-8) -- (4,-13) -- (9,-13);
				\draw[fill=white] (-7,10) rectangle (7,11);
				\draw[fill=white] (-7,-.5) rectangle (7,.5);
				\draw[fill=white] (-7,-10) rectangle (7,-11);
				\node at (-2,13) [right]{${\scriptstyle n}$};
				\node at (-2,-13) [right]{${\scriptstyle n}$};	
				\node at (7,13) [above]{${\scriptstyle k_1}$};
				\node at (7,2) [above]{${\scriptstyle k_2}$};
				\node at (7,-13) [below]{${\scriptstyle k_3}$};
		\end{scope}
            \draw[dashed, rounded corners, very thick, red] (-13,15) -- (-13,20) -- (13,20) -- (13,15);
            \draw[dashed, rounded corners, very thick, red] (-13,-15) -- (-13,-20) -- (13,-20) -- (13,-15);
		\begin{scope}[xshift=-15cm]
				\draw[->-=.2, rounded corners] (-9,-2) -- (-4,-2) -- (-4,2) -- (-9,2);
				\draw[-<-=.5, very thick, red] (2,-15) -- (2,0);
				\draw[->-=.5, very thick, red] (2,15) -- (2,0);
				\draw[fill=white] (-7,-.5) rectangle (7,.5);
				\node at (-7,2) [above]{${\scriptstyle k_m}$};
				\node at (2,13) [left]{${\scriptstyle n}$};
				\node at (2,-13) [left]{${\scriptstyle n}$};
		\end{scope}
		}
	\ }.
\]
Then, we obtain
\[
    \mathsf{G}^{(n+1)}(m-1)=\frac{[n+2][n+k_{m}+3]}{[n+1][n+k_{m}+2]}\mathsf{G}^{(n+1)}(m).
\]
by \cref{partialtrace}.
From the above equality, it is easily seen that
\[
    \mathsf{G}^{(n+1)}(1)\equiv_{n+1}\mathsf{G}^{(n+1)}(0)
\]
holds for any $k_m$.
Next, we consider the leftmost strand of the other $n+1$ parallel circles.
One can unclasp the leftmost strand from $\mathfrak{sl}_3$-clasps exactly in the same way.
The label $k_j$ in this argument might be $n$.
However, it works independently of $k_j$ as I mentioned above.
We repeatedly apply this argument until all $n+1$ parallel circles passing through $\mathfrak{sl}_3$-clasps become $n$ parallel strands. 
Consequently, we obtain $\mathsf{G}_{n+1}(\otimes_{i=1}^{k}\ID_{D_i})\equiv_{n+1}\mathsf{G}_{n}(\otimes_{i=1}^{k}\ID_{D_i})$.
\end{proof}

\begin{thm}\label{thm:anti-thm}
	Let $L$ be an anti-parallel $B$-adequate link.
	Then,
	\[
		\hat{J}^{\mathfrak{sl}_3}_{{L},{n+1}}(q)-\hat{J}^{\mathfrak{sl}_3}_{{L},{n}}(q)\in q^{n+1}\mathbb{Z}[[q]].
	\]
	In other words, the one-row colored $\mathfrak{sl}_{3}$-Jones polynomials $\{\hat{J}^{\mathfrak{sl}_3}_{{L},{n}}(q)\}_{n}$ of $L$ is zero stable.
\end{thm}
\begin{proof}
    Let us take a link diagram $\mathsf{G}(\otimes_{i=1}^{k}R_{l_i})$ of $L$ with an adequate graph $\mathsf{G}$ and twist regions $\bm{l}\colon \{D_i\}_{i=1}^{k}\to 2\mathbb{Z}_{>0}$.
    The one-row colored $\mathfrak{sl}_3$-Jones polynomial $\hat{J}^{\mathfrak{sl}_3}_{{L},{n}}(q)$ is given by the normalization in \cref{degnormalization} of the clasped $\mathfrak{sl}_3$-web $\mathsf{G}_{n}(\otimes_{i=1}^{k}({R_{l_i}}^{(n)}))$.
    \cref{linkexpansion} and \cref{linktograph} claim that
    \begin{align*}
        \mathsf{G}_{n}(\otimes_{i=1}^{k}(R_{l_i}^{(n)}))
        &\equiv_{2n+1}\prod_{i=1}^{k}\gamma^{(n)}(0)^{l_{i}}\mathsf{G}_{n}(\otimes_{i=1}^{k}\mathrm{JW}_{D_i})\\
        \mathsf{G}_{n+1}(\otimes_{i=1}^{k}(R_{l_i}^{(n+1)}))
        &\equiv_{2n+3}\prod_{i=1}^{k}\gamma^{(n+1)}(0)^{l_{i}}\mathsf{G}_{n+1}(\otimes_{i=1}^{k}\mathrm{JW}_{D_i}),
    \end{align*}
    and \cref{graphstability} claims
    \begin{align*}
        \mathsf{G}_{n}(\otimes_{i=1}^{k}\mathrm{JW}_{D_i})\equiv_{n+1}\mathsf{G}_{n+1}(\otimes_{i=1}^{k}\mathrm{JW}_{D_i}).
    \end{align*}
    It is easy to see that $f(q)\equiv_n g(q)$ if $f(q)\equiv_N g(q)$ for some $N\geq n$, and $(-1)^{k_1}q^{k_2}f(q)\equiv_{n}(-1)^{l_1}q^{l_2}g(q)$ if $f(q)\equiv_n g(q)$ for any $k_1,k_2,l_1,l_2$.
    Hence, the above equivalence relations derive
    \begin{align*}
        \mathsf{G}_{n}(\otimes_{i=1}^{k}(R_{l_i}^{(n)}))
        &\equiv_{n+1}\prod_{i=1}^{k}\gamma^{(n)}(0)^{l_{i}}\mathsf{G}_{n}(\otimes_{i=1}^{k}\mathrm{JW}_{D_i})\\
        &\equiv_{n+1}\prod_{i=1}^{k}\gamma^{(n+1)}(0)^{l_{i}}\mathsf{G}_{n+1}(\otimes_{i=1}^{k}\mathrm{JW}_{D_i})\\
        &\equiv_{n+1}\mathsf{G}_{n+1}(\otimes_{i=1}^{k}(R_{l_i}^{(n+1)})).
    \end{align*}
    It means $\hat{J}^{\mathfrak{sl}_3}_{{L},{n+1}}(q)-\hat{J}^{\mathfrak{sl}_3}_{{L},{n}}(q)\in q^{n+1}\mathbb{Z}[[q^{\frac{1}{6}}]]$.
\end{proof}

\appendix
\section{Formulas for clasped $\mathfrak{sl}_3$-webs}\label{Appendix}
It is well-known that the closure of $\mathrm{JW}_{{-}^{m}{+}^{n}}$ is given by
\[
	\Delta(m,n)=
	\mathord{\ 
		\tikz[baseline=-.6ex, scale=.1]{
			\draw[->-=.2] (0,0) circle [radius=5];
			\draw[-<-=.2] (0,0) circle [radius=8];
			\draw[fill=white] (-10,0) rectangle (-3,1);
			\node at (0,8) [above]{${\scriptstyle m}$};
			\node at (0,5) [below]{${\scriptstyle n}$};
		}
	\ }
	=\frac{[m+1][n+1][m+n+2]}{[2]}.
\]

\begin{prop}\label{partialtrace}
	\[
		\mathord{\ 
			\tikz[baseline=-.6ex, scale=.1]{
				\draw[->-=.2, ->-=.8] (-8,-6) -- (-8,6);
				\draw[-<-=.2, -<-=.8] (-5,-6) -- (-5,6);
				\draw[->-=.2, rounded corners] (-2,0) -- (-2,3) -- (2,3) -- (2,-3) -- (-2,-3) -- 	cycle;
				\draw[fill=white] (-10,0) rectangle (0,1);
				\node at (-8,6) [above]{${\scriptstyle m}$};
				\node at (-5,6) [above]{${\scriptstyle n}$};
				\node at (0,3) [above]{${\scriptstyle l}$};
			}
		\ }
		=\frac{\Delta(m+l,n)}{\Delta(m,n)}
		\mathord{\ 
			\tikz[baseline=-.6ex, scale=.1]{
				\draw[->-=.2, ->-=.8] (-8,-6) -- (-8,6);
				\draw[-<-=.2, -<-=.8] (-5,-6) -- (-5,6);
				\draw[fill=white] (-10,0) rectangle (-3,1);
				\node at (-8,6) [above]{${\scriptstyle m}$};
				\node at (-5,6) [above]{${\scriptstyle n}$};
			}
		\ }
	\]
\end{prop}
\begin{proof}
	It is known that this clasped $\mathfrak{sl}_3$-web space is a one-dimentional and it is spanned by $\mathrm{JW}_{{-}^{m}{+}^{n}}$.
	Thus, we set
	\[
		\mathord{\ 
			\tikz[baseline=-.6ex, scale=.1]{
				\draw[->-=.2, ->-=.8] (-8,-6) -- (-8,6);
				\draw[-<-=.2, -<-=.8] (-5,-6) -- (-5,6);
				\draw[->-=.2, rounded corners] (-2,0) -- (-2,3) -- (2,3) -- (2,-3) -- (-2,-3) -- 	cycle;
				\draw[fill=white] (-10,0) rectangle (0,1);
				\node at (-8,6) [above]{${\scriptstyle m}$};
				\node at (-5,6) [above]{${\scriptstyle n}$};
				\node at (0,3) [above]{${\scriptstyle l}$};
			}
		\ }
		=C
		\mathord{\ 
			\tikz[baseline=-.6ex, scale=.1]{
				\draw[->-=.2, ->-=.8] (-8,-6) -- (-8,6);
				\draw[-<-=.2, -<-=.8] (-5,-6) -- (-5,6);
				\draw[fill=white] (-10,0) rectangle (-3,1);
				\node at (-8,6) [above]{${\scriptstyle m}$};
				\node at (-5,6) [above]{${\scriptstyle n}$};
			}
		\ }.
	\]
	The closure of diagrams in the left- and right-hand sides are given by $\Delta(m+l,n)$ and $\Delta(m,n)$, respectively.
	Hence, $C=\Delta(m+l,n)/\Delta(m,n)$. 
\end{proof}

In order to prove \cref{unclasp}, we prepare some lemmas.

\begin{lem}[The bubble skein expansion formula~\cite{Yuasa17}]\label{lem:bubble}
	\[
		\,\tikz[baseline=-.6ex, scale=1]{
		\draw (-.4,.5) -- +(-.2,0);
		\draw (.4,-.5) -- +(.2,0);
		\draw (-.4,-.5) -- +(-.2,0);
		\draw (.4,.5) -- +(.2,0);
		\draw[-<-=.5] (-.4,.5) -- (0,.5);
		\draw[-<-=.5] (0,.5) -- (.4,.5);
		\draw[->-=.5] (-.4,-.5) -- (0,-.5);
		\draw[->-=.5] (0,-.5) -- (.4,-.5);
		\draw[-<-=.5] (.05,.3) to[out=east, in=east] (.05,.-.3);
		\draw[->-=.5] (-.05,.3) to[out=west, in=west] (-.05,-.3);
		\draw[fill=white] (-.4,.3) rectangle +(-.1,.3);
		\draw[fill=white] (.4,.3) rectangle +(.1,.3);
		\draw[fill=white] (-.4,-.3) rectangle +(-.1,-.3);
		\draw[fill=white] (.4,-.3) rectangle +(.1,-.3);
		\draw[fill=white] (-.05,.2) rectangle +(.1,.4);
		\draw[fill=white] (-.05,-.2) rectangle +(.1,-.4);
		\node at (.4,-.5)[below right]{$\scriptstyle{k-b}$};
		\node at (-.4,-.5)[below left]{$\scriptstyle{k-a}$};
		\node at (.4,.5)[above right]{$\scriptstyle{l-b}$};
		\node at (-.4,.5)[above left]{$\scriptstyle{l-a}$};
		\node at (-.2,0)[left]{$\scriptstyle{a}$};
		\node at (.2,0)[right]{$\scriptstyle{b}$};
		\node at (0,-.6)[below]{$\scriptstyle{k}$};
		\node at (0,.6)[above]{$\scriptstyle{l}$};
		}\, 
		=
		\sum_{t=\max\{a, b\}}^{\min\{a+b, k, l\}}
		\frac{{k\brack t}{l\brack t}{t\brack a}{t\brack b}{k+l-t+2\brack k+l-a-b+2}}{{k\brack a}{l\brack a}{k\brack b}{l\brack b}}
		\,\tikz[baseline=-.6ex, scale=1]{
		\draw (-.4,.4) -- +(-.2,0);
		\draw (.4,-.4) -- +(.2,0);
		\draw (-.4,-.4) -- +(-.2,0);
		\draw (.4,.4) -- +(.2,0);
		\draw[-<-=.5] (-.4,.5) -- (.4,.5);
		\draw[->-=.5] (-.4,-.5)  -- (.4,-.5);
		\draw[-<-=.5] (-.4,.3) to[out=east, in=east] (-.4,.-.3);
		\draw[->-=.5] (.4,.3) to[out=west, in=west] (.4,-.3);
		\draw[fill=white] (-.4,.2) rectangle +(-.1,.4);
		\draw[fill=white] (.4,.2) rectangle +(.1,.4);
		\draw[fill=white] (-.4,-.2) rectangle +(-.1,-.4);
		\draw[fill=white] (.4,-.2) rectangle +(.1,-.4);
		\node at (.4,-.6)[right]{$\scriptstyle{k-b}$};
		\node at (-.4,-.6)[left]{$\scriptstyle{k-a}$};
		\node at (.4,.6)[right]{$\scriptstyle{l-b}$};
		\node at (-.4,.6)[left]{$\scriptstyle{l-a}$};
		\node at (-.2,0)[left]{$\scriptstyle{t-a}$};
		\node at (.2,0)[right]{$\scriptstyle{t-b}$};
		\node at (0,-.5)[below]{$\scriptstyle{k-t}$};
		\node at (0,.5)[above]{$\scriptstyle{l-t}$};
		}\, 
		\]	
\end{lem}

\begin{lem}[{\cite[Theorem~3.3]{Kim07}}]\label{lem:two-row-recursion}
	\begin{align*}
		\,\tikz[baseline=-.6ex, scale=.1]{
			\draw[->-=.6] (-3,0) -- (-3,5);
			\draw[->-=.5] (-3,-5) -- (-3,0);
			\draw[->-=.6] (0,0) -- (0,5);
			\draw[->-=.5] (0,-5) -- (0,0);
			\draw[-<-=.6] (3,0) -- (3,5);
			\draw[-<-=.5] (3,-5) -- (3,0);
			\draw[fill=white] (-4,0) rectangle (4,1);
			\node at (-3,5) [above]{${\scriptstyle 1}$};
			\node at (-3,-5) [below]{${\scriptstyle 1}$};
			\node at (0,5) [above]{${\scriptstyle k}$};
			\node at (0,-5) [below]{${\scriptstyle k}$};
			\node at (3,5) [above]{${\scriptstyle l}$};
			\node at (3,-5) [below]{${\scriptstyle l}$};
		}\,
		=
		\,\tikz[baseline=-.6ex, scale=.1]{
			\draw[->-=.5] (-3,-5) -- (-3,5);
			\draw[->-=.6] (0,0) -- (0,5);
			\draw[->-=.5] (0,-5) -- (0,0);
			\draw[-<-=.6] (3,0) -- (3,5);
			\draw[-<-=.5] (3,-5) -- (3,0);
			\draw[fill=white] (-1,0) rectangle (4,1);
			\node at (-3,5) [above]{${\scriptstyle 1}$};
			\node at (-3,-5) [below]{${\scriptstyle 1}$};
			\node at (0,5) [above]{${\scriptstyle k}$};
			\node at (0,-5) [below]{${\scriptstyle k}$};
			\node at (3,5) [above]{${\scriptstyle l}$};
			\node at (3,-5) [below]{${\scriptstyle l}$};
		}\,
		-\frac{[k]}{[k+1]}
		\,\tikz[baseline=-.6ex, scale=.1]{
			\draw[->-=.2, ->-=.9] (-3,-7) -- (-3,7);
			\draw[->-=.5] (1,-3) -- (1,3);
			\draw[-<-=.5] (9,-3) -- (9,3);
			\draw[->-=.5] (-3,2) -- (0,2) -- (0,3);
			\draw[-<-=.5] (-3,-2) -- (0,-2) -- (0,-3);
			\node at (-3,7) [above]{${\scriptstyle 1}$};
			\node at (-3,-7) [below]{${\scriptstyle 1}$};
			\node at (.5,0) [right]{${\scriptstyle k-1}$};
			\node at (10,0) [right]{${\scriptstyle l}$};
			\begin{scope}[yshift=3cm, xshift=1cm]
				\draw[->-=.6] (2,1) -- (2,4);
				\draw[-<-=.6] (8,1) -- (8,4);
				\draw[fill=white] (-2,0) rectangle (10,1);
				\node at (2,4) [above]{${\scriptstyle k}$};
				\node at (8,4) [above]{${\scriptstyle l}$};
			\end{scope}
			\begin{scope}[yshift=-3cm, xshift=1cm]
				\draw[-<-=.6] (2,-1) -- (2,-4);
				\draw[->-=.6] (8,-1) -- (8,-4);
				\draw[fill=white] (-2,0) rectangle (10,-1);
				\node at (2,-4) [below]{${\scriptstyle k}$};
				\node at (8,-4) [below]{${\scriptstyle l}$};
			\end{scope}
		}\,
		-\frac{[l]}{[k+1][k+l+2]}
		\,\tikz[baseline=-.6ex, scale=.1]{
			\draw[->-=.5] (3,-3) -- (3,3);
			\draw[-<-=.5] (9,-3) -- (9,3);
			\draw[-<-=.3, rounded corners] (-3,7) -- (-3,1) -- (0,1) -- (0,3);
			\draw[->-=.3, rounded corners] (-3,-7) -- (-3,-1) -- (0,-1) -- (0,-3);
			\node at (-3,7) [above]{${\scriptstyle 1}$};
			\node at (-3,-7) [below]{${\scriptstyle 1}$};
			\node at (3,0) [right]{${\scriptstyle k}$};
			\node at (9,0) [right]{${\scriptstyle l-1}$};
			\begin{scope}[yshift=3cm, xshift=1cm]
				\draw[->-=.6] (2,1) -- (2,4);
				\draw[-<-=.6] (7,1) -- (7,4);
				\draw[fill=white] (-2,0) rectangle (10,1);
				\node at (2,4) [above]{${\scriptstyle k}$};
				\node at (7,4) [above]{${\scriptstyle l}$};
			\end{scope}
			\begin{scope}[yshift=-3cm, xshift=1cm]
				\draw[-<-=.6] (2,-1) -- (2,-4);
				\draw[->-=.6] (7,-1) -- (7,-4);
				\draw[fill=white] (-2,0) rectangle (10,-1);
				\node at (2,-4) [below]{${\scriptstyle k}$};
				\node at (7,-4) [below]{${\scriptstyle l}$};
			\end{scope}
		}\ .
	\end{align*}
\end{lem}

\begin{lem}\label{lem:lemma1}
	\[
		\,\tikz[baseline=-.6ex, scale=.1, yshift=-5cm]{
			\draw[->-=.8] (-5,6) -- (-5,10);
			\draw[->-=.5] (-4,0) -- (-4,6);
			\draw[->-=.5] (0,0) -- (0,6);
			\draw[->-=.5] (-1,6) to[out=north, in=west] (2,9) to[out=east, in=north] (5,0);
			\draw[fill=white] (-2,0) rectangle (6,1);
			\draw[fill=white] (-6,6) rectangle (2,7);
			\node at (-5,10) [above]{\scriptsize $k$};
			\node at (-4,3) [left]{\scriptsize $1$};
			\node at (0,3) [right]{\scriptsize $k$};
			\node at (2,9) [above]{\scriptsize $1$};
		}\
		=\frac{(-1)^k}{[k+1]}
		\,\tikz[baseline=-.6ex, scale=.1, yshift=-5cm]{
			\draw[->-=.5] (4,0) -- (4,8);
			\draw[->-=.5] (-4,0) to[out=north, in=west] (-2,4) to[out=east, in=north] (0,0);
			\draw[fill=white] (-2,0) rectangle (6,1);
			\node at (4,8) [above]{\scriptsize $k$};
			\node at (-2,4) [above]{\scriptsize $1$};
		}\
	\]
\end{lem}
\begin{proof}
    Apply \cref{singleexp} to an $\mathfrak{sl}_3$-clasp in the left above. 
    Then, one can see that diagrams in the expansion vanish except for the last term due to the bottom $\mathfrak{sl}_3$-clasp.
    It becomes the right-hand side by \cref{claspformula}~(2).
\end{proof}

\begin{lem}\label{lem:lemma2}
	\begin{align*}
		\,\tikz[baseline=-.6ex, scale=.1]{
			\draw[->-=.5] (3,-3) -- (3,3);
			\draw[-<-=.5] (9,-3) -- (9,3);
			\draw[-<-=.3, rounded corners] (-3,7) -- (-3,1) -- (0,1) -- (0,3);
			\node at (-3,2) [left]{${\scriptstyle 1}$};
			\node at (3,0) [right]{${\scriptstyle k}$};
			\node at (9,0) [right]{${\scriptstyle l-1}$};
			\begin{scope}[yshift=3cm, xshift=1cm]
				\draw[->-=.6] (2,1) -- (2,4);
				\draw[-<-=.6] (7,1) -- (7,7);
				\draw[fill=white] (-2,0) rectangle (10,1);
				\node at (2,3) [right]{${\scriptstyle k}$};
				\node at (7,7) [above]{${\scriptstyle l}$};
			\end{scope}
			\begin{scope}[yshift=-3cm, xshift=1cm]
				\draw[-<-=.6] (2,-1) -- (2,-4);
				\draw[->-=.6] (8,-1) -- (8,-4);
				\draw[fill=white] (-2,0) rectangle (10,-1);
				\node at (2,-4) [below]{${\scriptstyle k}$};
				\node at (8,-4) [below]{${\scriptstyle l-1}$};
			\end{scope}
			\draw (0,7) -- (0,10);
			\draw[fill=white] (-4,7) rectangle (5,8);
			\node at (0,10) [above]{\scriptsize $k+1$};
		}\ 
		=(-1)^{k}\frac{[l+1]}{[k+l+1]}
		\,\tikz[baseline=-.6ex, scale=.1]{
			\draw[-<-=.3] (-2,7) to[out=south, in=north] (3,-3);
			\draw[->-=.5] (9,7) -- (9,-3);
			\node at (-2,2) {${\scriptstyle k}$};
			\node at (9,0) [right]{${\scriptstyle l-1}$};
			\begin{scope}[yshift=3cm, xshift=1cm]
				\draw[->-=.6] (6,4) to[out=south, in=south] (2,4);
				\node at (4,0) {${\scriptstyle 1}$};
				\node at (7,7) [above]{${\scriptstyle l}$};
			\end{scope}
			\begin{scope}[yshift=-3cm, xshift=1cm]
				\draw[-<-=.6] (2,-1) -- (2,-4);
				\draw[->-=.6] (8,-1) -- (8,-4);
				\draw[fill=white] (-2,0) rectangle (10,-1);
				\node at (2,-4) [below]{${\scriptstyle k}$};
				\node at (8,-4) [below]{${\scriptstyle l-1}$};
			\end{scope}
			\draw (0,7) -- (0,10);
			\draw (8,7) -- (8,10);
			\draw[fill=white] (-4,7) rectangle (5,8);
			\draw[fill=white] (6,7) rectangle (10,8);
			\node at (0,10) [above]{\scriptsize $k+1$};
		}\ .
	\end{align*}
\end{lem}
\begin{proof}
	It is known that the $\mathfrak{sl}_3$-web space on a disk with clasped end points $\mathrm{JW}_{{-}^l}$, $\mathrm{JW}_{{+}^{k+1}}$, and $\mathrm{JW}_{{-}^k{+}^{l-1}}$ is spanned by one clasped $\mathfrak{sl}_3$-web in the ringht-hand side.
	See, for example, \cite{Kuperberg96, Kim07} for details.
	Hence we only have to determine the coefficient $C$ such that
	\begin{align*}
		\,\tikz[baseline=-.6ex, scale=.1]{
			\draw[->-=.5] (3,-3) -- (3,3);
			\draw[-<-=.5] (9,-3) -- (9,3);
			\draw[-<-=.3, rounded corners] (-3,7) -- (-3,1) -- (0,1) -- (0,3);
			\node at (-3,2) [left]{${\scriptstyle 1}$};
			\node at (3,0) [right]{${\scriptstyle k}$};
			\node at (9,0) [right]{${\scriptstyle l-1}$};
			\begin{scope}[yshift=3cm, xshift=1cm]
				\draw[->-=.6] (2,1) -- (2,4);
				\draw[-<-=.6] (7,1) -- (7,7);
				\draw[fill=white] (-2,0) rectangle (10,1);
				\node at (2,3) [right]{${\scriptstyle k}$};
				\node at (7,7) [above]{${\scriptstyle l}$};
			\end{scope}
			\begin{scope}[yshift=-3cm, xshift=1cm]
				\draw[-<-=.6] (2,-1) -- (2,-4);
				\draw[->-=.6] (8,-1) -- (8,-4);
				\draw[fill=white] (-2,0) rectangle (10,-1);
				\node at (2,-4) [below]{${\scriptstyle k}$};
				\node at (8,-4) [below]{${\scriptstyle l-1}$};
			\end{scope}
			\draw (0,7) -- (0,10);
			\draw[fill=white] (-4,7) rectangle (5,8);
			\node at (0,10) [above]{\scriptsize $k+1$};
		}\ 
		=C
		\,\tikz[baseline=-.6ex, scale=.1]{
			\draw[-<-=.3] (-2,7) to[out=south, in=north] (3,-3);
			\draw[->-=.5] (9,7) -- (9,-3);
			\node at (-2,2) {${\scriptstyle k}$};
			\node at (9,0) [right]{${\scriptstyle l-1}$};
			\begin{scope}[yshift=3cm, xshift=1cm]
				\draw[->-=.6] (6,4) to[out=south, in=south] (2,4);
				\node at (4,0) {${\scriptstyle 1}$};
				\node at (7,7) [above]{${\scriptstyle l}$};
			\end{scope}
			\begin{scope}[yshift=-3cm, xshift=1cm]
				\draw[-<-=.6] (2,-1) -- (2,-4);
				\draw[->-=.6] (8,-1) -- (8,-4);
				\draw[fill=white] (-2,0) rectangle (10,-1);
				\node at (2,-4) [below]{${\scriptstyle k}$};
				\node at (8,-4) [below]{${\scriptstyle l-1}$};
			\end{scope}
			\draw (0,7) -- (0,10);
			\draw (8,7) -- (8,10);
			\draw[fill=white] (-4,7) rectangle (5,8);
			\draw[fill=white] (6,7) rectangle (10,8);
			\node at (0,10) [above]{\scriptsize $k+1$};
		}\ .
	\end{align*}
	Attach an $\mathfrak{sl}_3$-web
	$
	\,\tikz[baseline=-.6ex, scale=.1, yshift=-2cm]{
		\draw[->-=.8] (-6,0) -- (-6,5);
		\draw[->-=.5] (-4,0) to[out=north, in=west] (-2,4) to[out=east, in=north] (0,0);
		\draw[-<-=.8] (3,0) -- (3,5);
		\node at (-6,5) [above]{\scriptsize $k$};
		\node at (-2,4) [above]{\scriptsize $1$};
		\node at (3,5) [above]{\scriptsize $l-1$};
	}\
	$
	to the top of $\mathfrak{sl}_3$-webs in the both sides.
	The left-hand side becomes $\frac{(-1)^k\Delta(k,l)}{[k+1]\Delta(k,l-1)}\mathrm{JW}_{{-}^k{+}^{l+1}}$ by \cref{lem:lemma1} and \cref{partialtrace}.
	The right-hand side becomes $C\frac{[k+l+2]}{[k+1][l]}\mathrm{JW}_{{-}^k{+}^{l+1}}$ by \cref{lem:bubble}.
	We obtain the value $C$ in the assertion by solving this equation.
\end{proof}

\begin{prop}\label{prop:unclasp-appendix}
	\begin{align*}
		\,\tikz[baseline=-.6ex, scale=.1]{
			\draw[->-=.6] (-3,0) -- (-3,5);
			\draw[->-=.5] (-3,-5) -- (-3,0);
			\draw[->-=.6] (0,0) -- (0,5);
			\draw[->-=.5] (0,-5) -- (0,0);
			\draw[-<-=.6] (3,0) -- (3,5);
			\draw[-<-=.5] (3,-5) -- (3,0);
			\draw[fill=white] (-4,0) rectangle (4,1);
			\node at (-3,5) [above]{${\scriptstyle 1}$};
			\node at (-3,-5) [below]{${\scriptstyle 1}$};
			\node at (0,5) [above]{${\scriptstyle k}$};
			\node at (0,-5) [below]{${\scriptstyle k}$};
			\node at (3,5) [above]{${\scriptstyle l}$};
			\node at (3,-5) [below]{${\scriptstyle l}$};
		}\,
		=
		\,\tikz[baseline=-.6ex, scale=.1, yshift=-1cm]{
			\draw[->-=.5] (-3,-5) -- (-3,5);
			\draw[->-=.6] (0,0) -- (0,5);
			\draw[->-=.5] (0,-5) -- (0,0);
			\draw[-<-=.6] (3,0) -- (3,8);
			\draw[-<-=.5] (3,-5) -- (3,0);
			\draw[fill=white] (-1,0) rectangle (4,1);
			\node at (-4,-5) [below]{${\scriptstyle 1}$};
			\node at (-2,8) [above]{${\scriptstyle k+1}$};
			\node at (0,-5) [below]{${\scriptstyle k}$};
			\node at (3,8) [above]{${\scriptstyle l}$};
			\node at (3,-5) [below]{${\scriptstyle l}$};
			\draw[->-=.8] (-2,5) -- (-2,8);
			\draw[fill=white] (-5,4) rectangle (1,5);
		}\,
		+(-1)^{k+1}\frac{[l]}{[k+l+2]}
		\,\tikz[baseline=-.6ex, scale=.1]{
			\draw[-<-=.3] (-4,5) to[out=south, in=north] (3,-3);
			\draw[->-=.5] (9,5) -- (9,-3);
			\node at (2,1) {${\scriptstyle k}$};
			\node at (9,0) [right]{${\scriptstyle l-1}$};
			\begin{scope}[yshift=3cm, xshift=1cm]
				\draw[->-=.6] (6,2) to[out=south, in=south] (2,2);
				\node at (4,-2) {${\scriptstyle 1}$};
				\node at (7,5) [above]{${\scriptstyle l}$};
			\end{scope}
			\begin{scope}[yshift=-3cm, xshift=1cm]
				\draw[-<-=.6] (2,-1) -- (2,-4);
				\draw[->-=.6] (7,-1) -- (7,-4);
				\draw[->-=.2] (-4,-4) to[out=north, in=west] (-2,2) to[out=east, in=north] (0,-1);
				\draw[fill=white] (-3,0) rectangle (10,-1);
				\node at (-4,-4) [below]{${\scriptstyle 1}$};
				\node at (2,-4) [below]{${\scriptstyle k}$};
				\node at (7,-4) [below]{${\scriptstyle l}$};
			\end{scope}
			\draw (0,5) -- (0,8);
			\draw (8,5) -- (8,8);
			\draw[fill=white] (-5,5) rectangle (4,6);
			\draw[fill=white] (6,5) rectangle (10,6);
			\node at (0,8) [above]{\scriptsize $k+1$};
		}\ .
	\end{align*}
\end{prop}
\begin{proof}
	Let us denote the second coefficient in the assertion by $a_k\coloneqq (-1)^{k+1}[l]/[k+l+2]$.
	We firstly attach an $\mathfrak{sl}_3$-clasp
	$
	\ \tikz[baseline=-.6ex, scale=.1, yshift=-5cm]{
		\draw[->-=.8] (-2,5) -- (-2,8);
		\draw[fill=white] (-5,4) rectangle (1,5);
		\node at (-2,8) [above]{${\scriptstyle k+1}$};
	}\ 
	$
	to the left top of diagrams in \cref{lem:two-row-recursion}.
	Next, we calculate the second and the third terms on the right-hand side of the resulting equation.
	More precisely, we will show
	\begin{align}
		-\frac{[k]}{[k+1]}
		\tikz[baseline=-.6ex, scale=.1]{
			\draw[->-=.5] (2,-3) -- (2,3);
			\draw[-<-=.5] (9,-3) -- (9,3);
			\draw[-<-=.3] (-3,7) to[out=south, in=west] (-1,1) to[out=east, in=south] (1,3);
			\node at (-3,2) [left]{${\scriptstyle 1}$};
			\node at (1.5,1) [right]{${\scriptstyle k-1}$};
			\node at (9,0) [right]{${\scriptstyle l}$};
			\begin{scope}[yshift=3cm, xshift=1cm]
				\draw[->-=.6] (2,1) -- (2,4);
				\draw[-<-=.6] (7,1) -- (7,7);
				\draw[fill=white] (-2,0) rectangle (10,1);
				\node at (2,3) [right]{${\scriptstyle k}$};
				\node at (7,7) [above]{${\scriptstyle l}$};
			\end{scope}
			\begin{scope}[yshift=-3cm, xshift=1cm]
				\draw[-<-=.6] (2,-1) -- (2,-4);
				\draw[->-=.6] (8,-1) -- (8,-4);
				\draw[->-=.2] (-4,-4) to[out=north, in=west] (-2,2) to[out=east, in=north] (0,-1);
				\draw[fill=white] (-2,0) rectangle (10,-1);
				\node at (-4,-4) [below]{${\scriptstyle 1}$};
				\node at (2,-4) [below]{${\scriptstyle k}$};
				\node at (8,-4) [below]{${\scriptstyle l}$};
			\end{scope}
			\draw[-<-=.8] (-1,-1) -- (-1,1);
			\draw (0,7) -- (0,10);
			\draw[fill=white] (-4,7) rectangle (5,8);
			\node at (0,10) [above]{\scriptsize $k+1$};
		}\ 
		-\frac{[l]}{[k+1][k+l+2]}
		\,\tikz[baseline=-.6ex, scale=.1]{
			\draw[->-=.5] (3,-3) -- (3,3);
			\draw[-<-=.5] (9,-3) -- (9,3);
			\draw[-<-=.3, rounded corners] (-3,7) -- (-3,1) -- (0,1) -- (0,3);
			\node at (-3,2) [left]{${\scriptstyle 1}$};
			\node at (3,0) [right]{${\scriptstyle k}$};
			\node at (9,0) [right]{${\scriptstyle l-1}$};
			\begin{scope}[yshift=3cm, xshift=1cm]
				\draw[->-=.6] (2,1) -- (2,4);
				\draw[-<-=.6] (7,1) -- (7,7);
				\draw[fill=white] (-2,0) rectangle (10,1);
				\node at (2,3) [right]{${\scriptstyle k}$};
				\node at (7,7) [above]{${\scriptstyle l}$};
			\end{scope}
			\begin{scope}[yshift=-3cm, xshift=1cm]
				\draw[-<-=.6] (2,-1) -- (2,-4);
				\draw[->-=.6] (8,-1) -- (8,-4);
				\draw[->-=.2, rounded corners] (-4,-4) -- (-4,2) -- (-1,2) -- (-1,-1);
				\draw[fill=white] (-2,0) rectangle (10,-1);
				\node at (-4,-4) [below]{${\scriptstyle 1}$};
				\node at (2,-4) [below]{${\scriptstyle k}$};
				\node at (8,-4) [below]{${\scriptstyle l}$};
			\end{scope}
			\draw (0,7) -- (0,10);
			\draw[fill=white] (-4,7) rectangle (5,8);
			\node at (0,10) [above]{\scriptsize $k+1$};
		}\ 
		=a_k
		\,\tikz[baseline=-.6ex, scale=.1]{
			\draw[-<-=.3] (-4,5) to[out=south, in=north] (3,-3);
			\draw[->-=.5] (9,5) -- (9,-3);
			\node at (2,1) {${\scriptstyle k}$};
			\node at (9,0) [right]{${\scriptstyle l-1}$};
			\begin{scope}[yshift=3cm, xshift=1cm]
				\draw[->-=.6] (6,2) to[out=south, in=south] (2,2);
				\node at (4,-2) {${\scriptstyle 1}$};
				\node at (7,5) [above]{${\scriptstyle l}$};
			\end{scope}
			\begin{scope}[yshift=-3cm, xshift=1cm]
				\draw[-<-=.6] (2,-1) -- (2,-4);
				\draw[->-=.6] (7,-1) -- (7,-4);
				\draw[->-=.2] (-4,-4) to[out=north, in=west] (-2,2) to[out=east, in=north] (0,-1);
				\draw[fill=white] (-3,0) rectangle (10,-1);
				\node at (-4,-4) [below]{${\scriptstyle 1}$};
				\node at (2,-4) [below]{${\scriptstyle k}$};
				\node at (7,-4) [below]{${\scriptstyle l}$};
			\end{scope}
			\draw (0,5) -- (0,8);
			\draw (8,5) -- (8,8);
			\draw[fill=white] (-5,5) rectangle (4,6);
			\draw[fill=white] (6,5) rectangle (10,6);
			\node at (0,8) [above]{\scriptsize $k+1$};
		}\ .\label{eq:unclasp1}
	\end{align}
	We remark that the second $\mathfrak{sl}_3$-web on the left-hand side is already done in \cref{lem:lemma2}, and it provides a coefficient $(-1)^k[l+1]/[k+l+1]$.
	\cref{eq:unclasp1} holds if the first $\mathfrak{sl}_3$-web provides a coefficient $a_{k-1}$ because the summation of these coefficients is calculated as follows.
	\begin{align*}
		&-\frac{[k]}{[k+1]}a_{k-1}+\left(-\frac{[l]}{[k+1][k+l+2]}\right)\left((-1)^k\frac{[l+1]}{[k+l+1]}\right)\\
		&\quad =(-1)^{k+1}\frac{[l]}{[k+1][k+l+1][k+l+2]}([k][k+l+2]+[l+1])\\
		&\quad =(-1)^{k+1}\frac{[l]}{[k+l+2]}=a_k.
	\end{align*}
	We used $[k+1][k+l+1]-[k][k+l+2]=[l+1]$ in the equation above.
	Hence, let us prove
	\begin{align}
		\tikz[baseline=-.6ex, scale=.1]{
			\draw[->-=.5] (2,-3) -- (2,3);
			\draw[-<-=.5] (9,-3) -- (9,3);
			\draw[-<-=.3] (-3,7) to[out=south, in=west] (-1,1) to[out=east, in=south] (1,3);
			\node at (-3,2) [left]{${\scriptstyle 1}$};
			\node at (1.5,1) [right]{${\scriptstyle k-1}$};
			\node at (9,0) [right]{${\scriptstyle l}$};
			\begin{scope}[yshift=3cm, xshift=1cm]
				\draw[->-=.6] (2,1) -- (2,4);
				\draw[-<-=.6] (7,1) -- (7,7);
				\draw[fill=white] (-2,0) rectangle (10,1);
				\node at (2,3) [right]{${\scriptstyle k}$};
				\node at (7,7) [above]{${\scriptstyle l}$};
			\end{scope}
			\begin{scope}[yshift=-3cm, xshift=1cm]
				\draw[-<-=.6] (2,-1) -- (2,-4);
				\draw[->-=.6] (8,-1) -- (8,-4);
				\draw[->-=.2] (-4,-4) to[out=north, in=west] (-2,2) to[out=east, in=north] (0,-1);
				\draw[fill=white] (-2,0) rectangle (10,-1);
				\node at (-4,-4) [below]{${\scriptstyle 1}$};
				\node at (2,-4) [below]{${\scriptstyle k}$};
				\node at (8,-4) [below]{${\scriptstyle l}$};
			\end{scope}
			\draw[-<-=.8] (-1,-1) -- (-1,1);
			\draw (0,7) -- (0,10);
			\draw[fill=white] (-4,7) rectangle (5,8);
			\node at (0,10) [above]{\scriptsize $k+1$};
		}\ 
		=a_{k-1}
		\,\tikz[baseline=-.6ex, scale=.1]{
			\draw[-<-=.3] (-4,5) to[out=south, in=north] (3,-3);
			\draw[->-=.5] (9,5) -- (9,-3);
			\node at (2,1) {${\scriptstyle k}$};
			\node at (9,0) [right]{${\scriptstyle l-1}$};
			\begin{scope}[yshift=3cm, xshift=1cm]
				\draw[->-=.6] (6,2) to[out=south, in=south] (2,2);
				\node at (4,-2) {${\scriptstyle 1}$};
				\node at (7,5) [above]{${\scriptstyle l}$};
			\end{scope}
			\begin{scope}[yshift=-3cm, xshift=1cm]
				\draw[-<-=.6] (2,-1) -- (2,-4);
				\draw[->-=.6] (7,-1) -- (7,-4);
				\draw[->-=.2] (-4,-4) to[out=north, in=west] (-2,2) to[out=east, in=north] (0,-1);
				\draw[fill=white] (-3,0) rectangle (10,-1);
				\node at (-4,-4) [below]{${\scriptstyle 1}$};
				\node at (2,-4) [below]{${\scriptstyle k}$};
				\node at (7,-4) [below]{${\scriptstyle l}$};
			\end{scope}
			\draw (0,5) -- (0,8);
			\draw (8,5) -- (8,8);
			\draw[fill=white] (-5,5) rectangle (4,6);
			\draw[fill=white] (6,5) rectangle (10,6);
			\node at (0,8) [above]{\scriptsize $k+1$};
		}\ \label{eq:unclasp2}
		\end{align}
	by induction on $k$.
	One can easily prove \cref{eq:clasped-graph-deg2} at $k=1$ by expanding the middle $\mathfrak{sl}_3$-clasp.
	We assume that \cref{eq:unclasp2} holds when $k=n$.
	It means that \cref{prop:unclasp-appendix} also holds when $k=n$ by the above argument.
	Thus, we apply \cref{prop:unclasp-appendix} to the middle $\mathfrak{sl}_3$-clasp on the left-hand side of \cref{eq:unclasp2} at $k=n+1$, and one can confirm that it concludes the right-hand side of \cref{eq:unclasp2}.
\end{proof}

\bibliographystyle{amsalpha}

\begin{thebibliography}{GMV13}

  \bibitem[AD11]{ArmondDasbach11A}
  {Cody Armond and Oliver~T. Dasbach}, \href{https://arxiv.org/abs/1106.3948}{\color{black}{\em Rogers-Ramanujan type identities and the head and tail of the colored Jones polynomial}}, {arXiv:1106.3948} {(2011)}.
  $\uparrow$
  
  \bibitem[AD17]{ArmondDasbach17}
  {\bysame}, \href{https://doi.org/10.1090/proc/13211}{\color{black}{\em The head and tail of the colored Jones polynomial for adequate knots}}, {Proc. Amer. Math. Soc.} {\textbf{145}} {(2017)}, {no.~3}, {1357--1367}.
  $\uparrow$

  \bibitem[Arm13]{Armond13}
  {Cody Armond}, \href{https://doi.org/10.2140/agt.2013.13.2809}{\color{black}{\em The head and tail conjecture for alternating knots}}, {Algebr. Geom. Topol.} {\textbf{13}} {(2013)}, {no.~5}, {2809--2826}. 
  $\uparrow$

  \bibitem[BO17]{BeirneOsburn17}
  {Paul Beirne and Robert Osburn}, \href{https://doi.org/10.1016/j.indag.2016.11.016}{\color{black}{\em $q$-series and tails of colored Jones polynomials}}, {Indag. Math. (N.S.)} {\textbf{28}} {(2017)}, {no.~1}, {247--260}.
  $\uparrow$

  \bibitem[DL06]{DasbachLin06}
  {Oliver~T. Dasbach and Xiao-Song Lin}, \href{https://doi.org/10.1112/S0010437X06002296}{\color{black}{\em On the head and the tail of the colored Jones polynomial}}, {Compos. Math.} {\textbf{142}} {(2006)}, {no.~5}, {1332--1342}. 
  $\uparrow$

  \bibitem[DL07]{DasbachLin07}
  {\bysame}, \href{https://doi.org/10.2140/pjm.2007.231.279}{\color{black}{\em A volumish theorem for the Jones polynomial of alternating knots}}, {Pacific J. Math.} {\textbf{231}} {(2007)}, {no.~2}, {279--291}. 
  $\uparrow$

  \bibitem[EH17]{ElhamdadiHajij17}
  {Mohamed Elhamdadi and Mustafa Hajij}, \href{http://projecteuclid.org/euclid.ojm/1496282430}{\color{black}{\em Pretzel knots and $q$-series}}, {Osaka J. Math.} \textbf{54} {(2017)}, {no.~2}, {363--381}. 
  $\uparrow$

  \bibitem[FS22]{FrohmanSikora22}
  {Charles Frohman and Adam S. Sikora}, \href{https://doi.org/10.1007/s00209-021-02765-z}{\color{black}{\em $SU(3)$-skein algebras and webs on surfaces}}, {Math. Z.} {\textbf{300}} {(2022)}, {no.~1}, {33--56}. 
  $\uparrow$

  \bibitem[GL15]{GaroufalidisLe15}
  {Stavros Garoufalidis and Thang T.~Q. L{\^e}}, \href{https://doi.org/10.1186/2197-9847-2-1}{\color{black}{\em Nahm sums, stability and the colored Jones polynomial}}, {Res. Math. Sci.} {\textbf{2}} {(2015)}, {Art. 1}, {55}.
  $\uparrow$

  \bibitem[GMV13]{GaroufalidisMortonVuong13}
  {Stavros Garoufalidis, Hugh Morton, and Thao Vuong}, \href{https://doi.org/10.1090/S0002-9939-2013-11582-0}{\color{black}{\em The $\mathrm{SL}_3$ colored Jones polynomial of the trefoil}}, {Proc. Amer. Math. Soc.} {\textbf{141}} {(2013)}, {no.~6}, {2209--2220}. 
  $\uparrow$

  \bibitem[GNV16]{GaroufalidisNorinVuong16}
  {Stavros Garoufalidis, Sergey Norin, and Thao Vuong}, \href{https://doi.org/10.1016/j.ejc.2015.05.001}{\color{black}{\em Flag algebras and the stable coefficients of the Jones polynomial}}, {European J. Combin.} {\textbf{51}} {(2016)}, {165--189}. 
  $\uparrow$

  \bibitem[GV17]{GaroufalidisVuong17}
  {Stavros Garoufalidis and Thao Vuong}, \href{}{\em A stability conjecture for the colored Jones polynomial}, {Topology Proc.} {\textbf{49}} {(2017)}, {211--249}.
  $\uparrow$

  \bibitem[Haj16]{Hajij16}
  {Mustafa Hajij}, \href{https://doi.org/10.1007/s11139-015-9705-9}{\color{black}{\em The tail of a quantum spin network}}, {Ramanujan J.} {\textbf{40}} {(2016)}, {no.~1}, {135--176}. 
  $\uparrow$
  
  \bibitem[Kaw22]{Kawasoe22}
  {Kotaro Kawasoe}, \href{https://doi.org/10.1142/S021821652250105X}{\color{black}{\em The one-row-colored $\mathfrak{sl}_3$ Jones polynomials for pretzel links}}, {J. Knot Theory Ramifications}, \textbf{} {(2022)}, {2250105, 44}. 
  $\uparrow$

  \bibitem[Kim06]{Kim06}
  {Dongseok Kim}, \href{https://doi.org/10.1142/S0218216506004579}{\color{black}{\em Trihedron coefficients for $\mathcal{U}_q(\mathfrak{sl}(3,\mathbb{C}))$}}, {J. Knot Theory Ramifications}, {\textbf{15}} {(2006)}, {no.~4}, {453--469}. 
  $\uparrow$

  \bibitem[Kim07]{Kim07}
  {\bysame}, \href{http://projecteuclid.org/euclid.ojm/1189717429}{\color{black}{\em Jones-Wenzl idempotents for rank $2$ simple Lie algebras}}, {Osaka J. Math.} {\textbf{44}} {(2007)}, {no.~3}, {691--722}. 
  $\uparrow$

  \bibitem[KO16]{KeilthyOsburn16}
  {Adam Keilthy and Robert Osburn}, \href{https://doi.org/10.1016/j.jnt.2015.02.002}{\color{black}{\em Rogers-Ramanujan type identities for alternating knots}}, {J. Number Theory}, {\textbf{161}} {(2016)}, {255--280}.
  $\uparrow$

  \bibitem[Kup96]{Kuperberg96}
  {Greg Kuperberg}, \href{http://projecteuclid.org/euclid.cmp/1104287237}{\color{black}{\em Spiders for rank $2$ Lie algebras}}, {Comm. Math. Phys.} {\textbf{180}} {(1996)}, {no.~1}, {109--151}. 
  $\uparrow$

  \bibitem[Law03]{Lawrence03}
  {Ruth Lawrence}, \href{https://doi.org/10.1016/S0166-8641(02)00057-3}{\color{black}{\em The $\rm PSU(3)$ invariant of the Poincar\'e homology sphere}}, {Topology Appl.}, {\textbf{127}} (2003), {no.~1-2}, {153--168}. 
  $\uparrow$

  \bibitem[L\^e00]{Le00}
  {Thang T.~Q. L\^e}, \href{https://doi.org/10.1215/S0012-7094-00-10224-4}{\color{black}{\em Integrality and symmetry of quantum link invariants}}, {Duke Math. J.} \textbf{102} {(2000)}, {no.~2}, {273--306}. 
  $\uparrow$

  \bibitem[OY97]{OhtsukiYamada97}
  {Tomotada Ohtsuki and Shuji Yamada}, \href{https://doi.org/10.1142/S021821659700025X}{\color{black}{\em Quantum $\mathrm{SU}(3)$ invariant of $3$-manifolds via linear skein theory}}, {J. Knot Theory Ramifications}, \textbf{6} (1997), {no.~3}, {373--404}. 
  $\uparrow$

  \bibitem[RJ93]{RossoJones93}
  {Marc Rosso and Vaughan Jones}, \href{https://doi.org/10.1142/S0218216593000064}{\color{black}{\em On the invariants of torus knots derived from quantum groups}}, {J. Knot Theory Ramifications}, {\textbf{2}} {(1993)}, {no.~1}, {97--112}. 
  $\uparrow$

  \bibitem[SW07]{SikoraWestbury}
  {Adam S. Sikora and Bruce W. Westbury}, \href{https://doi.org/10.2140/agt.2007.7.439}{\color{black}{\em Confluence theory for graphs}}, {Algebr. Geom. Topol.} {\textbf{7}} {(2007)}, {439--478}. 
  $\uparrow$
  
  \bibitem[Yua17]{Yuasa17}
  {Wataru Yuasa}, \href{https://doi.org/10.1142/S0218216517500389}{\color{black}{\em The $\mathfrak{sl}_3$ colored Jones polynomials for $2$-bridge links}}, {J. Knot Theory Ramifications}, \textbf{26} {(2017)}, {no.~7}, {1750038, 37}. 
  $\uparrow$

  \bibitem[Yua18]{Yuasa18}
  {\bysame}, \href{https://doi.org/10.1090/proc/13907}{\color{black}{\em A $q$-series identity via the $\mathfrak{sl}_3$ colored Jones polynomials for the $(2,2m)$-torus link}}, {Proc. Amer. Math. Soc.} \textbf{146} (2018), {no.~7}, {3153--3166}. 
  $\uparrow$

  \bibitem[Yua21]{Yuasa20A}
  {\bysame}, \href{https://doi.org/10.1007/s40306-020-00397-9}{\color{black}{\em Twist formulas for one-row colored $A_{2}$ webs and $\mathfrak{sl}_3$~tails of $(2,2m)$-torus links}}, {Acta Math. Vietnam.} \textbf{46} (2021), {no.~2}, {369--387}.
  $\uparrow$

  \end{thebibliography}
\def\cprime{$'$}
\providecommand{\bysame}{\leavevmode\hbox to3em{\hrulefill}\thinspace}

\end{document}